\documentclass[11pt]{amsart}

\usepackage{amssymb, amscd}
\usepackage{tikz-cd}
\usepackage{epsfig, mathtools}
\usepackage[vmargin=1in, hmargin=1.25in]{geometry}
\usepackage[font=small,format=plain,labelfont=bf,up,textfont=it,up]{caption}
\usepackage{microtype}
\usepackage[shortlabels]{enumitem}
\usepackage[backref=page, bookmarks, bookmarksdepth=2, colorlinks=true, linkcolor=blue, citecolor=blue, urlcolor=blue]{hyperref}

\setenumerate[0]{label=(\alph*)}

\renewcommand*{\backref}[1]{}
\renewcommand*{\backrefalt}[4]{%
    \ifcase #1 (Not cited.)%
    \or        (Cited on page~#2.)%
    \else      (Cited on pages~#2.)%
    \fi}

\newcommand{\arxiv}[1]{\href{http://arxiv.org/abs/#1}{{\tt arXiv:#1}}}

\numberwithin{equation}{section}

\theoremstyle{plain}
\newtheorem{theorem}{Theorem}[section]
\newtheorem{maintheorem}{Theorem}

\newtheorem{maintheoremprime}{Theorem}

\newtheorem{lemma}[theorem]{Lemma}

\newtheorem*{claim}{Claim}

\newtheorem{stepsa}{Step}
\newtheorem{stepsb}{Step}
\newtheorem{stepsc}{Step}

\theoremstyle{definition}
\newtheorem{asm}[theorem]{Assumption}

\newtheorem{defn}[theorem]{Definition}
\newenvironment{definition}[1][]{\begin{defn}[#1]\pushQED{\qed}}{\popQED \end{defn}}
\newtheorem{notn}[theorem]{Notation}

\theoremstyle{remark}
\newtheorem{rmk}[theorem]{Remark}
\newenvironment{remark}[1][]{\begin{rmk}[#1] \pushQED{\qed}}{\popQED \end{rmk}}
\newtheorem{eg}[theorem]{Example}
\newenvironment{example}[1][]{\begin{eg}[#1] \pushQED{\qed}}{\popQED \end{eg}}

\DeclareMathOperator{\Hom}{Hom}
\DeclareMathOperator{\End}{End}
\DeclareMathOperator{\Mor}{Mor}

\DeclareMathOperator{\coker}{coker}

\DeclareMathOperator{\Image}{Im}

\DeclareMathOperator{\GL}{GL}
\DeclareMathOperator{\EL}{EL}
\DeclareMathOperator{\oEL}{\overline{EL}}

\DeclareMathOperator{\SL}{SL}

\newcommand\R{\ensuremath{\mathbb{R}}}
\newcommand\C{\ensuremath{\mathbb{C}}}
\newcommand\Z{\ensuremath{\mathbb{Z}}}

\newcommand\Field{\ensuremath{\mathbb{F}}}

\DeclareMathOperator{\HH}{H}
\newcommand\RH{\ensuremath{\widetilde{\HH}}}
\DeclareMathOperator{\CC}{C}
\DeclareMathOperator{\KK}{K}
\newcommand\RC{\ensuremath{\widetilde{\CC}}}

\DeclareMathOperator{\rk}{rk}

\DeclareMathOperator{\Aut}{Aut}

\DeclareMathOperator{\Ind}{Ind}

\newcommand\Set[2]{\ensuremath{\left\{\text{#1 $|$ #2}\right\}}}

\newcommand\Figure[4]{
\begin{figure}[t]
\centering
\centerline{\psfig{file=#2,scale=#4}}
\caption{#3}
\label{#1}
\end{figure}}

\newcommand\cD{\ensuremath{\mathcal{D}}}

\newcommand\cF{\ensuremath{\mathcal{F}}}
\newcommand\cG{\ensuremath{\mathcal{G}}}
\newcommand\cH{\ensuremath{\mathcal{H}}}

\newcommand\cK{\ensuremath{\mathcal{K}}}
\newcommand\cL{\ensuremath{\mathcal{L}}}
\newcommand\cM{\ensuremath{\mathcal{M}}}

\newcommand\cO{\ensuremath{\mathcal{O}}}

\newcommand\fF{\ensuremath{\mathfrak{F}}}

\newcommand\fS{\ensuremath{\mathfrak{S}}}

\newcommand\fd{\ensuremath{\mathfrak{d}}}

\newcommand\fq{\ensuremath{\mathfrak{q}}}

\newcommand\bbX{\ensuremath{\mathbb{X}}}
\newcommand\bbY{\ensuremath{\mathbb{Y}}}

\newcommand\bbk{\ensuremath{\Bbbk}}

\newcommand\tsigma{\ensuremath{\widetilde{\sigma}}}
\newcommand\ttau{\ensuremath{\widetilde{\tau}}}

\newcommand\tast{\ensuremath{\widetilde{\ast}}}

\newcommand\on{\ensuremath{\overline{n}}}

\newcommand\ov{\ensuremath{\overline{v}}}

\newcommand\osigma{\ensuremath{\overline{\sigma}}}

\newcommand\uV{\ensuremath{\underline{V}}}

\newcommand\opi{\ensuremath{\overline{\pi}}}

\newcommand\ocF{\ensuremath{\overline{\cF}}}

\newcommand\ocK{\ensuremath{\overline{\cK}}}

\DeclareMathOperator{\FI}{\tt FI}

\DeclareMathOperator{\VIC}{\tt VIC}
\newcommand\ubbk{\ensuremath{\underline{\bbk}}}

\DeclareMathOperator{\Sim}{Sim}
\DeclareMathOperator{\OSim}{{\mathbb S}im}
\DeclareMathOperator{\SR}{SR}
\DeclareMathOperator{\sr}{sr}
\DeclareMathOperator{\ord}{ord}
\DeclareMathOperator{\Link}{Link}
\DeclareMathOperator{\FLink}{\overrightarrow{Link}}

\DeclareMathOperator{\cSet}{\tt Set}

\DeclareMathOperator{\Simp}{\tt Simp}
\DeclareMathOperator{\Bases}{\mathcal B}
\DeclareMathOperator{\OBases}{\mathbb B}

\DeclareMathOperator{\tSimp}{\widetilde{\Simp}}

\title[Twisted homological stability and congruence subgroups]{A new approach to twisted homological stability, with applications to congruence subgroups}

\author{Andrew Putman}
\address{Dept of Mathematics; University of Notre Dame; 255 Hurley Hall; Notre Dame, IN 46556}
\email{andyp@nd.edu}
\thanks{AP was supported in part by NSF grant DMS-1811210.}

\date{December 5, 2022}

\begin{document}

\newpage

\begin{abstract}
We introduce a new method for proving twisted homological stability, and use it to prove
such results for symmetric groups and general linear groups.  In addition to sometimes slightly improving
the stable range given by the traditional method (due to Dwyer), it is easier to adapt to nonstandard
situations.  As an illustration of this, we generalize to $\GL_n$ of many rings $R$
a theorem of Borel which says that passing from $\GL_n$ of a number
ring to a finite-index subgroup does not change the rational cohomology.  Charney proved this generalization
for trivial coefficients, and we extend it to twisted coefficients.
\end{abstract}

\maketitle

\section{Introduction}
\label{section:introduction}

A sequence of groups $\{G_n\}_{n=1}^{\infty}$ exhibits {\em homological stability}
if for each $k \geq 0$, the
value $\HH_k(G_n)$ is independent of $n$ for $n \gg k$.  There is a vast
literature on this, starting with early work of Nakaoka on symmetric groups \cite{Nakaoka} and unpublished work of Quillen dealing with
$\GL_n(\Field_q)$.  Many sequences of groups exhibit
homological stability: general linear groups over many rings \cite{VanDerKallen}, mapping
class groups \cite{HarerStability}, automorphism groups of free groups \cite{HatcherVogtmannCerf}, etc.  See \cite{RandalWilliamsWahl} for a general
framework that encompasses many of these results, as well as a survey of the literature.

\subsection{Twisted coefficients}

Dwyer \cite{DwyerTwisted} showed how to extend this to homology with certain
kinds of twisted coefficients.  For instance, it follows from his work that
if $\bbk$ is a field, then for all $m \geq 1$ and $k \geq 0$ we have
\[\HH_k(\GL_{n}(\bbk);(\bbk^{n})^{\otimes m}) \cong \HH_k(\GL_{n+1}(\bbk);(\bbk^{n+1})^{\otimes m}) \quad \text{for $n \gg 0$}.\]
This has been generalized to many different groups and coefficient systems.  The
most general result we are aware of is in \cite{RandalWilliamsWahl}, which
shows how to prove such a result for ``polynomial coefficient systems'' on
many classes of groups.

\subsection{New approach}

In this paper, we give an alternate approach to twisted homological stability
for polynomial coefficient systems.  Our new approach is actually closer
to the usual proof for constant coefficients, which uses spectral sequences
associated to group actions on carefully chosen simplicial complexes.  What
we do is encode our twisted coefficient systems using certain kinds
of ``coefficient systems'' on those simplicial complexes.
The high connectivity of the simplicial complexes that is
used in the constant coefficient proof is replaced by the vanishing of the homology
of these simplicial complexes with respect to our ``coefficient systems''.

A small advantage of our approach is that it sometimes gives better stability
ranges.  More importantly, our approach is more flexible than Dwyer's, which
only works well in settings where the traditional proof of homological
stability applies in its simplest form.  
For constant coefficients, the
standard proof of homological stability can give useful information even
in settings where it does not apply directly.  Examples of this from
the author's work include the proofs of \cite[Theorem 1.1]{AshPutmanSam}
and \cite[Theorem A]{PutmanStudenmund}.  Our approach to twisted coefficients
is similarly flexible.  We will illustrate this 
by proving Theorem \ref{maintheorem:congruence} below, which shows
how the homology of the general linear group changes when you pass
to a finite-index congruence subgroup.  With twisted coefficients, this
seems hard to prove using the traditional approach.

There is a tension between on the one hand developing abstract frameworks like the one in
\cite{RandalWilliamsWahl} that systematize the homological stability machine, and on the
other hand giving flexible tools that can be adapted to nonstandard situations,
including ones that are not necessarily about homological stability per se.
Both of these goals are important.  In writing this paper, we had the latter goal in mind, so our focus
will be on basic examples rather than a general abstract result.

\begin{remark}
We wrote this paper in part to provide tools for our paper
\cite{PutmanStableLevel}, which applies our machine to study
the homology of the moduli space of curves with level structures.
The main result of \cite{PutmanStableLevel} is about untwisted homology,
but the proof requires working with twisted homology as well.
\end{remark}

\begin{remark}
\label{remark:mpp}
Miller--Patzt--Petersen \cite{MillerPatztPetersen} have recently developed
an approach to twisted homological stability that shares some ideas with our paper,
though the technical details and intended applications are different.
Unlike us, they prove a general theorem in the spirit of \cite{RandalWilliamsWahl}.
Our work was done independently of theirs.\footnote{See \cite[Remark 1.2]{MillerPatztPetersen}.  We
apologize for taking so long to write our approach up.}
\end{remark}

\subsection{FI-modules}
\label{section:fimodules}

The easiest example to which our results apply is the symmetric group.
For these groups we will encode our twisted coefficient systems using Church--Ellenberg--Farb's
theory of $\FI$-modules \cite{ChurchEllenbergFarbFI}.  Let $\FI$ be the category whose
objects are finite sets and whose morphisms are injections.  For a commutative ring $\bbk$, an {\em $\FI$-module}
over $\bbk$ is a functor $M$ from $\FI$ to the category of $\bbk$-modules.  For $n \geq 0$, 
let\footnote{It is more common to write $[n]$ for $\{1,\ldots,n\}$, but later on when discussing
semisimplicial sets we will need to use the notation $[n]$ for $\{0,\ldots,n\}$.}
$\overline{n} = \{1,\ldots,n\}$.  Every object of $\FI$ is isomorphic to $\overline{n}$ for some $n \geq 0$, so
the data of an $\FI$-module $M$ consists of:
\begin{itemize}
\item a $\bbk$-module $M(\overline{n})$ for each $n \geq 0$, and
\item for each injection $f\colon \overline{n} \rightarrow \overline{m}$, an induced $\bbk$-module homomorphism
$f_{\ast}\colon M(\overline{n}) \rightarrow M(\overline{m})$.
\end{itemize}
In particular, the inclusions $\overline{n} \rightarrow \overline{n+1}$ induce a sequence of morphisms
\begin{equation}
\label{eqn:increasingfi}
M(\overline{0}) \rightarrow M(\overline{1}) \rightarrow M(\overline{2}) \rightarrow \cdots.
\end{equation}
The group of $\FI$-automorphisms of $\overline{n}$ is the symmetric group $\fS_n$.  This acts
on $M(\overline{n})$, making $M(\overline{n})$ into a $\bbk[\fS_n]$-module.  More generally,
for a finite set $S$ the group $\fS_S$ of bijections of $S$ acts on $M(S)$, making $M(S)$
into a $\bbk[\fS_S]$-module.

\begin{example}
\label{example:easyfi}
The group $\fS_n$ acts on $\bbk^n$ for each $n \geq 0$.  We can fit the increasing sequence
\[\bbk^0 \rightarrow \bbk^1 \rightarrow \bbk^2 \rightarrow \cdots\]
of $\fS_n$-representations into an $\FI$-module $M$ over $\bbk$ by defining
\[M(S) = \bbk^S \quad \text{for a finite set $S$}.\]
Here the notation $\bbk^S$ means the free $\bbk$-module with basis $S$, and an
injective map $f\colon S \rightarrow T$ between finite sets induces a map
$f_{\ast}\colon M(S) \rightarrow M(T)$ taking basis elements to basis elements.  As
a $\bbk[\fS_n]$-module, we have $M(\overline{n}) = \bbk^{\overline{n}} \cong \bbk^n$.
\end{example}

\subsection{Polynomial FI-modules}

For an $\FI$-module $M$ over $\bbk$, the inclusions \eqref{eqn:increasingfi} induce maps between homology groups
for each $k$:
\[\HH_k(\fS_0;M(\overline{0})) \rightarrow \HH_k(\fS_1;M(\overline{1})) \rightarrow \HH_k(\fS_2;M(\overline{2})) \rightarrow \cdots.\]
We would like to give conditions under which these stabilize.  We start with
the following.  Fix a functorial coproduct $\sqcup$ on $\FI$, which
we think of as ``disjoint union''.  Letting $\ast$ be a formal symbol,
for a finite set $S$ we then have the finite set $S \sqcup \{\ast\}$
of cardinality $|S|+1$.

\begin{definition}
\label{definition:derivedfi}
Let $\bbk$ be a commutative ring and let $M$ be an $\FI$-module over $\bbk$.  
\begin{itemize}
\item The {\em shifted} $\FI$-module of $M$, denoted $\Sigma M$, is the $\FI$-module over $\bbk$ defined via the formula
$\Sigma M(S) = M(S \sqcup \{\ast\})$ for a finite set $S$.
\item The {\em derived} $\FI$-module
of $M$, denoted $DM$, is the $\FI$-module over $\bbk$ defined via the formula
\[DM(S) = \frac{M(S \sqcup \{\ast\})}{\Image(M(S) \rightarrow M(S \sqcup \{\ast\}))} \quad \text{for a finite set $S$}.\qedhere\]
\end{itemize}
\end{definition}

\begin{remark}
Morphisms between $\FI$-modules over $\bbk$ are natural transformations between functors.  This makes the collection
of $\FI$-modules over $\bbk$ into an abelian category, where kernels and cokernel are computed pointwise.  With
these conventions, there is a morphism $M \rightarrow \Sigma M$, and $DM = \coker(M \rightarrow \Sigma M)$.
\end{remark}

The idea behind
the following definition goes back to Dwyer's work \cite{DwyerTwisted}, and has
been elaborated upon by many people; see \cite{RandalWilliamsWahl} for a
more complete history.  

\begin{definition}
\label{definition:polyfi}
Let $\bbk$ be a commutative ring and let $M$ be an $\FI$-module over $\bbk$.  We say that
$M$ is {\em polynomial} of degree $d \geq -1$ starting at $m \in \Z$ if it satisfies
the following inductive condition:
\begin{itemize}
\item If $d = -1$, then for finite sets $S$ with $|S| \geq m$ we require $M(S) = 0$.
\item If $d \geq 0$, then we require the following two conditions:
\begin{itemize}
\item For all injective maps $f\colon S \rightarrow T$ between finite sets with $|S| \geq m$,
the induced map $f_{\ast}\colon M(S) \rightarrow M(T)$ must be an injection.
\item The derived $\FI$-module $DM$ must be polynomial of degree $(d-1)$ starting at $(m-1)$.\qedhere
\end{itemize}
\end{itemize}
\end{definition}

\begin{remark}
\label{remark:exactsequence}
If the $\FI$-module $M$ over $\bbk$ is polynomial of degree $d$ starting at $m \in \Z$, then for all finite
sets $S$ with $|S| \geq m$ we have a short exact sequence
\begin{equation}
\label{eqn:polyfiseq}
0 \longrightarrow M(S) \longrightarrow \Sigma M(S) \longrightarrow DM(S) \longrightarrow 0
\end{equation}
of $\bbk[\fS_S]$-modules.  
\end{remark}

\begin{remark}
\label{remark:split}
There is also a notion of a polynomial $\FI$-module being {\em split}, which can slightly
improve the bounds on where stability begins.  See \cite{RandalWilliamsWahl} for the definition.
To avoid complicating our arguments, we will not incorporate this into our results.
\end{remark}

\begin{example}
\label{example:deg0fi}
An $\FI$-module $M$ over $\bbk$ is polynomial of degree $0$ starting at $m \in \Z$ if
it satisfies the following two conditions:
\begin{itemize}
\item For all injective maps $f\colon S \rightarrow T$ between finite sets with $|S| \geq m$, 
the induced map $f_{\ast}\colon M(S) \rightarrow M(T)$ must be an injection.
\item For all finite sets $S$ of with $|S| \geq m-1$, we must have
\[DM(S) = \frac{M(S \sqcup \{\ast\})}{\Image(M(S) \rightarrow M(S \sqcup \{\ast\}))} = 0.\]
In other words, the map $M(S) \rightarrow M(S \sqcup \{\ast\})$ must be surjective.  This
implies more generally that for all injective maps $f\colon S \rightarrow T$ between finite sets with
$|S| \geq m-1$, the induced map $f_{\ast}\colon M(S) \rightarrow M(T)$ must be a surjection.
\end{itemize}
Combining these two facts, we see that $M$ is polynomial of degree $0$ starting at $m \in \Z$ if
for all injective maps $f\colon S \rightarrow T$ between finite sets, the map $f_{\ast}\colon M(S) \rightarrow M(T)$
is an isomorphism if $|S| \geq m$ and a surjection if $|S| = m-1$. 
\end{example}

\begin{example}
The $\FI$-module $M$ in Example \ref{example:easyfi} is polynomial of degree $1$ starting
at $0$.  More generally, for $d \geq 1$ we can define an $\FI$-module $M^{\otimes d}$ via the formula
\[M^{\otimes d}(S) = \left(\bbk^S\right)^{\otimes d} \quad \text{for a finite set $S$}.\]
The $\FI$-module $M^{\otimes d}$ is polynomial of degree $d$ starting at $0$.
\end{example}

\begin{example}
There is a natural notion of an FI-module being generated and related in finite degree (see
\cite{ChurchEllenbergFarbFI} for the definition).
As was observed in \cite[Example 4.18]{RandalWilliamsWahl}, it follows from work of
Church--Ellenberg \cite{ChurchEllenbergHomology} that if $M$ is generated in degree $d$
and related in degree $r$, then $M$ is polynomial of degree $d$ starting at
$r+\min(d,r)$.  We remark that though $r+\min(d,r)$ is sharp in general, in practice
many such $\FI$-modules are polynomial of degree $d$ starting far earlier than $r+\min(d,r)$, often at $0$.
\end{example}

\subsection{Symmetric groups}

Our main theorem about the symmetric group is as follows:

\begin{maintheorem}
\label{maintheorem:sn}
Let $\bbk$ be a commutative ring and let $M$ be an $\FI$-module over $\bbk$ that is polynomial
of degree $d \geq -1$ starting at $m \geq 0$.  For each $k \geq 0$, the map
\[\HH_k(\fS_{n};M(\overline{n})) \rightarrow \HH_k(\fS_{n+1};M(\overline{n+1}))\]
is an isomorphism for $n \geq 2k+\max(d,m)+1$ and a surjection for $n=2k+\max(d,m)$.
\end{maintheorem}

In the key case $m=0$, the stabilization map is thus an isomorphism for $n \geq 2k+d+1$ and
a surjection for $n=2k+d$.  Theorem \ref{maintheorem:sn} is the natural output of the machine we develop, but
for $m \gg 0$ the following result sometimes gives better bounds (though often worse ones):

\begin{maintheoremprime}
\label{maintheorem:snprime}
Let $\bbk$ be a commutative ring and let $M$ be an $\FI$-module over $\bbk$ that is polynomial
of degree $d \geq -1$ starting at $m \geq 0$.  For each $k \geq 0$, the map
\begin{equation}
\label{eqn:snstabmap}
\HH_k(\fS_{n};M(\overline{n})) \rightarrow \HH_k(\fS_{n+1};M(\overline{n+1}))
\end{equation}
is an isomorphism for $n \geq \max(m,2k+2d+2)$ and a surjection for $n \geq \max(m,2k+2d)$.
\end{maintheoremprime}

\begin{remark}
Theorem \ref{maintheorem:snprime} will be derived from Theorem \ref{maintheorem:sn} using
\eqref{eqn:polyfiseq}.
\end{remark}

The first twisted homological stability result for the symmetric group was proved by
Betley \cite{BetleySymmetric}, who only considered split coefficient systems starting at $m=0$ (c.f.\ Remark \ref{remark:split})
but proved that \eqref{eqn:snstabmap} is an isomorphism for $n \geq 2k+d$.  
Randal-Williams--Wahl \cite[Theorem 5.1]{RandalWilliamsWahl}
showed how to deal with general polynomial coefficients, but in this level of generality could only prove that \eqref{eqn:snstabmap}
is an isomorphism for $n \geq \max(2m+1,2k+2d+2)$, though in the split case they could reduce this to
$n \geq \max(m+1,2k+d+2)$.

\subsection{VIC-modules}
\label{section:vicmodules}

We now turn to general linear groups.  Let $R$ be a ring, possibly noncommutative.  To encode our
coefficient systems on $\GL_n(R)$, we will use $\VIC(R)$-modules, which were introduced by the author
and Sam \cite{PutmanSam, PutmanSamNonCom}.  Define $\VIC(R)$ to be the following category:
\begin{itemize}
\item The objects of $\VIC(R)$ are finite-rank free right $R$-modules.\footnote{We use right $R$-modules since this
ensures that $\GL_n(R)$ acts on $R^n$ on the left.}  For a finite-rank free right $R$-module $A$, we
will sometimes write $[A]$ for the associated object of $\VIC(R)$ to clarify our statements.
\item For finite-rank free right $R$-modules $A_1$ and $A_2$, a $\VIC(R)$-morphism $[A_1] \rightarrow [A_2]$
is a pair $(f,C)$, where $f\colon A_1 \rightarrow A_2$ is an injection and $C \subset A_2$ is a
submodule\footnote{It would also be reasonable to require $C$ to be free with $\rk(C) = \rk(A_2) - \rk(A_1)$.  We will not
make this restriction.  For the rings we consider, this frequently holds automatically, at least when $\rk(A_2) - \rk(A_1)$ is large enough (see Lemma \ref{lemma:stabledecomp}).} such that $A_2 = f(A_1) \oplus C$.  The composition of $\VIC(R)$-morphisms $(f_1,C_1)\colon [A_1] \rightarrow [A_2]$
and $(f_2,C_2)\colon [A_2] \rightarrow [A_3]$ is the $\VIC(R)$-morphism $(f_2 \circ f_1, C_2 \oplus f_1(C_1))\colon [A_1] \rightarrow [A_3]$.
\end{itemize}
For a commutative ring $\bbk$, a $\VIC(R)$-module over $\bbk$ is a functor $M$ from $\VIC(R)$ to the category
of $\bbk$-modules.  Every object of $\VIC(R)$ is isomorphic to $R^n$ for some $n \geq 0$, so the
data of a $\VIC(R)$-module $M$ consists of:
\begin{itemize}
\item a $\bbk$-module $M(R^n)$ for each $n \geq 0$, and
\item for each $\VIC(R)$-morphism $(f,C)\colon [R^n] \rightarrow [R^m]$, an induced $\bbk$-module homomorphism
$(f,C)_{\ast}\colon M(R^n) \rightarrow M(R^m)$.
\end{itemize}
For each $n$, let $\iota_n\colon R^n \rightarrow R^{n+1}$ be the inclusion into the first $n$ coordinates.
We then have a $\VIC(R)$-morphism $(\iota_n,0 \oplus R)\colon [R^n] \rightarrow [R^{n+1}]$.  These
induce a sequence of morphisms
\begin{equation}
\label{eqn:increasingvic}
M(R^0) \rightarrow M(R^1) \rightarrow M(R^2) \rightarrow \cdots.
\end{equation}
The group of $\VIC(R)$-automorphisms of $[R^n]$ is $\GL_n(R)$.  This acts
on $M(R^n)$, making $M(R^n)$ into a $\bbk[\GL_n(R)]$-module.  More generally, for a finite-rank free
right $R$-module $A$ the group $\GL(A)$ acts on $M(A)$.

\begin{example}
\label{example:easyvic1}
We can fit the increasing sequence
\[\Z^0 \rightarrow \Z^1 \rightarrow \Z^2 \rightarrow \cdots\]
of $\GL_n(\Z)$-representations into a $\VIC(\Z)$-module $M$ over $\Z$ by defining
\[M(A) = A \quad \text{for a finite-rank free $\Z$-module $A$}.\]
For a $\VIC(\Z)$-morphism $(f,C)\colon [A_1] \rightarrow [A_2]$, the associated
map $(f,C)_{\ast}\colon M(A_1) \rightarrow M(A_2)$ is simply $f$.
\end{example}

\begin{example}
\label{example:easyvic2}
Example \ref{example:easyvic1} can be generalized as follows.  Let $R$ be a ring, let $\bbk$ be a commutative ring,
let $V$ be a $\bbk$-module, and let $\lambda\colon R \rightarrow \End_{\bbk}(V)$ be a ring homomorphism.
Example \ref{example:easyvic1} will correspond to $R = \bbk = V = \Z$ and 
\[\lambda(r)(v) = r v \quad \text{for $r \in R = \Z$ and $v \in V = \Z$}.\]
Another example would be $R = \bbk[G]$ for a group $G$ and $V$ a representation of $G$ over $\bbk$.
For each $n \geq 0$, the ring homomorphism $\lambda$ induces a group homomorphism
\[\GL_n(R) \rightarrow \GL_n(\End_{\bbk}(V)),\]
endowing the $\bbk$-module $V^{\oplus n}$ with the structure of a $\bbk[\GL_n(R)]$-module.  
We can fit the increasing sequence
\[V^{\oplus 0} \rightarrow V^{\oplus 1} \rightarrow V^{\oplus 2} \rightarrow \cdots\]
of $\GL_n(R)$-representations into a $\VIC(R)$-module $M$ by defining
\[M(A) = A \otimes_R V \quad \text{for a finite-rank free right $R$-module $A$}.\]
Here we use $\lambda$ to regard $V$ as a left $R$-module.  For a $\VIC(R)$-morphism
$(f,C)\colon [A_1] \rightarrow [A_2]$, the induced map $(f,C)_{\ast}\colon M(A_1) \rightarrow M(A_2)$
is $f \otimes \text{id}$.  As a $\bbk[\GL_n(R)]$-module, we have
\[M(R^n) = R^n \otimes_R V = V^{\oplus n}.\qedhere\]
\end{example}

\begin{example}
\label{example:easyvic3}
In Examples \ref{example:easyvic1} and \ref{example:easyvic2}, the $C$ in a $\VIC(R)$-morphism played no role.  For
an easy example of a $\VIC(R)$-module where it is important, consider the dual $M^{\ast}$ of the $\VIC(\Z)$-module
$M$ over $\Z$ from Example \ref{example:easyvic1}:
\[M^{\ast}(A) = \Hom(A,\Z) \quad \text{for a finite-rank free $\Z$-module $A$}.\]
For a $\VIC(\Z)$-morphism $(f,C)\colon [A_1] \rightarrow [A_2]$, the associated
map $(f,C)_{\ast}\colon M^{\ast}(A_1) \rightarrow M^{\ast}(A_2)$ takes $\phi \in \Hom(A_1,\Z)$ to the composition
\[A_2 \stackrel{/C}{\longrightarrow} f(A_1) \stackrel{f^{-1}}{\longrightarrow} A_1 \stackrel{\phi}{\longrightarrow} \Z,\]
where the first map is the projection $A_2 = f(A_1) \oplus C \rightarrow f(A_1)$.
\end{example}

\subsection{Polynomial VIC-modules}

For a $\VIC(R)$-module $M$ over $\bbk$, the inclusions \eqref{eqn:increasingvic} induce maps between homology groups
for each $k$:
\[\HH_k(\GL_0(R);M(R^0)) \rightarrow \HH_k(\GL_1(R);M(R^1)) \rightarrow \HH_k(\GL_2(R);M(R^2)) \rightarrow \cdots.\]
Just like for for $\FI$-modules, to make this stabilize we will need to impose a polynomiality condition.\footnote{For finite rings $R$,
the author and Sam proved in \cite{PutmanSam, PutmanSamNonCom} a homological stability result which replaces
the polynomiality condition by a much weaker ``finite generation'' condition.  That proof is very different from the one
we will give, and cannot possibly work for infinite rings like $R = \Z$.} The definitions are similar to those for $\FI$-modules:

\begin{definition}
\label{definition:derivedvic}
Let $R$ be a ring, $\bbk$ be a commutative ring, and $M$ be a $\VIC(R)$-module over $\bbk$.
\begin{itemize}
\item The {\em shifted} $\VIC(R)$-module of $M$, denoted $\Sigma M$, is the $\VIC(R)$-module over $\bbk$ defined via the formula
$\Sigma M(A) = M(A \oplus R^1)$ for a finite-rank free right $R$-module $A$.
\item The {\em derived} $\VIC(R)$-module
of $M$, denoted $DM$, is the $\VIC(R)$-module over $\bbk$ defined via the formula
\[DM(A) = \frac{M(A \oplus R^1)}{\Image(M(A) \rightarrow M(A \oplus R^1))} \quad \text{for a finite-rank free right $R$-module $A$}.\qedhere\]
\end{itemize}
\end{definition}

\begin{definition}
\label{definition:polyvic}
Let $R$ be a ring, $\bbk$ be a commutative ring, and $M$ be a $\VIC(R)$-module over $\bbk$.
We say that $M$ is {\em polynomial} of degree $d \geq -1$ starting at $m \in \Z$ if it satisfies
the following inductive condition:
\begin{itemize}
\item If $d = -1$, then for all finite-rank free right $R$-modules $A$ with $\rk(A) \geq m$ we require $M(A) = 0$.
\item If $d \geq 0$, then we require the following two conditions:
\begin{itemize}
\item For all $\VIC(R)$-morphisms $(f,C) \colon [A_1] \rightarrow [A_2]$ with $\rk(A_1) \geq m$,
the induced map $(f,C)_{\ast}\colon M(A_1) \rightarrow M(A_2)$ must be an injection.
\item The derived $\VIC(R)$-module $DM$ must be polynomial of degree $(d-1)$ starting at $(m-1)$.\qedhere
\end{itemize}
\end{itemize}
\end{definition}

\begin{example}
The $\VIC(R)$-modules in Examples \ref{example:easyvic1}, \ref{example:easyvic2}, and \ref{example:easyvic3} are all
polynomial of degree $1$ starting at $0$.  Letting $M$ be one of these, for $d \geq 0$ we can define another
$\VIC(R)$-module $M^{\otimes d}$ via the formula
\[M^{\otimes d}(A) = \left(M\left(A\right)\right)^{\otimes d} \quad \text{for a finite-rank free right $R$-module $A$}.\]
This is easily seen to be polynomial of degree $d$ starting at $0$.
\end{example}

\subsection{General linear groups}

We now turn to our stability theorem, which will concern the homology of $\GL_n(R)$.
We will need to impose a ``stable rank'' condition on $R$ called $(\SR_r)$ that was introduced
by Bass \cite{BassKTheory}.  See \S \ref{section:stablerank} below for the definition and a
survey.  Here we will simply say this condition is satisfied by
many rings; in particular, fields satisfy $(\SR_2)$ and PIDs satisfy $(\SR_3)$.  More generally
a ring $R$ that is finitely generated as a module over a Noetherian commutative
ring of Krull dimension $r$ satisfies $(\SR_{r+2})$.  Thus for instance if $\bbk$ is
a field and $G$ is a finite group, then the group ring $\bbk[G]$ satisfies $(\SR_2)$.

\begin{maintheorem}
\label{maintheorem:gl}
Let $R$ be a ring satisfying $(\SR_r)$, let $\bbk$ be a commutative ring, and let
$M$ be a $\VIC(R)$-module over $\bbk$ that is polynomial of degree $d \geq -1$
starting at $m \geq 0$.
For each $k \geq 0$, the map
\[\HH_k(\GL_n(R);M(R^n)) \rightarrow \HH_k(\GL_{n+1}(R);M(R^{n+1}))\]
is an isomorphism for $n \geq 2k+\max(2d+r-1,m,r)+1$ and a surjection for $n = 2k+\max(2d+r-1,m,r)$.
\end{maintheorem}

In the key case $d \geq 1$\footnote{That is, for $M$ not constant.} and $m=0$,
the stabilization map is thus an isomorphism for $n \geq 2k+2d+r$
and a surjection for $n = 2k+2d+r-1$.  Just like for the symmetric group, we will derive from
Theorem \ref{maintheorem:gl} the following variant result which for $m \gg 0$ sometimes gives better bounds:

\begin{maintheoremprime}
\label{maintheorem:glprime}
Let $R$ be a ring satisfying $(\SR_r)$, let $\bbk$ be a commutative ring, and let
$M$ be a $\VIC(R)$-module over $\bbk$ that is polynomial of degree $d \geq -1$ starting at $m \geq 0$.
For each $k \geq 0$, the map
\begin{equation}
\label{eqn:glstabmap}
\HH_k(\GL_n(R);M(R^n)) \rightarrow \HH_k(\GL_{n+1}(R);M(R^{n+1}))
\end{equation}
is an isomorphism for $n \geq \max(m,2k+2d+r+2)$ and a surjection for $n \geq \max(m,2k+2d+r)$.
\end{maintheoremprime}

Dwyer \cite{DwyerTwisted} proved a version of this for $R$ a PID,
though he only worked with split coefficient systems starting at $m=0$ (c.f.\ Remark \ref{remark:split}).  He
did not identify a stable range.  Later van der Kallen \cite[Theorem 5.6]{VanDerKallen} extended this
to rings satisfying $(\SR_r)$, though again he only worked with split coefficient systems
starting at $m=0$.  He proved that \eqref{eqn:glstabmap} is an isomorphism for $n \geq 2k+d+r$.
Randal-Williams--Wahl \cite[Theorem 5.11]{RandalWilliamsWahl} then showed how to deal with
general polynomial modules, though they only stated a result for ones that started at $m=0$.
They proved that \eqref{eqn:glstabmap} is an isomorphism for $n \geq 2k+2d+r+1$.

\begin{remark}
The result we will prove is more general than Theorems \ref{maintheorem:gl} and \ref{maintheorem:glprime} and
applies to certain subgroups of $\GL_n(R)$ as well.\footnote{The work of van der Kallen also applies to these
subgroups.}  For instance, if $R$ is commutative then it applies
to $\SL_n(R)$.  See \S \ref{section:ktheory} for the definition of the groups we will consider
and Theorems \ref{theorem:xl} and \ref{theorem:xlprime} for the statement of our theorem.
\end{remark}

\subsection{Congruence subgroups}

Our final theorem illustrates how our machinery can be applied to prove a theorem that is not
(directly) about homological stability.  Borel \cite{BorelStability1, BorelStability2} proved
that if $\Gamma$ is a lattice in $\SL_n(\R)$ and $V$ is a rational representation of the algebraic
group $\SL_n$, then for $n \gg k$ the homology group $\HH_k(\Gamma;V)$ depends only on $k$ and $V$, not
on the lattice $\Gamma$.  In particular, it is unchanged if you pass from $\Gamma$ to a finite-index
subgroup.

We will prove a version of this for $\GL_n(R)$ for rings $R$ satisfying $(\SR_r)$.  The basic
flavor of the result will be that passing from $\GL_n(R)$ to appropriate finite-index subgroups
does not change rational homology, at least in some stable range.  The finite-index subgroups
we will consider are the finite-index congruence subgroups, which are defined as follows:

\begin{definition}
Let $R$ be a ring and let $\alpha$ be a two sided ideal of $R$.  The {\em level-$\alpha$ congruence
subgroup}\footnote{In many contexts it is common to call these {\em principal congruence subgroups}, and
to define a congruence subgroup as a subgroup that contains a principal congruence subgroup.  For general
rings $R$, this more general notion
of ``congruence subgroups'' can have some pathological properties.
For number rings $R$, all nontrivial ideals $\alpha$
of $R$ satisfy $|R/\alpha|<\infty$, so for number rings all principal congruence subgroups are finite-index.  
This is false for general rings.  Even worse, as Tom Church pointed out to me there might exist finite-index subgroups
$\Gamma < \GL_n(R)$ that contain principal congruence subgroups, but do not contain
finite-index principal congruence subgroups (though I do not know any concrete examples of this).}
of $\GL_n(R)$, denoted $\GL_n(R,\alpha)$, is the kernel of the natural group
homomorphism $\GL_n(R) \rightarrow \GL_n(R/\alpha)$.
\end{definition}

Our theorem is as follows:

\begin{maintheorem}
\label{maintheorem:congruence}
Let $R$ be a ring satisfying $(\SR_r)$, let $\bbk$ be a field of characteristic $0$, and let
$M$ be a $\VIC(R)$-module over $\bbk$ that is polynomial of degree $d \geq -1$ starting at $m \geq 0$.
Assume furthermore that $M(R^n)$ is a finite-dimensional vector space over $\bbk$ for all $n \geq 0$.
Then for all two sided ideals $\alpha$ of $R$ such that $|R/\alpha| < \infty$, the map
\[\HH_k(\GL_n(R,\alpha);M(R^n)) \rightarrow \HH_k(\GL_n(R);M(R^n))\]
is an isomorphism for $n \geq \max(m,2k+2d+2r-1)$.
\end{maintheorem}

We emphasize that Theorem \ref{maintheorem:congruence} is {\em not} a homological stability theorem:\footnote{Though in
light of Theorems \ref{maintheorem:gl} and \ref{maintheorem:glprime} it implies that $\HH_k(\GL_n(R,\alpha);M(R^n))$ stabilizes.}
rather than increasing $\GL_n(R)$ to $\GL_{n+1}(R)$, we are decreasing $\GL_n(R)$ by passing
to the finite-index subgroup $\GL_n(R,\alpha)$.  For untwisted rational coefficients,
something like Theorem \ref{maintheorem:congruence} is implicit in work of Charney \cite{CharneyCongruence}.
Though she did not state this, it can easily be derived from her work that for $R$ and $\bbk$ and
$\alpha$ as in Theorem \ref{maintheorem:congruence}, the map
\[\HH_k(\GL_n(R,\alpha);\bbk) \rightarrow \HH_k(\GL_n(R);\bbk)\]
is an isomorphism for $n \geq 2k+2r+4$.\footnote{The constant coefficients $\bbk$ fit into an $\FI$-module
of degree $0$ starting at $0$, so even in this case Theorem \ref{maintheorem:congruence} gives
the slightly better stable range $n \geq 2k+2r-1$.}  Later Cohen \cite{CohenCongruence} proved an
analogous result for the symplectic group, and more importantly for us showed how
to simplify Charney's argument.  Our proof of Theorem \ref{maintheorem:congruence} follows
the outline of Cohen's proof, but using our new approach to twisted homological stability.
It seems very hard to do this using the traditional proof of twisted homological stability.

\begin{remark}
The fact that the field $\bbk$ in Theorem \ref{maintheorem:congruence} has characteristic $0$ is essential.
This theorem is false over fields of finite characteristic or over more general rings like $\Z$.  It can fail
even when $M$ is polynomial of degree $0$, i.e., constant.  For instance, for $n \geq 3$ we have
$\HH_1(\GL_n(\Z);\Z) = \Z/2$ with the $\Z/2$ coming from the determinant, but Lee--Szczarba \cite{LeeSzczarbaCongruence} proved that\footnote{Their
theorem is actually about congruence subgroups of $\SL_n(\Z)$, but for $\ell \geq 3$ all elements of $\GL_n(\Z,\ell\Z)$ have
determinant $1$.}
\[\HH_1(\GL_n(\Z,\ell \Z);\Z) \cong \mathfrak{sl}_n(\Z/\ell) \quad \text{for $n,\ell \geq 3$},\]
where $\mathfrak{sl}_n(\Z/\ell)$ is the abelian group of $n \times n$ trace $0$ matrices over $\Z/\ell$.
\end{remark}

\begin{remark}
Just like for Theorems \ref{maintheorem:gl} and \ref{maintheorem:glprime}, we will actually
prove something more general than Theorem \ref{maintheorem:congruence} that will apply
to congruence subgroups of certain subgroups of $\GL_n(R)$, e.g., to $\SL_n(R)$ if $R$ is
commutative.  See Theorem \ref{theorem:xlcongruence}.
\end{remark}

\begin{remark}
The proof of Theorem \ref{maintheorem:congruence} also requires some recent
work of Harman \cite{HarmanVIC} classifying certain kinds of finitely generated $\VIC(\Z)$-modules.
\end{remark}

\begin{remark}
If the ring $R$ in Theorem \ref{maintheorem:congruence} is a {\em finite} ring, then we can take $\alpha = \{0\}$, so $\GL_n(R,\alpha)$ is trivial.
The case $k=0$ of the theorem thus implies that under its hypotheses, for $R$ a finite ring the action of $\GL_n(R)$ on
$M(R^n)$ is {\em trivial} for $n \geq \max(m,2d+2r)$.  It is enlightening to go through our proof and see
how it proves this special case.  This will also clarify to the reader how the work of Harman
discussed in the previous remark is used.
\end{remark}

\begin{remark}
Though I have not worked out the details, I expect that a version of Theorem \ref{maintheorem:congruence}
for the symplectic group can also be proved by these methods.  Indeed, as I mentioned above
our proof is inspired by work of Cohen \cite{CohenCongruence} on the symplectic group.
\end{remark}

\subsection{Outline}
We start with three sections of preliminary results in \S \ref{section:simplicial} -- \S \ref{section:coefficients}.  We then
discuss our twisted homological stability machine in \S \ref{section:stabilitymachine}.  To make this useful we also
need an accompanying ``vanishing theorem'', which is in \S \ref{section:vanishingtheorem}.  We then have
\S \ref{section:sn} on symmetric groups, which proves Theorems \ref{maintheorem:sn} and \ref{maintheorem:snprime}.
After that, we have four sections of background on rings and general linear groups: \S \ref{section:stablerank}
discusses the stable rank condition $(\SR_r)$, \S \ref{section:ktheory} has some K-theoretic background, 
and \S \ref{section:splitbases}--\S \ref{section:glncoefficient}
introduce some important simplicial complexes associated to $\GL_n(R)$.  We then have
\S \ref{section:glstability} on general linear groups, which proves Theorems \ref{maintheorem:gl} and
\ref{maintheorem:glprime}.  We finally turn our attention to congruence subgroups.  This requires some
preliminary results on unipotent representations that are discussed in \S \ref{section:congruenceunipotence}.  We
close with \S \ref{section:congruencestability} on congruence subgroups, which proves Theorem \ref{maintheorem:congruence}.

\subsection{Acknowledgments}
I would like to thank Nate Harman, Jeremy Miller, and Nick Salter for helpful conversations.  In particular, I would
like to thank Nate Harman for explaining how to prove Lemma \ref{lemma:vicunipotent} below.  I would also like to
thank Tom Church, Benson Farb, Peter Patzt, and Nathalie Wahl for providing comments on a previous draft of this paper.
Finally, I would like to thank the referee for an unusually thorough and detailed report.

\section{Background I: simplicial complexes}
\label{section:simplicial}

This section contains background material on simplicial complexes.  Its main purpose is to establish
notation.  See \cite{FriedmanSemisimplicial} for more details.

\subsection{Basic definitions}
A {\em simplicial complex} $X$ consists of the following data:
\begin{itemize}
\item A set $X^{(0)}$ called the {\em $0$-simplices} or {\em vertices}.
\item For each $k \geq 1$, a set $X^{(k)}$ of $(k+1)$-element subsets of the vertices called the {\em $k$-simplices}.  These
are required to satisfy the following condition:
\begin{itemize}
\item Consider $\sigma \in X^{(k)}$.  Then for $\sigma' \subseteq \sigma$ with $|\sigma'|>0$, we must have $\sigma' \in X^{(|\sigma'|-1)}$.
In this case, we call $\sigma'$ a {\em face} of $\sigma$.
\end{itemize}
\end{itemize}
A simplicial complex $X$ has a geometric realization $|X|$ obtained by gluing together
geometric $k$-simplices (one for each $k$-simplex in $X^{(k)}$) according to the face relation.
Whenever we talk about topological properties of $X$ (e.g., being connected), we are
referring to its geometric realization.

\subsection{Links and Cohen--Macaulay}
\label{section:cmcomplex}

Let $X$ be a simplicial complex.  The {\em link} of a simplex $\sigma$ of $X$, denoted
$\Link_X(\sigma)$, is the subcomplex of $X$ consisting of all simplices $\tau$
satisfying the following two conditions.
\begin{itemize}
\item The simplices $\tau$ and $\sigma$ are disjoint, i.e., have no vertices in common.
\item The union $\tau \cup \sigma$ is a simplex.
\end{itemize}
We say that $X$ is {\em weakly Cohen--Macaulay} of dimension $n \in \Z$
if it satisfies the following:
\begin{itemize}
\item The complex $X$ must be $(n-1)$-connected.  Here our convention is that a space is $(-1)$-connected if
it is nonempty, and all spaces are $k$-connected for $k \leq -2$.
\item For all $k$-simplices $\sigma$ of $X$, the complex $\Link_X(\sigma)$ must be $(n-k-2)$-connected.
\end{itemize}

\begin{example}
Let $X$ be a PL triangulation of an $n$-sphere.  We claim that $X$ is weakly Cohen--Macaulay
of dimension $n$.  There are two things to check:
\begin{itemize}
\item The simplicial complex $X$ is $(n-1)$-connected, which is clear.
\item For a $k$-simplex $\sigma$ of $X$, the complex $\Link_X(\sigma)$ is $(n-k-2)$-connected.
In fact, since our triangulation is PL the link of $\sigma$
is a PL triangulation of an $(n-k-1)$-sphere.\qedhere
\end{itemize}
\end{example}

\begin{remark}
\label{remark:weakmeaning}
The adjective ``weak'' is here since we do not require $X$ to be $n$-dimensional, nor do we require links
of $k$-simplices to be $(n-k-1)$-dimensional.  However, since our conventions imply that the empty set is
not $(-1)$-connected, if $X$ is weakly Cohen--Macaulay of dimension $n$ then $X$ is at least $n$-dimensional
and the links of $k$-simplices are at least $(n-k-1)$-dimensional.
\end{remark}

One basic property of weakly Cohen--Macaulay complexes is as follows:

\begin{lemma}
\label{lemma:linkcm}
Let $X$ be a simplicial complex that is weakly Cohen--Macaulay of dimension $n$ and let $\sigma$
be a $k$-simplex of $X$.  Then $\Link_X(\sigma)$ is weakly Cohen--Macaulay of dimension $(n-k-1)$.
\end{lemma}
\begin{proof}
By definition, $\Link_X(\sigma)$ is $(n-k-2)$-connected.  Also, if $\tau$ is an $\ell$-simplex
of $\Link_X(\sigma)$, then 
$\tau \cup \sigma$ is a $(k+\ell+1)$-simplex of $X$ and
$\Link_{\Link_X(\sigma)}(\tau) = \Link_X(\tau \cup \sigma)$.
This is $n-(k+\ell+1)-2 = n-k-\ell-3$ connected by assumption.
\end{proof}

\section{Background II: semisimplicial sets and ordered simplicial complexes}
\label{section:semisimplicial}

The category of simplicial complexes has some undesirable features that make it awkward for homological
stability proofs.  For instance, if $X$ is a simplicial complex and $G$ is a group acting on $X$, then
one might expect the quotient $X/G$ to be a simplicial complex whose $k$-simplices are the $G$-orbits
of simplices of $X$.  Unfortunately, this need not hold.  In this section, we discuss the category
of semisimplicial sets, which does not have such pathologies.  See \cite{FriedmanSemisimplicial}\footnote{This reference
calls semisimplicial sets $\Delta$-sets} and \cite{EbertRandalWilliamsSemisimplicial} for
more details.

\subsection{Semisimplicial sets}
Let $\Delta$ be the category whose objects are the finite sets $[k] = \{0,\ldots,k\}$ with $k \geq 0$ and
whose morphisms $[\ell] \rightarrow [k]$ are order-preserving injections.  A {\em semisimplicial set}
is a contravariant functor $\bbX\colon \Delta \rightarrow \cSet$.  Unwinding this,
$\bbX$ consists of the following two pieces of data: 
\begin{itemize}
\item For each $k \geq 0$, a set $\bbX^k = \bbX([k])$ called the {\em $k$-simplices}.
\item For each order-preserving injection $\iota\colon [\ell] \rightarrow [k]$, a map
$\iota^{\ast}\colon \bbX^k \rightarrow \bbX^{\ell}$ called a {\em face map}.  For $\sigma \in \bbX^k$, the
image $\iota^{\ast}(\sigma) \in \bbX^{\ell}$ is called a {\em face} of $\sigma$.
\end{itemize}
A semisimplicial set $\bbX$ has a geometric realization $|\bbX|$ obtained by gluing geometric
$n$-simplices together (one for each $n$-simplex) according to the face maps.  See
\cite{FriedmanSemisimplicial} 
for more details.  Whenever we talk about topological properties of $\bbX$
(e.g., being connected), we are referring to its geometric realization.

\subsection{Morphisms}
Below we will compare semisimplicial sets with simplicial complexes, but first we describe
their morphisms.
A morphism $f\colon \bbX \rightarrow \bbY$ between semisimplicial sets $\bbX$ and $\bbY$ is a natural
transformation between the functors $\bbX$ and $\bbY$.  In other words,
$f$ consists of set maps $f_k\colon \bbX^k \rightarrow \bbY^k$ for each
$k \geq 0$ that commute with the face maps.  Such a morphism induces a continuous map
$|f|\colon |\bbX| \rightarrow |\bbY|$.  Using this definition, an action of a group
$G$ on a semisimplicial set $\bbX$ consists of actions of $G$ on each set $\bbX^k$
that commute with the face maps, and $\bbX/G$ is a semisimplicial set with
$(\bbX/G)^k = \bbX^k/G$ for each $k \geq 0$.

\subsection{Semisimplicial sets vs simplicial complexes}

Let $\bbX$ be a semisimplicial set.  The {\em vertices} of $\bbX$ are the elements of the set $\bbX^0$ of $0$-simplices.
For each $k$-simplex $\sigma \in \bbX^k$, we can define an ordered $(k+1)$-tuple $(v_0,\ldots,v_k)$ of vertices
via the formula
\[v_i = \iota_i^{\ast}(\sigma) \text{\ with $\iota_i\colon [0] \rightarrow [k]$ the map $\iota_i(0) = i$}.\]
We will call $v_0,\ldots,v_k$ the vertices of $\sigma$.  This is similar to a simplicial complex, whose vertices are $(k+1)$-element sets of vertices.  However, there
are three essential differences:
\begin{itemize}
\item The vertices of a simplex in a semisimplicial set have a natural ordering.
\item The vertices of a simplex in a semisimplicial set need not be distinct.
\item Simplices in a semisimplicial set are {\em not} determined by their vertices.
\end{itemize}

\subsection{Ordered simplicial complexes}

Ordered simplicial complexes are special kinds of semisimplicial sets where simplices are determined
by the vertices, which are also required to be distinct.  More precisely, an {\em ordered simplicial complex}
is a semisimplicial set $\bbX$ of the following form:
\begin{itemize}
\item The vertices $\bbX^0$ are an arbitrary set $V$.
\item For each $k \geq 1$, the set $\bbX^k$ of $k$-simplices is a subset of
\[\Set{$(v_0,\ldots,v_k)$}{the $v_i$ are distinct elements of $V$} \subset V^k.\]
\item For an order-preserving injection $\iota\colon [\ell] \rightarrow [k]$, the face
map $\iota^{\ast}\colon \bbX^k \rightarrow \bbX^{\ell}$ equals\footnote{Here we write $\iota^{\ast}(v_0,\ldots,v_k)$ rather than
the more precise but ugly $\iota^{\ast}((v_0,\ldots,v_k))$.  We will silently do this kind of thing throughout the paper.}
\[\iota^{\ast}(v_0,\ldots,v_k) = (v_{\iota(0)},\ldots,v_{\iota(\ell)}) \quad \text{for all $(v_0,\ldots,v_k) \in \bbX^k$}.\]
In particular, for $(v_0,\ldots,v_k) \in \bbX^k$, all tuples obtained by deleting some of the $v_i$ are also simplices.
\end{itemize}

\begin{remark}
Though many of the concrete spaces we will consider are ordered simplicial complexes, the quotient of an ordered
simplicial complex by a group action is not necessarily an ordered simplicial complex.  Because of this,
we will need general semisimplicial sets for the homological stability machine.
\end{remark}

\subsection{Small ordering}
Consider a simplicial complex $X$.  There are two natural ways to associate to $X$ an ordered simplicial
complex.  First, a {\em small ordering}
of $X$ is an ordered simplicial complex $\bbX$ with the following properties:
\begin{itemize}
\item[(a)] The vertices of $\bbX$ are the same as the vertices of $X$.
\item[(b)] The (unordered) set of vertices of a $k$-simplex of $\bbX$ is a $k$-simplex of $X$.
\item[(c)] The map $\bbX^k \rightarrow X^{(k)}$ taking a simplex to its set of vertices is a bijection.
\end{itemize}
The only difference between $\bbX$ and $X$ is thus that
the vertices of a simplex of $\bbX$ have a natural ordering, while the vertices of a simplex of $X$ are unordered.
It is clear that the geometric realizations $|\bbX|$ and $|X|$ are homeomorphic.

Small orderings always exist; for instance, if you choose a total ordering on the the vertices of $X$, then
you can construct a small ordering $\bbX$ of $X$ by letting
\[\bbX^k = \Set{$(v_0,\ldots,v_k) \in \left(X^{(0)}\right)^{k+1}$}{$\{v_0,\ldots,v_k\} \in X^{(k)}$ and $v_0 < \cdots < v_k$}\]
for all $k \geq 0$.
However, since this depends on the total ordering on $X^{(0)}$ it might have fewer symmetries than $X$.  The
example we will give of this will be used repeatedly, so we make a formal definition.

\begin{definition}
For $n \geq 0$, let $\Sim_n$ be the {\em $n$-simplex},\footnote{It would be natural to instead
denote this by $\Delta^n$ or $\Delta_n$, but we are already using $\Delta$ for the category of finite sets used to define
semisimplicial sets.} i.e., the simplicial complex whose vertex set is
$[n] = \{0,\ldots,n\}$ and whose $k$-simplices are all $(k+1)$-element subsets of $[n]$.  The geometric realization $|\Sim_n|$ is the usual
geometric $n$-simplex.
\end{definition}

\begin{example}
The symmetric group $\fS_{n+1}$ acts on $\Sim_n$ via its natural action on $[n] = \{0,\ldots,n\}$.  However, if we choose a total
ordering on $[n]$ and use this to define a small ordering $\bbX$ of $\Sim_n$ as above, then $\fS_{n+1}$ does
not act on $\bbX$ except in the degenerate case $n=0$.  Indeed, it is easy to see that the automorphism group of $\bbX$ is trivial.
\end{example}

\begin{remark}
If a group $G$ acts on a simplicial complex $X$, then $G$ also acts on the set $\cO(X)$ of all small orderings of $X$.  If
a small ordering $\bbX$ of $X$ is a fixed point for the action of $G$ on $\cO(X)$, then $G$ acts on $\bbX$.
\end{remark}

\begin{remark}
One way to construct a small ordering of a simplicial complex $X$ is to choose a relation $\prec$ on $X^{(0)}$ such that
the restriction of $\prec$ to each simplex of $X$ is a total ordering.  Since $\prec$ is only a relation (not even a partial ordering!), 
not all vertices need to be comparable, and indeed it is natural to only define $\prec$ on vertices $v$ and $w$ such that $\{v,w\}$ is an edge.
In many situations there is a ``natural'' choice of
such a $\prec$ that is preserved by the symmetries that are relevant for the problem at hand.  For instance, this is how
small orderings arise in our followup paper \cite{PutmanStableLevel}.
\end{remark}

\subsection{Large ordering}
If a sufficiently symmetric small ordering of a simplicial complex does not exist, there is another
construction that is often useful.  Let $X$ be a simplicial complex.  The {\em large ordering} of
$X$, denoted $X_{\ord}$, is the ordered simplicial complex $X_{\ord}$ with the same vertex set as $X$ whose $k$-simplices
$X_{\ord}^k$ are {\em all} ordered $(k+1)$-tuples $(v_0,\ldots,v_k)$ of distinct vertices of $X$
such that the unordered set $\{v_0,\ldots,v_k\}$ is a $k$-simplex of $X$.

\Figure{figure:largeordering}{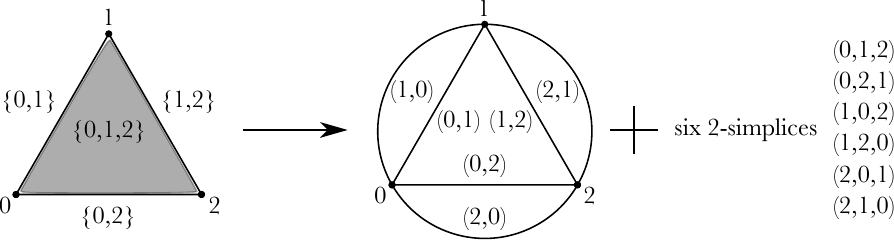}{Left: the $2$-simplex $\Sim_2$.  Right: its large ordering $\OSim_2$.
The drawn portion is the $1$-skeleton.}{0.9}

The following example of this will be used several times, so we introduce notation for it.

\begin{definition}
For $n \geq 0$, let $\OSim_n$ be the large ordering $\left(\Sim_n\right)_{\ord}$ of the $n$-simplex
$\Sim_n$.  The $k$-simplices of $\OSim_n$ are thus ordered sequences $(i_0,\ldots,i_k)$ of distinct
elements of $[n]$.
\end{definition}

\begin{example}
\label{example:ordernsimplex}
Each $\Sim_n$ is contractible.  However, none of the $\OSim_n$ are contractible except
for $\OSim_0$.  For instance, $\OSim_1$ has two $0$-cells $0$ and $1$ and two $1$-cells $(0,1)$ and $(1,0)$, and
its geometric realization is homeomorphic to $S^1$.  For an even more complicated example, see
the picture of $\OSim_2$ in Figure \ref{figure:largeordering}.
\end{example}

The semisimplicial set $X_{\ord}$ is much larger than $X$; indeed, each $k$-simplex of $X$ corresponds to $(k+1)!$ simplices
of $X_{\ord}$, one for each total ordering of its vertices.  It is clear from its construction that if a group $G$ acts
on $X$, then $G$ also acts on $X_{\ord}$.  However, examples like Example \ref{example:ordernsimplex}
might lead one to think that there is no simple relationship between the topologies of $X$ and $X_{\ord}$.  This makes
the following theorem of Randal-Williams--Wahl \cite[Proposition 2.14]{RandalWilliamsWahl} somewhat surprising.  
See also \cite[Proposition 2.10]{HatcherVogtmannTethers}
for an alternate proof.

\begin{theorem}[{\cite[Theorem 2.14]{RandalWilliamsWahl}}]
\label{theorem:largeordering}
Let $X$ be a simplicial complex that is weakly Cohen--Macaulay of dimension $n$.  Then $X_{\ord}$ is $(n-1)$-connected.
\end{theorem}

\begin{example}
Since $\Sim_n$ is clearly weakly Cohen--Macaulay of dimension $n$, Theorem \ref{theorem:largeordering}
implies that $\OSim_n$ is $(n-1)$-connected.  The semisimplicial set $\OSim_n$ is also
called the {\em complex of injective words} on $(n+1)$ letters, and the fact that it
is $(n-1)$-connected was originally proved by Farmer \cite{FarmerWords}.
\end{example}

\begin{remark}
For a simplicial complex $X$, every small ordering of $X$ appears as a subcomplex of $X_{\ord}$.
\end{remark}

\subsection{Forward link and forward Cohen--Macaulay}
\label{section:forwardlink}
 
Let $\bbX$ be an ordered simplicial complex.  As notation, if
$\sigma = (v_0,\ldots,v_k)$ and $\tau = (w_0,\ldots,w_{\ell})$ are ordered
sequences of vertices of $\bbX$, then we will write $\sigma \cdot \tau$ for $(v_0,\ldots,v_k,w_0,\ldots,w_{\ell})$. Of course,
$\sigma \cdot \tau$ need not be a simplex; for instance, its vertices need not be distinct.

Given a simplex $\sigma$ of $\bbX$, the {\em forward link} of $\sigma$,
denoted $\FLink_{\bbX}(\sigma)$, is the ordered simplicial complex whose $\ell$-simplices
are $\ell$-simplices $\tau$ of $\bbX$ such that $\sigma \cdot \tau$ is a $(k+\ell+1)$-simplex of
$\bbX$.  We will say that $\bbX$ is {\em weakly forward Cohen--Macaulay} of dimension $n$
if $\bbX$ is $(n-1)$-connected and for all $k$-simplices $\sigma$ of $\bbX$, the forward
link $\FLink_{\bbX}(\sigma)$ is $(n-k-2)$-connected.  

One source of examples of this is as follows:

\begin{lemma}
\label{lemma:largerorderingcm}
Let $X$ be a simplicial complex that is weakly Cohen--Macaulay of dimension $n$.  Then
its large ordering $X_{\ord}$ is weakly forward Cohen--Macaulay of dimension $n$.
\end{lemma}
\begin{proof}
Theorem \ref{theorem:largeordering} implies that $X_{\ord}$ is $(n-1)$-connected.  Let $\sigma$
be a $k$-simplex of $X_{\ord}$, and let $\osigma$ be the corresponding (unordered) $k$-simplex of
$X$.  Since $X$ is weakly Cohen--Macaulay of dimension $n$, the complex $\Link_X(\osigma)$ is
weakly Cohen--Macaulay of dimension $(n-k-1)$.  Since
\[\FLink_{\bbX}(\sigma) = \left(\Link_X\left(\osigma\right)\right)_{\ord},\]
Theorem \ref{theorem:largeordering} implies that $\FLink_{\bbX}(\sigma)$ is $(n-k-2)$-connected.
The lemma follows.
\end{proof}

\section{Background III: coefficient systems}
\label{section:coefficients}

In this section, we define coefficient systems on semisimplicial sets.  Informally, these are natural assignments
of abelian groups to each simplex.

\subsection{Simplex category}
To formalize this, we introduce the {\em simplex category} of a semisimplicial set $\bbX$, which is the following
category\footnote{This is related to the poset of simplices of $\bbX$, but is {\em not} always a poset since there can
be multiple morphisms between simplices in it.} $\Simp(\bbX)$:
\begin{itemize}
\item The objects of $\Simp(\bbX)$ are the simplices of $\bbX$.
\item For $\sigma,\sigma' \in \Simp(\bbX)$ with $\sigma \in \bbX^k$ and $\sigma' \in \bbX^{\ell}$, the morphisms from $\sigma$ to $\sigma'$ are
\[\Mor(\sigma,\sigma') = \Set{$\iota\colon [\ell] \rightarrow [k]$}{$\iota$ order-preserving injection with $\iota^{\ast}(\sigma) = \sigma'$}.\]
This is nonempty precisely when $\sigma'$ is a face of $\sigma$.
\end{itemize}
The {\em augmented simplex category} of $\bbX$, denoted $\tSimp(\bbX)$, is obtained by adjoining a terminal object
$\ast$ to $\Simp(\bbX)$ that we will call the {\em $(-1)$-simplex}.

\subsection{Coefficient systems}
Let $\bbk$ be a commutative ring.  A {\em coefficient system} over $\bbk$ on a semisimplicial set $\bbX$ is a covariant functor
$\cF$ from $\Simp(\bbX)$ to the category of $\bbk$-modules.  Unpacking this, $\cF$ consists of the following data:
\begin{itemize}
\item For each simplex $\sigma$ of $\bbX$, a $\bbk$-module $\cF(\sigma)$.
\item For $\sigma \in \bbX^k$ and $\sigma' \in \bbX^{\ell}$ and $\iota\colon [\ell] \rightarrow [k]$ an 
order-preserving injection with $\iota^{\ast}(\sigma) = \sigma'$,
a $\bbk$-module morphism $\iota^{\ast}\colon \cF(\sigma) \rightarrow \cF(\sigma')$.
\end{itemize}
These must satisfy the evident compatibility conditions.  Similarly, an {\em augmented coefficient system} on
$\bbX$ is a covariant functor $\cF$ from $\tSimp(\bbX)$ to the category of $\bbk$-modules.

\begin{example}
If $\bbX$ is a semisimplicial set and $\bbk$ is a commutative ring, then we have the constant
coefficient system $\ubbk$ on $\bbX$ with $\ubbk(\sigma) = \bbk$ for all simplices $\sigma$.  This can
be extended to an augmented coefficient system by setting $\ubbk(\ast) = \bbk$ for the $(-1)$-simplex $\ast$.
\end{example}

\begin{example}
\label{example:orderingsystem}
Recall that $\OSim_n$ is the large ordering of the $n$-simplex $\Sim_n$, and that we introduced
$\FI$-modules in \S \ref{section:fimodules}.
For an $\FI$-module $M$, we can define a coefficient system $\cF_{M,n}$ on $\OSim_n$ via the formula
\[\cF_{M,n}(i_0,\ldots,i_k) = M([n] \setminus \{i_0,\ldots,i_k\}) \quad \text{for a simplex $(i_0,\ldots,i_k)$ of $\OSim_n$}.\]
For an order-preserving injective map $\iota\colon [\ell] \rightarrow [k]$, the induced map 
\[\iota^{\ast}\colon \cF_{M,n}(i_0,\ldots,i_k) \rightarrow \cF_{M,n}(i_{\iota(0)},\ldots,i_{\iota(\ell)})\]
is the one induced by the inclusion
\[[n] \setminus \{i_0,\ldots,i_k\} \hookrightarrow [n] \setminus \{i_{\iota(0)},\ldots,i_{\iota(\ell)}\}.\]
This can be extended to an augmented coefficient system by setting
$\cF_{M,n}(\ast) = M([n])$ for the $(-1)$-simplex $\ast$.
\end{example}

The collection of coefficient systems (resp.\ augmented coefficient systems) over $\bbk$ on $\bbX$ forms an abelian category.

\subsection{Equivariant coefficient systems}
\label{section:equivariantcoefficient}

Let $G$ be a group, let $\bbX$ be a semisimplicial set on which $G$ acts, and let $\bbk$ be a commutative ring.  Consider
a (possibly augmented) coefficient system $\cF$ on $\bbX$.  We want to equip $\cF$ with an ``action'' of $G$ that
is compatible with the $G$-action on $\bbX$.  For $g \in G$ and a simplex $\sigma$ of $\bbX$, the data of such
an action should include isomorphisms $\cF(\sigma) \rightarrow \cF(g \cdot \sigma)$, and more generally for
$h \in G$ should include isomorphisms $\cF(h \cdot \sigma) \rightarrow \cF(gh \cdot \sigma)$.  Moreover, these
isomorphisms should be be compatible with $\cF$ in an appropriate sense.

We formalize this as follows.  For $h \in G$, let $\cF_h$ be the coefficient system over $\bbk$ on $\bbX$
defined via the formula
\[\cF_h(\sigma) = \cF(h \cdot \sigma) \quad \text{for a simplex $\sigma$ of $\bbX$}.\]
We say that $\cF$ is a {\em $G$-equivariant coefficient system} if for all $g,h \in G$ we are given
a natural transformation $\Phi_{g,h}\colon \cF_h \Rightarrow \cF_{gh}$.  These natural transformations
should satisfy the following two properties:
\begin{itemize}
\item For all $h \in G$, the natural transformation $\Phi_{1,h}\colon \cF_h \Rightarrow \cF_h$ should be
the identity natural transformation.
\item For all $g_1,g_2 \in G$ and all $h \in H$, we require the two natural transformations
\[\cF_h \xRightarrow{\Phi_{g_1 g_2,h}} \cF_{g_1 g_2 h} \quad \text{and} \quad
\cF_h \xRightarrow{\Phi_{g_2,h}} \cF_{g_2 h} \xRightarrow{\Phi_{g_1,g_2 h}} \cF_{g_1 g_2 h}\]
to be equal.
\end{itemize}

Let us unpack this a bit.  For all simplices $\sigma$ of $\bbX$ and all $g, h \in G$, the natural transformation 
$\Phi_{g,h}$ gives a homomorphism $\cF(h \cdot \sigma) \rightarrow \cF(gh \cdot \sigma)$.  This homomorphism
is an isomorphism whose inverse is given by the map $\cF(gh \cdot \sigma) \rightarrow \cF(h \cdot \sigma)$ induced
by $\Phi_{g^{-1},gh}$.  Moreover, it must be natural in the sense
that if $\sigma'$ is a face of $\sigma$ via some face map, then the diagram
\[\begin{CD}
\cF(h \cdot \sigma)  @>>> \cF(g h \cdot \sigma) \\
@VVV                        @VVV \\
\cF(h \cdot \sigma') @>>> \cF(g h \cdot \sigma')
\end{CD}\]
must commute.  The fact that the natural transformations respect the group law implies
that the stabilizer subgroup $G_{\sigma}$ acts on $\cF(\sigma)$, making it into a $\bbk[G_{\sigma}]$-module.
If $\sigma'$ is a face of $\sigma$ via some face map, the induced map $\cF(\sigma) \rightarrow \cF(\sigma')$
is a map of $\bbk[G_{\sigma}]$-modules, where $G_{\sigma}$ acts on $\cF(\sigma')$ via the inclusion
$G_{\sigma} \hookrightarrow G_{\sigma'}$.  

\begin{example}
\label{example:snequivariant}
Let $M$ be an $\FI$-module over $\bbk$ and let $\cF_{M,n}$ be the augmented coefficient system on $\OSim_n$
from Example \ref{example:orderingsystem} defined via the formula
\[\cF_{M,n}(i_0,\ldots,i_k) = M([n] \setminus \{i_0,\ldots,i_k\}) \quad \text{for a simplex $(i_0,\ldots,i_k)$ of $\OSim_n$.}\]
Recalling that $[n] = \{0,\ldots,n\}$, the symmetric group $\fS_{n+1}$ acts on $\OSim_n$.  The
augmented coefficient system $\cF_{M,n}$ can be endowed with the structure of an $\fS_{n+1}$-equivariant
augmented coefficient system in the following way.  Consider $g,h \in \fS_{n+1}$.  We then
define $\Phi_{g,h}$ to be the following natural transformation:
\begin{itemize}
\item For a simplex $(i_0,\ldots,i_k)$ of $\OSim_n$, let the induced map
\[\cF_{M,n}(h(i_0),\ldots,h(i_k)) \rightarrow \cF_{M,n}(g h(i_0),\ldots,g h(i_k))\]
be the map
\[M([n] \setminus \{h(i_0),\ldots,h(i_k)\}) \rightarrow M([n] \setminus \{g h(i_0),\ldots,g h(i_k)\}\]
induced by the bijection
\[[n] \setminus \{h(i_0),\ldots,h(i_k)\} \rightarrow [n] \setminus \{g h(i_0),\ldots,g h(i_k)\}\]
obtained by restricting $g \in \fS_{n+1}$ to $[n] \setminus \{h(i_0),\ldots,h(i_k)\}$.\qedhere
\end{itemize}
\end{example}

Another consequence of the fact that the natural
transformations respect the group law is that if $\cF$ is a $G$-equivariant coefficient system on $\bbX$, then for all $k \geq 0$ the direct sum
\[\bigoplus_{\sigma \in \bbX^k} \cF(\sigma)\]
is a $\bbk[G]$-module in a natural way, where the $G$-action restricts to the $G_{\sigma}$-action on
$\cF(\sigma)$ for each $\sigma \in \bbX^k$.  This $\bbk[G]$-module structure can be described
in terms of the quotient $\bbX/G$ as follows.  For each $\tau \in \bbX^k/G$, fix a lift
$\ttau \in \bbX^k$.  We then have
\[\bigoplus_{\sigma \in \bbX^k} \cF(\sigma) = \bigoplus_{\tau \in \bbX^k/G} \left(\bigoplus_{\sigma \in G \cdot \ttau} \cF(\sigma)\right).\]
Using \cite[Corollary III.5.4]{BrownCohomology}, we have a $\bbk[G]$-module isomorphism
\[\bigoplus_{\sigma \in G \cdot \ttau} \cF(\sigma) \cong \Ind_{G_{\ttau}}^G \cF(\ttau).\]
We conclude that
\begin{equation}
\label{eqn:decomposeinduced}
\bigoplus_{\sigma \in \bbX^k} \cF(\sigma) = \bigoplus_{\tau \in \bbX^k/G} \Ind_{G_{\ttau}}^G \cF(\ttau).
\end{equation}

\subsection{Chain complex and homology}
Let $\bbX$ be a semisimplicial set and let $\cF$ be a coefficient system on $\bbX$.  Define the
{\em simplicial chain complex} of $\bbX$ with coefficients in $\cF$ to be the chain complex
$\CC_{\bullet}(\bbX;\cF)$ defined as follows:
\begin{itemize}
\item For $k \geq 0$, we have
\[\CC_k(\bbX;\cF) = \bigoplus_{\sigma \in \bbX^k} \cF(\sigma).\]
\item The boundary map $d\colon \CC_k(\bbX;\cF) \rightarrow \CC_{k-1}(\bbX;\cF)$ is
$d = \sum_{i=0}^k (-1)^i d_i$, where the map $d_i\colon \CC_k(\bbX;\cF) \rightarrow \CC_{k-1}(\bbX;\cF)$ is
as follows.  Consider $\sigma \in \bbX^k$.  Let $\iota\colon [k-1] \rightarrow [k]$ be
the order-preserving map whose image omits $i$.  Then on the $\cF(\sigma)$ factor of $\CC_n(\bbX;\cF)$,
the map $d_i$ is
\[\cF(\sigma) \stackrel{\iota^{\ast}}{\longrightarrow} \cF(\iota^{\ast}(\sigma)) \hookrightarrow \bigoplus_{\sigma' \in \bbX^{k-1}} \cF(\sigma') = \CC_{k-1}(\bbX;\cF).\]
\end{itemize}
Define
\[\HH_k(\bbX;\cF) = \HH_k(\CC_{\bullet}(\bbX;\cF)).\]
For an augmented coefficient system $\cF$ on $\bbX$, define
$\RC_{\bullet}(\bbX;\cF)$ to be the augmented chain complex defined just like we did above but with
$\RC_{-1}(\bbX;\cF) = \cF(\ast)$ for the $(-1)$-simplex $\ast$ and define
\[\RH_k(\bbX;\cF) = \HH_k(\RC_{\bullet}(\bbX;\cF)).\]

\begin{example}
For a semisimplicial set $\bbX$ and a commutative ring $\bbk$, we have
\[\HH_k(\bbX;\ubbk) = \HH_k(|\bbX|;\bbk) \quad \text{and} \quad \RH_k(\bbX;\ubbk) = \RH_k(|\bbX|;\bbk).\qedhere.\]
\end{example}

\begin{remark}
With our definition, $\RH_{-1}(\bbX;\cF)$ is a quotient of $\cF(\ast)$.  This quotient can sometimes be
nonzero.  It vanishes precisely when the map
\[\bigoplus_{v \in \bbX^0} \cF(v) \rightarrow \cF(\ast)\]
is surjective.
\end{remark}

Note that if a group $G$ acts on $\bbX$ and $\cF$ is a $G$-equivariant coefficient system, then
\[\cdots \rightarrow \RC_{2}(\bbX;\cF) \rightarrow \RC_1(\bbX;\cF) \rightarrow \RC_0(\bbX;\cF) \rightarrow \RC_{-1}(\bbX;\cF) = \cF(\ast) \rightarrow 0\]
is a chain complex of $\bbk[G]$-modules, and each $\RH_k(\bbX;\cF)$ is a $\bbk[G]$-module.

\subsection{Long exact sequences}
Consider a short exact sequence
\[0 \longrightarrow \cF_1 \longrightarrow \cF_2 \longrightarrow \cF_3 \longrightarrow 0\]
of coefficient systems over $\bbk$ on $\bbX$.  For each simplex $\sigma$ of $\bbX$, we thus have a short
exact sequence
\[0 \longrightarrow \cF_1(\sigma) \longrightarrow \cF_2(\sigma) \longrightarrow \cF_3(\sigma) \longrightarrow 0\]
of $\bbk$-modules.  These fit together into a short exact sequence
\[0 \longrightarrow \CC_{\bullet}(\bbX;\cF_1) \longrightarrow \CC_{\bullet}(\bbX;\cF_2) \longrightarrow \CC_{\bullet}(\bbX;\cF_3) \longrightarrow 0\]
of chain complexes, and thus induce a long exact sequence in homology of the form
\[\cdots \longrightarrow \HH_k(\bbX;\cF_1) \longrightarrow \HH_k(\bbX;\cF_2) \longrightarrow \HH_k(\bbX;\cF_3) \longrightarrow \HH_{k-1}(\bbX;\cF_1) \longrightarrow \cdots.\]
A similar result holds for augmented coefficient systems and reduced homology.

\section{Stability I: stability machine}
\label{section:stabilitymachine}

In this section, we describe our machine for proving twisted homological stability.

\subsection{Classical homological stability}

An {\em increasing sequence of groups} is an indexed sequence of groups $\{G_n\}_{n=0}^{\infty}$ such that
\[G_0 \subseteq G_1 \subseteq G_2 \subseteq \cdots.\]
For each $k \geq 0$, we get maps
\[\HH_k(G_0) \rightarrow \HH_k(G_1) \rightarrow \HH_k(G_2) \rightarrow \cdots.\]
The classical homological stability machine gives conditions under which these stabilize, i.e., such that
the maps $\HH_k(G_{n-1}) \rightarrow \HH_k(G_n)$ are isomorphisms for $n \gg k$.  One version of it is as follows.

\begin{theorem}[Classical homological stability]
\label{theorem:classicalstabilitymachine}
Let $\{G_n\}_{n=0}^{\infty}$ be an increasing sequence of groups.
For each $n \geq 1$, let $\bbX_n$ be a semisimplicial set upon which $G_n$ acts.  Assume
for some $c \geq 2$ that the following hold for all $n \geq 1$:
\begin{itemize}
\item[(a)] For all $-1 \leq k \leq \left\lfloor\frac{n-2}{c}\right\rfloor$, we have $\RH_k(\bbX_n) = 0$.
\item[(b)] For all $0 \leq k < n$, the group $G_{n-k-1}$ is the $G_n$-stabilizer of a $k$-simplex of $\bbX_n$.
\item[(c)] For all $0 \leq k < n$, the group $G_n$ acts transitively on the $k$-simplices of $\bbX_n$.
\item[(d)] For all $n \geq 2$ and all $1$-simplices $e$ of $\bbX_n$ whose proper faces consist of $0$-simplices
$v$ and $v'$, there exists some $\lambda \in G_n$ with $\lambda(v) = v'$ such that
$\lambda$ commutes with all elements of $(G_n)_e$.
\end{itemize}
Then for all $k$ the map $\HH_k(G_{n-1}) \rightarrow \HH_k(G_n)$ is an isomorphism for
$n \geq ck+2$ and a surjection for $n = ck+1$.
\end{theorem}
\begin{proof}
This can be proved exactly like \cite[Theorem 1.1]{HatcherVogtmannTethers} -- the only major difference
between our theorem and \cite[Theorem 1.1]{HatcherVogtmannTethers} is that we assume a weaker connectivity
range on the $\bbX_n$, which causes stability to happen at a slower rate.\footnote{Another minor difference
is that we work with semisimplicial sets rather than simplicial complexes, so unlike \cite[Theorem 1.1]{HatcherVogtmannTethers}
we do not need to assume in (b) that the $G_n$-stabilizer of a simplex $\sigma$ of $\bbX_n$ stabilizes
$\sigma$ pointwise.  This is automatic for semisimplicial sets.}  We thus omit the details of the
proof, but to clarify our indexing conventions we make a few remarks.  

The proof is by induction on $k$.
The base case is $k \leq -1$, where the result is trivial since $\HH_{k}(G) = 0$ for all groups $G$ when $k$ is negative.  We could
also start with $k=0$ since $\HH_0(G) = \Z$, but later when we work with twisted coefficients even
the $\HH_0$ statement will be nontrivial.  In any case, to go from $\HH_{k-1}$ to $\HH_k$ two steps
are needed (which is why we require $c \geq 2$). 

The first step is to prove that the map $\HH_k(G_{n-1}) \rightarrow \HH_k(G_n)$ is surjective for
all $n$ sufficiently large, which requires that $\RH_i(\bbX_n) = 0$ for $-1 \leq i \leq k-1$.  Once this has been done, we then
prove that $\HH_k(G_{n-1}) \rightarrow \HH_k(G_{n})$ is also injective for $n$ one step larger than needed
for surjective stability.  This requires $\RH_i(\bbX_{n}) = 0$ for $-1 \leq i \leq k$.  We remark that condition (d) is used
for injective stability but not surjective stability, which is why we only assume it for $n \geq 2$.  See
Remark \ref{remark:conditiond} below for more discussion of condition (d).

This explains our indexing conventions:
\begin{itemize}
\item Surjective stability for $\HH_0$ starts with $\HH_0(G_0) \rightarrow \HH_0(G_1)$, so we need 
$\RH_{-1}(\bbX_1) = 0$.
\item Injective stability for $\HH_0$ starts with $\HH_0(G_1) \rightarrow \HH_0(G_2)$, so we need
$\RH_{-1}(\bbX_2) = \RH_{0}(\bbX_2) = 0$.\qedhere
\end{itemize}
\end{proof}

\subsection{Setup for twisted coefficients}
\label{section:twistedsetup}

We want to give a version of this with twisted coefficients.  Fix a commutative ring $\bbk$.
An {\em increasing sequence of groups and modules} is an indexed sequence of pairs $\{(G_n,M_n)\}_{n=0}^{\infty}$, where the
$G_n$ and the $M_n$ are as follows:
\begin{itemize}
\item The $\{G_n\}_{n=0}^{\infty}$ are an increasing sequence of groups.
\item Each $M_n$ is a $\bbk[G_n]$-module.
\item As abelian groups, the $M_n$ satisfy
\[M_0 \subseteq M_1 \subseteq M_2 \subseteq \cdots.\]
\item For each $n \geq 0$, the inclusion $M_n \hookrightarrow M_{n+1}$ is $G_n$-equivariant, where $G_n$
acts on $M_{n+1}$ via the inclusion $G_n \hookrightarrow G_{n+1}$.
\end{itemize}
Given an increasing sequence of groups and modules $\{(G_n,M_n)\}_{n=0}^{\infty}$, for each $k$ we get maps
\[\HH_k(G_0;M_0) \rightarrow \HH_k(G_1;M_1) \rightarrow \HH_k(G_2;M_2) \rightarrow \cdots\]
between the associated twisted homology groups.  We want to show that this stabilizes via a machine
similar to Theorem \ref{theorem:classicalstabilitymachine}.

We will incorporate the $M_n$ into our machine via a $G_n$-equivariant augmented coefficient system $\cM_n$
on the semisimplicial set $\bbX_n$ with $\cM_n(\ast) = M_n$ for the $(-1)$-simplex $\ast$.  Having done this, we will be forced to replace the
requirement that $\RH_k(\bbX_n) = 0$ in condition (a) by $\RH_k(\bbX_n;\cM_n) = 0$.  This
is not easy to check, but we will give a useful criterion in \S \ref{section:vanishingtheorem} below.
The twisted analogue of Theorem \ref{theorem:classicalstabilitymachine} is as follows.

\begin{theorem}[Twisted homological stability]
\label{theorem:stabilitymachine}
Let $\{(G_n,M_n)\}_{n=0}^{\infty}$ be an increasing sequence of groups and modules.
For each $n \geq 1$, let $\bbX_n$ be a semisimplicial set upon which $G_n$ acts and let
$\cM_n$ be a $G_n$-equivariant augmented coefficient system on $\bbX_n$.  Assume for some $c \geq 2$ that the following hold
for all $n \geq 1$:
\begin{itemize}
\item[(a)] For all $-1 \leq k \leq \left\lfloor\frac{n-2}{c}\right\rfloor$, we have $\RH_k(\bbX_n;\cM_n) = 0$.
\item[(b)] For all $-1 \leq k < n$, the group $G_{n-k-1}$ is the $G_n$-stabilizer of a $k$-simplex $\sigma_k$ of $\bbX_n$
with $\cM_n(\sigma_k) = M_{n-k-1}$.  In particular, $\cM_n(\ast) = M_n$.
\item[(c)] For all $0 \leq k < n$, the group $G_n$ acts transitively on the $k$-simplices of $\bbX_n$.
\item[(d)] For all $n \geq 2$ and all $1$-simplices $e$ of $\bbX_n$ whose proper faces consist of $0$-simplices
$v$ and $v'$, there exists some $\lambda \in G_n$ with $\lambda(v) = v'$ such that
$\lambda$ commutes with all elements of $(G_n)_e$ and fixes all elements of $\cM_n(e)$.
\end{itemize}
Then for $k \geq 0$ the map $\HH_k(G_{n-1};M_{n-1}) \rightarrow \HH_k(G_n;M_n)$ is an isomorphism for
$n \geq ck+2$ and a surjection for $n = ck+1$.
\end{theorem}

The proof of Theorem \ref{theorem:stabilitymachine} is almost identical to that of Theorem \ref{theorem:classicalstabilitymachine} (for
which we referred to \cite[Theorem 1.1]{HatcherVogtmannTethers}).  We will therefore not give full details, but only describe
how to construct the key spectral sequence (see Example \ref{example:standardexample} below), which will explain the role played by condition (a).

\begin{remark}
\label{remark:conditiond}
One other potentially confusing thing is how condition (d) is used, and in particular why we have to assume
that the $\lambda$ in it fixes all elements of $\cM_n(e)$.  As we said above, the corresponding condition (d)
in the classical homological stability machine (Theorem \ref{theorem:classicalstabilitymachine}) is used
to prove injective stability.  It arises in the following way.  Let the notation be as in Theorem \ref{theorem:classicalstabilitymachine}.  
Let $e$ be an edge of $\bbX_n$ connecting vertices $v$ and $v'$.
The stabilizer $(G_n)_e$ is thus a subgroup of both $(G_n)_v$ and $(G_n)_{v'}$.  
If $\phi \in G_n$ satisfies $\phi(v) = v'$, then conjugation by $\phi$ gives an isomorphism
$\iota_{\phi}\colon (G_n)_{v} \rightarrow (G_n)_{v'}$.  We therefore get two different maps
\[\HH_i((G_n)_e) \rightarrow \HH_i((G_n)_{v'}) \quad \text{and} \quad \HH_i((G_n)_e) \rightarrow \HH_i((G_n)_{v}) \stackrel{(\iota_{\phi})_{\ast}}{\rightarrow} \HH_i((G_n)_{v'}).\]
To prove injective stability, it turns out that you need to ensure that these maps are {\em equal}.  The point here
is that their difference equals a differential in a spectral sequence, and one needs to prove that this differential vanishes.

Any choice of $\phi \in G_n$ with $\phi(v) = v'$ will work here.  Condition (d) from Theorem \ref{theorem:classicalstabilitymachine}
gives a $\lambda \in G_n$ with $\lambda(v) = v'$ such that $\lambda$ also commutes with all elements of $(G_n)_e$.  This
implies that the maps
\[(G_n)_e \rightarrow (G_n)_{v'} \quad \text{and} \quad (G_n)_e \rightarrow (G_n)_{v} \stackrel{\iota_{\lambda}}{\rightarrow} (G_n)_{v'}\]
are actually equal to each other, and therefore induce the same map on $\HH_i$.  We can thus take $\phi = \lambda$.

In the setting of Theorem \ref{theorem:stabilitymachine}, the $G_n$-equivariance of the coefficient system $\cM_n$ implies
that a $\phi \in G_n$ with $\phi(v) = v'$ induces an isomorphism $\phi_{\ast}\colon \cM_n(v) \rightarrow \cM_n(v')$.  The
conjugation isomorphism $\iota_{\phi}\colon (G_n)_v \rightarrow (G_n)_{v'}$ and $\phi_{\ast}\colon \cM_n(v) \rightarrow \cM_n(v')$
induce an isomorphism
\[(\iota_{\phi},\phi_{\ast})_{\ast}\colon \HH_i((G_n)_v;\cM_n(v)) \rightarrow \HH_i((G_n)_{v'};\cM_n(v')),\]
so we get two maps
\[\HH_i((G_n)_e;\cM_n(e)) \rightarrow \HH_i((G_n)_{v'};\cM_n(v'))\]
and
\[\HH_i((G_n)_e;\cM_n(e)) \rightarrow \HH_i((G_n)_{v};\cM_n(v)) \stackrel{(\iota_{\phi}, \phi_{\ast})_{\ast}}{\rightarrow} \HH_i((G_n)_{v'};\cM_n(v')),\]
and again we need these to be equal.  This will be ensured by taking $\phi = \lambda$ for the $\lambda \in G_n$
provided by condition (d) in Theorem \ref{theorem:stabilitymachine}, which not only commutes with all elements
of $(G_n)_e$ but also fixes all elements of $\cM_n(e)$.
\end{remark}

\begin{remark}
Theorem \ref{theorem:stabilitymachine} is related to \cite[Theorem D]{PatztCentralStability}.
\end{remark}

\subsection{Homology of stabilizers}
\label{section:stabhomology}

To construct the spectral sequence, we need the following construction.  Fix a commutative ring $\bbk$.
Let $G$ be a group acting on a semisimplicial set $\bbX$ and let
$\cM$ be a $G$-equivariant augmented coefficient system on $\bbX$ over $\bbk$.  For a simplex $\tsigma$ of $\bbX$, the value $\cM(\tsigma)$ is a
$\bbk[G_{\tsigma}]$-module.  For $q \geq 0$, define an augmented coefficient system $\cH_q(\cM)$ on
$\bbX/G$ as follows.
Consider a simplex $\sigma$ of $\bbX/G$, and let $\tsigma$ be a lift of $\sigma$ to $\bbX$.  We then define
\[\cH_q(\cM)(\sigma) = \HH_q(G_{\tsigma};\cM(\tsigma)).\]
To see that this is well-defined, let $\tsigma'$ be another lift of $\sigma$ to $\bbX$.  There exists some $g \in G$ with
$g \tsigma = \tsigma'$, so $g G_{\tsigma} g^{-1} = G_{\tsigma'}$ and $g \cdot \cM(\tsigma) = \cM(\tsigma')$.  Conjugation/multiplication
by $g$ thus induces an isomorphism
\[\HH_q(G_{\tsigma};\cM(\tsigma)) \cong \HH_q(G_{\tsigma'};\cM(\tsigma')).\]
What is more, since inner automorphisms induce the identity on homology (even with twisted coefficients; see \cite[Proposition III.8.1]{BrownCohomology}),
this isomorphism is independent of the choice of $g$, and thus is completely canonical.  That it is a coefficient system follows
immediately.

\begin{remark}
If $\ast$ is the $(-1)$-simplex of $\bbX/G$, then the only possible lift $\tast$ is the $(-1)$-simplex of $\bbX$, and by
definition its stabilizer is the entire group $G$.  It follows that $\cH_q(\cM)(\ast) = \HH_q(G;\cM(\tast))$.
\end{remark}

\subsection{Spectral sequence}

The spectral sequence that underlies Theorem \ref{theorem:stabilitymachine} is as follows.  In it, our convention is that
$\cH_q(\cM) = \underline{0}$ for $q < 0$.

\begin{theorem}[Spectral sequence]
\label{theorem:spectralsequence}
Let $G$ be a group acting on a semisimplicial set $\bbX$, let $\bbk$ be a commutative ring, let
$\cM$ be a $G$-equivariant augmented coefficient system over $\bbk$ on $\bbX$.  Assume that $\RH_k(\bbX;\cM) = 0$ for
$0 \leq k \leq r$.  Then there exists a spectral sequence $E^{\bullet}_{pq}$ with the following
properties:
\begin{itemize}
\item[(i)] We have $E^1_{pq} = \RC_p(\bbX/G;\cH_q(\cM))$, and the differential
$E^1_{pq} \rightarrow E^1_{p-1,q}$ is the differential on $\RC_{\bullet}(\bbX/G;\cH_q(\cM))$.  In particular,
$E^1_{pq} = 0$ if $p<-1$ or if $q < 0$.
\item[(ii)] For $p + q \leq r$, we have $E^{\infty}_{pq} = 0$.
\end{itemize}
\end{theorem}

\begin{example}
\label{example:standardexample}
Let the notation and assumptions be as in Theorem \ref{theorem:stabilitymachine}, and apply Theorem
\ref{theorem:spectralsequence} to $G_n$ and $\bbX_n$ and $\cM_n$.  Since $G_n$ acts
transitively on the $k$-simplices of $\bbX_n$ for $0 \leq k < n$, there is a single
$k$-simplex in $\bbX_n/G_n$.  Since $G_{n-k-1}$ is the $G_n$ stabilizer of a $k$-simplex
$\sigma_k$ of $\bbX_n$ with $\cM_n(\sigma_k) = M_{n-k-1}$, we conclude that our spectral
sequence has
\[E^1_{pq} \cong \RC_p(\bbX_n/G_n;\cH_q(\cM_n)) = \HH_q(G_{n-p-1};M_{n-p-1}) \quad \text{for $-1 \leq p < n$}.\qedhere\]
\end{example}

\begin{proof}[Proof of Theorem \ref{theorem:spectralsequence}]
Let
\[\cdots \rightarrow F_2(G) \rightarrow F_1(G) \rightarrow F_0(G) \rightarrow \bbk\]
be a resolution of the trivial $\bbk[G]$-module $\bbk$ by free $\bbk[G]$-modules.  The action of $G$ on $\bbX$ makes $\RC_{\bullet}(\bbX;\cM)$ into
a chain complex of $\bbk[G]$-modules, so we can consider the double complex $C_{\bullet,\bullet}$ defined by
\begin{equation}
\label{eqn:doublecomplex}
C_{pq} = \RC_{p}(\bbX;\cM) \otimes_{G} F_{q}(G).
\end{equation}
The spectral sequence we are looking for converges to the homology of this double complex.  In fact,
there are two spectral sequences converging to the homology of a double complex, one arising by
filtering it ``horizontally'' and the other by filtering it ``vertically'' (see, e.g., \cite[\S VII.3]{BrownCohomology}).
We will use the horizontal filtration to show that the homology of \eqref{eqn:doublecomplex} vanishes
up to degree $r$ (conclusion (ii)), and then prove that the vertical filtration gives the spectral
sequence described in the theorem.

The spectral sequence arising from the horizontal filtration has
\[E^1_{pq} \cong \HH_p\left(\RC_{\bullet}\left(\bbX;\cM\right) \otimes_G F_q\left(G\right)\right).\]
Our assumptions imply that $\RC_{\bullet}(\bbX;\cM)$ is exact up to degree $r$, and since $F_q(G)$ is a free
$\Z[G]$-module it follows that $\RC_{\bullet}(\bbX;\cM) \otimes_G F_q(G)$ is also exact up to degree $r$.
It follows that $E^1_{pq} = 0$ for $p \leq r$.
We deduce that the homology of the double complex \eqref{eqn:doublecomplex} vanishes up to degree $r$.

The spectral sequence arising from the vertical filtration has
\[E^1_{pq} \cong \HH_q\left(\RC_{p}\left(\bbX;\cM\right) \otimes_G F_{\bullet}\left(G\right)\right).\]
This is the spectral sequence that is referred to in the theorem, and we must prove that it
satisfies (i).  Since $F_{\bullet}(G)$ is a free resolution of the trivial $\bbk[G]$-module
$\Z$, the above expression simplifies to
\[E^1_{pq} \cong \HH_q\left(G;\RC_{p}\left(\bbX;\cM\right)\right).\]
For each $p$-simplex $\sigma$ of $\bbX/G$, fix a lift $\tsigma$ to $\bbX$.  As we discussed in \S \ref{section:equivariantcoefficient} (see
\eqref{eqn:decomposeinduced}), as a $\bbk[G]$-module we have
\[\RC_p\left(\bbX;\cM\right) \cong \bigoplus_{\sigma \in \left(\bbX/G\right)^p} \Ind_{G_{\tsigma}}^G \cM\left(\tsigma\right).\]
Plugging this in, we get
\[E^1_{pq} \cong \bigoplus_{\sigma \in \left(\bbX/G\right)^p} \HH_q\left(G;\Ind_{G_{\tsigma}}^G \cM\left(\tsigma\right)\right).\]
Applying Shapiro's Lemma, the right hand side simplifies to
\[\bigoplus_{\sigma \in \left(\bbX/G\right)^p} \HH_q\left(G_{\tsigma};\cM\left(\tsigma\right)\right) = \RC_p(\bbX/G;\cH_q(\cM)).\]
That the differential is as described in (i) is clear.  The theorem follows.
\end{proof}

\subsection{Stability machine and finite-index subgroups}
We close this section by proving a variant of Theorem \ref{theorem:stabilitymachine} that we will use
to analyze congruence subgroups of $\GL_n(R)$.  To motivate its statement, consider a finite-index
normal subgroup $G'$ of a group $G$ and a $\bbk[G]$-module $M$ over a field $\bbk$ of characteristic $0$.
The conjugation action of $G$ on $G'$ induces an action of $G$ on each $\HH_k(G';M)$, and using
the transfer map (see \cite[Chapter III.9]{BrownCohomology}) one can show that $\HH_k(G;M)$ is
the $G$-coinvariants of the action of $G$ on $\HH_k(G';M)$.  From this, we see that the map
\[\HH_k(G';M) \rightarrow \HH_k(G;M)\]
induced by the inclusion $G' \hookrightarrow G$ is an isomorphism if and only if the action
of $G$ on $\HH_k(G';M)$ is trivial.  The point of the following is that under
the conditions of our stability machine, it is enough to check a weaker 
version of this triviality (condition (h)). 

\begin{theorem}
\label{theorem:stabilitymachinefi}
Let $\{(G_n,M_n)\}_{n=0}^{\infty}$ be an increasing sequence of groups and modules.
For each $n \geq 1$, let $\bbX_n$ be a semisimplicial set upon which $G_n$ acts and let
$\cM_n$ be a $G_n$-equivariant augmented coefficient system on $\bbX_n$.  Assume for some $c \geq 2$
that conditions (a)-(d) of Theorem \ref{theorem:stabilitymachine} hold, so by that theorem
$\HH_k(G_n;M_n)$ stabilizes.  Furthermore, assume that $\{G'_n\}_{n=0}^{\infty}$ is an
increasing sequence of groups such that each $G'_n$ is a finite-index normal subgroup
of $G_n$, and that the following hold:
\begin{itemize}
\item[(e)] Each $M_n$ is a vector space over a field $\bbk$ of characteristic $0$.
\item[(f)] For the $k$-simplex $\sigma_k$ of $\bbX_n$ from condition (b) whose
$G_n$-stabilizer is $G_{n-k-1}$, the $G'_n$-stabilizer of $\sigma_k$ is $G'_{n-k-1}$.
\item[(g)] The quotient $\bbX_n / G'_n$ is $\left\lfloor\frac{n-2}{c}\right\rfloor$-connected.
\item[(h)] For $k \geq 0$ and $n \geq ck+2$, the action of $G_n$ on $\HH_k(G'_n;M_n)$
induced by the conjugation action of $G_n$ on $G'_n$ fixes pointwise the image of the stabilization
map 
\begin{equation}
\label{eqn:gprimestab}
\HH_k(G'_{n-1};M_{n-1}) \rightarrow \HH_k(G'_n;M_n).
\end{equation}
\end{itemize}
Then for $n \geq ck+2$ the map $\HH_k(G'_n;M_n) \rightarrow \HH_k(G_n;M_n)$ induced
by the inclusion $G'_n \hookrightarrow G_n$ is an isomorphism.
\end{theorem}
\begin{proof}
The proof will be by induction on $k$.  The base case $k \leq -1$ is trivial, so assume that $k \geq 0$
and that the result is true for all smaller $k$.  Consider some $n \geq ck+2$.  As we discussed
before the theorem, because of (e) to prove that $\HH_k(G'_n;M_n) \cong \HH_k(G_n;M_n)$ it
is enough to prove that $G_n$ acts trivially on $\HH_k(G'_n;M_n)$.  Condition (h) implies
that to do this, it is enough to prove that $\HH_k(G'_n;M_n)$ is the span of the 
$G_n$-orbit of the image of the stabilization map \eqref{eqn:gprimestab}.

Condition (a) says that $\RH_i(\bbX_n;\cM_n) = 0$ for $-1 \leq i \leq \left\lfloor\frac{n-2}{c}\right\rfloor$.  Since
$n \geq ck+2$, this vanishing holds for $-1 \leq i \leq k$.  Applying Theorem \ref{theorem:spectralsequence},
we obtain a spectral sequence $E^{r}_{pq}$ with the following properties:
\begin{itemize}
\item[(i)] We have $E^1_{pq} = \RC_p(\bbX_n/G'_n;\cH_q(\cM_n))$, and the differential
$E^1_{pq} \rightarrow E^1_{p-1,q}$ is the differential on $\RC_{\bullet}(\bbX_n/G'_n;\cH_q(\cM_n))$.  In particular,
$E^1_{pq} = 0$ if $p<-1$ or if $q < 0$.
\item[(ii)] For $p + q \leq k$, we have $E^{\infty}_{pq} = 0$.
\end{itemize}
Unwinding the definition of $\cH_q(\cM_n)$, we see that for $\ast$ the $(-1)$-simplex of $\bbX_n$ we have
\[E^1_{-1,k} = \RC_{-1}(\bbX_n/G'_n;\cH_k(\cM_n)) = \HH_k((G'_n)_{\ast};\cM_n(\ast)) = \HH_k(G'_n;M_n).\]
For each simplex $\tau$ of $\bbX_n/G'_n$, fix a lift $\ttau$ to $\bbX_n$.  We then have
\begin{equation}
\label{eqn:e10k}
E^1_{0,k} = \RC_0(\bbX_n/G'_n;\cH_k(\cM_n)) = \bigoplus_{\tau \in (\bbX_n/G'_n)^0} \HH_k((G'_n)_{\ttau};\cM_n(\ttau)).
\end{equation}
By conditions (b) and (f), we have
\[\HH_k((G'_n)_{\sigma_0};\cM_n(\sigma_0)) = \HH_k(G'_{n-1};M_{n-1}).\]
Condition (c) says that $G_n$ acts transitively on the $0$-simplices of $\bbX_n$.  For each $\tau \in (\bbX_n/G'_n)^0$, there
thus exists some $g \in G_n$ such that $g \cdot \sigma_0 = \ttau$.  We thus have
\[g G'_{n-1} g^{-1} = (G'_n)_{\ttau} \quad \text{and} \quad g \cdot M_{n-1} = \cM_n(\ttau),\]
where the second equality uses the $G_n$-equivariant structure on $\cM_n$.

From these observations, we see that the image of the differential
\[E^1_{0,k} = \bigoplus_{\tau \in (\bbX_n/G'_n)^0} \HH_k((G'_n)_{\ttau};\cM_n(\ttau)) \rightarrow \HH_k(G'_n;M_n) = E^1_{-1,k}\]
is the span of the $G_n$-orbit of the image of the stabilization map
\[\HH_k(G'_{n-1};M_{n-1}) \rightarrow \HH_k(G'_n;M_n).\]
To see that $\HH_k(G'_n;M_n)$ is the span of the $G_n$-orbit of the image of this stabilization map,
it is therefore enough to prove that the differential $E^1_{0,k} \rightarrow E^1_{-1,k}$ is surjective.

The only nonzero differentials that can possibly hit $E^{r}_{-1,k}$ for some $r \geq 1$ are
\begin{align*}
E^1_{0,k}     &\rightarrow E^1_{-1,k}\\ 
E^2_{1,k-1}   &\rightarrow E^2_{-1,k}\\
              &\vdots \\
E^{k+1}_{k,0} &\rightarrow E^{k+1}_{-1,k}.
\end{align*}
By (ii) above we have $E^{\infty}_{-1,k} = 0$, so to prove that the differential $E^1_{0,k} \rightarrow E^1_{-1,k}$ is surjective
it is enough to prove that $E^{i+1}_{i,k-i} = 0$ for $1 \leq i \leq k$.  We will actually
prove the following more general result:

\begin{claim}
Fix some $0 \leq q \leq k-1$.  Then $E^2_{pq} = 0$ for $-1 \leq p \leq k-q$.
\end{claim}

The terms $E^1_{pq}$ with $-1 \leq p \leq k-q+1$ along
with the relevant $E^1$-differentials are
\[\RC_{-1}(\bbX_n/G'_n;\cH_q(\cM_n)) \leftarrow \RC_{0}(\bbX_n/G'_n;\cH_q(\cM_n)) \leftarrow \cdots \leftarrow \RC_{k-q+1}(\bbX_n/G'_n;\cH_q(\cM_n)).\]
To prove that the $E^2_{pq}$ with $-1 \leq p \leq k-q$ are all zero, we must prove that this chain complex is exact
except possibly at its rightmost term $\RC_{k-q+1}(\bbX_n/G'_n;\cH_q(\cM_n))$.

For a $p$-simplex $\tau$ of $\bbX_n/G'_n$, we have by definition
\[\cH_q(\cM_n)(\tau) = \HH_q((G'_n)_{\ttau};\cM_n(\ttau)).\]
Here $\ttau$ is our chosen lift of $\tau$ to $\bbX_n$.  Setting
$V = \HH_q(G_n;M_n)$, the inclusion $(G'_n)_{\ttau} \hookrightarrow G_n$
along with the map
\[\cM_n(\ttau) \rightarrow \cM_n(\ast) = M_n \quad \text{with $\ast$ the $(-1)$-simplex of $\bbX_n$}\]
coming from our coefficient system induce a map
\begin{equation}
\label{eqn:stabproj}
\cH_q(\cM_n)(\tau) = \HH_q((G'_n)_{\ttau};\cM_n(\ttau)) \rightarrow \HH_q(G_n;M_n) = V.
\end{equation}
Letting $\uV$ be the constant coefficient system on $\bbX_n/G'_n$ with value $V$, the maps
\eqref{eqn:stabproj} assemble to a map of coefficient systems $\cH_q(\cM_n) \rightarrow \uV$.

Letting $f_p\colon \RC_{p}(\bbX_n/G'_n;\cH_q(\cM_n)) \rightarrow \RC_{p}(\bbX_n/G'_n;\uV)$ be
the induced map on reduced chain complexes, we have a commutative diagram
\[\minCDarrowwidth10pt\begin{CD}
\RC_{-1}(\bbX_n/G'_n;\cH_q(\cM_n)) @<<< \RC_{0}(\bbX_n/G'_n;\cH_q(\cM_n)) @<<< \cdots @<<< \RC_{k-q+1}(\bbX_n/G'_n;\cH_q(\cM_n)) \\
@VV{f_{-1}}V                            @VV{f_0}V                              @.          @VV{f_{k-q+1}}V                      \\
\RC_{-1}(\bbX_n/G'_n;\uV)          @<<< \RC_{0}(\bbX_n/G'_n;\uV)          @<<< \cdots @<<< \RC_{k-q+1}(\bbX_n/G'_n;\uV).
\end{CD}\]
Condition (g) says that $\bbX_n / G'_n$ is $\left\lfloor\frac{n-2}{c}\right\rfloor$-connected.  Since $n \geq ck+2$, this means that
$\bbX_n / G'_n$ is $k$-connected.  Since $0 \leq q \leq k-1$, we deduce that the bottom chain complex
of this diagram is exact except possibly at its rightmost term.  To prove that the top
chain complex is exact except possibly at its rightmost term, it is thus enough to prove that
$f_p$ is an isomorphism for $-1 \leq p \leq k-q$ and a surjection for $p=k-q+1$.

Letting $\tau$ be a $p$-simplex of $\bbX_n/G'_n$, the map
$\cH_q(\cM_n)(\tau) \rightarrow \uV(\tau)$ is the map
\begin{equation}
\label{eqn:taustab}
\HH_q((G'_n)_{\ttau};\cM_n(\ttau)) \rightarrow \HH_q(G_n;M_n)
\end{equation}
We must prove that this is an isomorphism for $-1 \leq p \leq k-q$ and a surjection for $p=k-q+1$.
For $-1 \leq p \leq k-q+1$, the fact that $q \geq 0$ and $c \geq 2$ imply that
\[p \leq k-q+1 \leq k+1 \leq ck+2 \leq n.\]
Condition (c) thus says that for $-1 \leq p \leq k-q+1$, the group $G_n$ acts transitively on the $p$-simplices of $\bbX_n$.
In particular, if $p$ is in this range then the $p$-simplex $\ttau$ 
is in the same $G_n$-orbit as the $p$-simplex $\sigma_p$ from conditions
(b) and (f), which satisfies
\[\HH_q((G'_n)_{\sigma_p};\cM_n(\sigma_p)) = \HH_q(G'_{n-p-1};M_{n-p-1}).\]
Whether or not \eqref{eqn:taustab} is an isomorphism/surjection is invariant under the
$G_n$-action, so we deduce that it is enough to prove that
\[\HH_q(G'_{n-p-1};M_{n-p-1}) \rightarrow \HH_q(G_n;M_n)\]
is an isomorphism for $-1 \leq p \leq k-q$ and a surjection for $p=k-q+1$.
Factor this as
\[\HH_q(G'_{n-p-1};M_{n-p-1}) \rightarrow \HH_q(G_{n-p-1};M_{n-p-1}) \rightarrow \HH_q(G_n;M_n).\]
Condition (e) says that $M_{n-p-1}$ is a vector space over a field of characteristic $0$.
Since $G'_{n-p-1}$ is a finite-index subgroup of $G_{n-p-1}$, the transfer map (see \cite[Chapter III.9]{BrownCohomology})
therefore implies that the first map is always a surjection, and our inductive hypothesis says that it is an
isomorphism if $n-p-1 \geq cq+2$.  Also, Theorem \ref{theorem:stabilitymachine} says that the second map
is an isomorphism if $n-p \geq cq+2$ and a surjection if $n-p = cq+1$.  To prove the
theorem, it is thus enough to prove that $n-p-1 \geq cq+2$ if $-1 \leq p \leq k-q$ and
that $n-p \geq cq+1$ if $p = k-q+1$.
 
We check this as follows.  If $-1 \leq p \leq k-q$, then since $n \geq ck+2$ and $q \leq k-1$ and $c \geq 2$ we have
\begin{align*}
n-p-1 &\geq (ck+2) - (k-q) - 1 = (c-1)k+q+1 \\
      &\geq (c-1)(q+1)+q+1 = cq+c \\
      &\geq cq+2.
\end{align*}
If instead $p = k-q+1$, then we have
\begin{align*}
n-p &\geq (ck+2) - (k-q+1) = (c-1)k+q +1\\
    &\geq (c-1)(q+1) + q +1 = cq+c \\
    &\geq cq+1.\qedhere
\end{align*}
\end{proof}

\section{Stability II: the vanishing theorem}
\label{section:vanishingtheorem}

Let $\bbX$ be an ordered simplicial complex.\footnote{The results in this section work naturally
in this level of generality rather than for general semisimplicial sets.}  To
use Theorem \ref{theorem:stabilitymachine}, we need a way to prove that $\RH_k(\bbX;\cF) = 0$
in a range for an augmented coefficient system $\cF$.  This section contains a useful criterion for this
that applies in many situations.  Since we will be talking about augmented coefficient systems,
we will need to discuss the $(-1)$-simplex, which corresponds to the {\em empty} ordered tuple
of vertices, so we will write it as $\emptyset$.

\subsection{Polynomiality}
Our criterion applies to augmented coefficient systems $\cF$ that are {\em polynomial} of degree $d \geq -1$ up
to dimension $e \geq -1$.
This condition is inspired by the notion of polynomial $\FI$-modules (Definition \ref{definition:polyfi}).
It is defined inductively in $d$ as follows:
\begin{itemize}
\item A coefficient system $\cF$ is polynomial of degree $-1$ up to dimension $e$ if for all simplices\footnote{Including
the $(-1)$-simplex, so this has content even if $e=-1$.}
$\sigma$ of dimension at most $e$, we have $\cF(\sigma) = 0$.  In particular, $\cF(\emptyset)=0$.
\item A coefficient system $\cF$ is polynomial of degree $d \geq 0$ up to dimension $e$ if it satisfies
the following two conditions:
\begin{itemize}
\item If $\sigma$ is a simplex of dimension at most $e$, then the map $\cF(\sigma) \rightarrow \cF(\emptyset)$
is injective.\footnote{This implies that if $\sigma$ is a simplex of dimension at most $e$ and $\sigma'$ is a face
of $\sigma$, then the map $\cF(\sigma) \rightarrow \cF(\sigma')$ is injective.}
\item Let $w$ be a vertex of $\bbX$.  Let $D_w \cF$ be the coefficient
system on the forward link $\FLink_{\bbX}(w)$ defined by the formula
\[\quad\quad\quad\quad D_w \cF(\sigma) = \frac{\cF(\sigma)}{\Image\left(\cF\left(w \cdot \sigma\right) \rightarrow \cF\left(\sigma\right)\right)} \quad \text{for a simplex $\sigma$ of $\FLink_{\bbX}(w)$}.\]
Then $D_w \cF$ must be polynomial of degree $d-1$ up to dimension $e$.
\end{itemize}
\end{itemize}

\begin{remark}
\label{remark:degeneratepoly}
To give a sense as to what this definition means, consider the edge case where $\cF$ is polynomial of degree $d \geq 0$
up to dimension $(-1)$.  Unwinding the above definition, we see that it is equivalent to requiring that for all
$d$-simplices $(v_0,\ldots,v_d)$, the map
\[\bigoplus_{i=0}^d \cF(v_i) \rightarrow \cF(\emptyset)\]
is surjective.  As long as $\bbX$ is at least $d$-dimensional, this implies in particular that $\RH_{-1}(\bbX;\cF) = 0$.  Our vanishing theorem below
will generalize this.
\end{remark}

\subsection{Key example}
The following lemma provides motivation for this definition.

\begin{lemma}
\label{lemma:fisystempoly}
Let $M$ be an $\FI$-module that is polynomial of degree $d \geq -1$ starting at $m \geq 0$ (see Definition \ref{definition:polyfi}).
Fix\footnote{The purpose of the condition $n \geq m-1$ is simply to ensure that $n-m \geq -1$.  Here $n-m$ is the dimension we
are claiming that $\cF_{M,n}$ is polynomial of degree $d$ up to, and for this to make sense we need it to be at least $-1$.} some 
$n \geq \max(m-1,0)$, and let $\cF_{M,n}$ be the coefficient system on $\OSim_n$ discussed in Example \ref{example:orderingsystem}, so
\[\cF_{M,n}(i_0,\ldots,i_k) = M([n] \setminus \{i_0,\ldots,i_k\}) \quad \text{for all simplices $(i_0,\ldots,i_k)$ of $\OSim_n$}.\]
Then $\cF_{M,n}$ is polynomial of degree $d$ up to dimension $n-m$.
\end{lemma}
\begin{proof}
The proof will be by induction on $d$.
If $d=-1$, then consider a simplex $\sigma = (i_0,\ldots,i_k)$ with $k \leq n-m$.  We allow $k=-1$, in which case $\sigma = \emptyset$.  We then have
\[\cF_{M,n}(\sigma) = M([n] \setminus \{i_0,\ldots,i_k\}).\]
Since
\begin{equation}
\label{eqn:calccodim}
|[n] \setminus \{i_0,\ldots,i_k\}| = (n+1) - (k+1) = n-k \geq m,
\end{equation}
it follows from the fact that $M$ is polynomial of degree $-1$ starting at $m$ that $\cF_{M,n}(\sigma) = 0$, as desired.

Now assume that $d \geq 0$.  There are two things to check.  For the first, let $\sigma = (i_0,\ldots,i_k)$ be a simplex
with $k \leq n-m$.  We must prove that the map $\cF_{M,n}(\sigma) \rightarrow \cF_{M,n}(\emptyset)$ is injective, i.e., that
the map
\[M([n] \setminus \{i_0,\ldots,i_k\}) \rightarrow M([n])\]
is injective.  The calculation \eqref{eqn:calccodim} shows that this injectivity follows from the fact that
$M$ is polynomial of degree $d$ starting at $m$.

For the second, let $w$ be any vertex.  Let $D_w \cF_{M,n}$ be the coefficient
system on $\FLink_{\OSim_n}(w)$ defined by the formula
\[D_w \cF_{M,n}(\sigma) = \frac{\cF_{M,n}(\sigma)}{\Image\left(\cF_{M,n}\left(w \cdot \sigma\right) \rightarrow \cF_{M,n}\left(\sigma\right)\right)} \quad \text{for a simplex $\sigma$ of $\FLink_{\OSim_n}(w)$}.\]
We must prove that $D_w \cF_{M,n}$ is polynomial of degree $d-1$ up to dimension $n-m$.
We have $w \in [n] = \{0,\ldots,n\}$, and what we have to prove is invariant under the action
of the symmetric group on $[n]$.  Applying an appropriate element of this symmetric group,
we can assume without loss of generality that $w=n$, so $\FLink_{\OSim_n}(w) = \OSim_{n-1}$.

Recall that we defined the derived $\FI$-module $DM$ in
Definition \ref{definition:derivedfi}.  By construction, there is an isomorphism between
the coefficient systems $D_w \cF_{M,n}$ and $\cF_{DM,n-1}$ on $\OSim_{n-1}$.  Since $M$ is polynomial of
degree $d$ starting at $m$, the $\FI$-module $DM$ is
polynomial of degree $d-1$ starting at $m-1$.  By induction, $D_w \cF_{M,n}$ is polynomial of degree $d-1$ up to dimension
$(n-1)-(m-1)=n-m$, as desired.
\end{proof}

\subsection{Subcomplexes}
A subcomplex of an ordered simplicial complex is also an ordered simplicial complex.  Our definition
behaves well under restriction:

\begin{lemma}
\label{lemma:polynomialsubcomplex}
Let $\bbX$ be an ordered simplicial complex and let $\cF$ be an
augmented coefficient system on $\bbX$ that is polynomial of degree $d$ up to dimension $e$.  Let $\bbX'$
be a subcomplex of $\bbX$ and let $\cF'$ be the restriction of $\cF$ to $\bbX'$.  Then
$\cF'$ is polynomial of degree $d$ up to dimension $e$.
\end{lemma}
\begin{proof}
The proof is by induction on $d$.  The base case $d=-1$ is trivial: if $\sigma$ is a simplex of
$\bbX'$ of dimension at most $e$, then $\cF'(\sigma) = \cF(\sigma) = 0$.  Assume now that $d \geq 0$
and that the result is true for all smaller degrees.  We must check two conditions.

For the first, consider a simplex $\sigma$ of $\bbX'$ of dimension at most $e$.  We must prove that the map
$\cF'(\sigma) \rightarrow \cF'(\emptyset)$
is injective.  This is immediate from the corresponding property for the map $\cF(\sigma) \rightarrow \cF(\emptyset)$.

For the second, let $w$ be a vertex of $\bbX'$.  Let $D_w \cF'$ be the coefficient
system on $\FLink_{\bbX'}(w)$ defined by the formula
\[D_w \cF'(\sigma) = \frac{\cF'(\sigma)}{\Image\left(\cF'\left(w \cdot \sigma\right) \rightarrow \cF'\left(\sigma\right)\right)} \quad \text{for a simplex $\sigma$ of $\FLink_{\bbX'}(w)$}.\]
We must prove that $D_w \cF'$ is polynomial of degree $d-1$ up to dimension $e$.
Let $D_w \cF$ be the coefficient system on $\FLink_{\bbX}(w)$ defined by the formula
\[D_w \cF(\sigma) = \frac{\cF(\sigma)}{\Image\left(\cF\left(w \cdot \sigma\right) \rightarrow \cF\left(\sigma\right)\right)} \quad \text{for a simplex $\sigma$ of $\FLink_{\bbX}(w)$}.\]
By definition, $D_w \cF$ is polynomial of degree $d-1$ up to dimension $e$.  Since $\FLink_{\bbX'}(w)$ is a 
subcomplex of $\FLink_{\bbX}(w)$ and $D_w \cF'$ is the restriction
of $D_w \cF$ to $\FLink_{\bbX'}(w)$, the desired result follows from our inductive hypothesis.
\end{proof}

\subsection{Vanishing theorem}
Our vanishing theorem is as follows:\footnote{To connect this to Remark \ref{remark:degeneratepoly}, note that this
remark deals with the case $N=-1$.  In that case, Theorem \ref{theorem:vanishing} requires $\bbX$ to be
weakly forward Cohen--Macaulay of dimension $d$, and asserts that then $\RH_{-1}(\bbX;\cF) = 0$ if $\cF$ is polynomial
of degree $d$ up to dimension $(-1)$.  Using our convention that a space is $(-1)$-connected precisely when it is nonempty, the
assumption that $\bbX$ is weakly forward Cohen--Macaulay of dimension $d$ implies that $\bbX$ is at least $d$-dimensional (c.f. Remark \ref{remark:weakmeaning}),
as required in Remark \ref{remark:degeneratepoly} for $\RH_{-1}(\bbX;\cF) = 0$.}

\begin{theorem}
\label{theorem:vanishing}
For some $N \geq -1$ and $d \geq 0$, let $\bbX$ be an ordered simplicial complex that is weakly forward Cohen--Macaulay of dimension $N+d+1$
and let $\cF$ be an augmented coefficient system on $\bbX$
that is polynomial of degree $d$ up to dimension $N$.  Then $\RH_k(\bbX;\cF) = 0$ for $-1 \leq k \leq N$.
\end{theorem}
\begin{proof}
The proof will be by induction on $d$.  For the base case $d=-1$, by definition
the coefficient system $\cF$ equals the constant coefficient system $\underline{0}$ on the $N$-skeleton
of $\bbX$.  This trivially implies that for $-1 \leq k \leq N$ we have $\RH_k(\bbX;\cF) = 0$.

Assume now that $d \geq 0$ and that the theorem is true for smaller $d$.  We divide the rest of the
proof into four steps.  The descriptions of the steps describe the objects that are introduced
during that step and what the step reduces the theorem to, with the fourth step proving the result.
The inductive hypothesis and the polynomiality assumption are only invoked in the fourth step.

\begin{stepsa}
For $n \geq -1$, we introduce the shifted coefficient systems $\cF_n$ and reduce the theorem to proving that
$\RH_k(\bbX;\cF_{n-1}) \cong \RH_k(\bbX;\cF_{n})$ for $-1 \leq k \leq N$ and $0 \leq n \leq N+1$.
\end{stepsa}

For a simplex $\sigma = (v_0,\ldots,v_k)$ of $\bbX$ and $n \geq -1$, define
\[\sigma_{>n} = \begin{cases}
(v_{n+1},\ldots,v_k) & \text{if $k>n$} \\
\emptyset            & \text{if $k \leq n$}.
\end{cases}\]
Observe that if $\sigma'$ is a face of $\sigma$, then $\sigma'_{>n}$ is a face of $\sigma_{>n}$.  We can thus
define an augmented coefficient system $\cF_n$ on $\bbX$ via the formula
\[\cF_n(\sigma) = \cF(\sigma_{>n}) \quad \text{for a simplex $\sigma$ of $\bbX$.}\]
The maps $\cF_n(\sigma) \rightarrow \cF_n(\sigma')$ when $\sigma'$ is a face of $\sigma$ are the ones induced
by $\cF$.  The augmented coefficient system $\cF_{N+1}$ equals the constant coefficient system $\underline{\cF(\emptyset)}$
on all simplices of dimension at most $(N+1)$.  The fact that $\bbX$ is weakly forward Cohen--Macaulay
of dimension $(N+d+1)$ implies that $\bbX$ is $N+d \geq N$ connected and hence
\[\RH_k(\bbX;\cF_{N+1}) = \RH_k(\bbX;\underline{\cF(\emptyset)}) = 0 \quad \text{for $-1 \leq k \leq N$}.\]
We want to prove a similar equality with $\cF_{N+1}$ replaced with $\cF$.
We can fit $\cF$ and $\cF_{N+1}$ into a chain of maps
\[\cF = \cF_{-1} \rightarrow \cF_0 \rightarrow \cF_1 \rightarrow \cdots \rightarrow \cF_{N+1}.\]
To prove the theorem, it is enough to prove that the map $\cF_{n-1} \rightarrow \cF_n$ induces
an isomorphism $\RH_k(\bbX;\cF_{n-1}) \cong \RH_k(\bbX;\cF_n)$ for $0 \leq n \leq N+1$
and $-1 \leq k \leq N$.

\begin{stepsa}
Fix some $0 \leq n \leq N+1$.  We introduce the quotiented shifted coefficient systems $\ocF_{n}$ and reduce the theorem
to proving that $\RH_k(\bbX;\ocF_{n}) = 0$ for $n \leq k \leq N+1$.
\end{stepsa}

Define
\begin{align*}
\ocF_n     &= \coker(\cF_{n-1} \rightarrow \cF_n), \\
\cF'_{n-1}     &= \Image(\cF_{n-1} \rightarrow \cF_n),\\
\cF''_{n-1} &= \ker(\cF_{n-1} \rightarrow \cF_n).
\end{align*}
We thus have short exact sequences of coefficient systems
\begin{equation}
\label{eqn:ses1}
0 \longrightarrow \cF'_{n-1} \longrightarrow \cF_n \longrightarrow \ocF_n \longrightarrow 0
\end{equation}
and
\begin{equation}
\label{eqn:ses2}
0 \longrightarrow \cF''_{n-1} \longrightarrow \cF_{n-1} \longrightarrow \cF'_{n-1} \longrightarrow 0.
\end{equation}
Both of these induce long exact sequences in homology.  Let us focus first on the one associated to
\eqref{eqn:ses2}, which contains segments of the form
\begin{equation}
\label{eqn:les2}
\RH_k(\bbX;\cF''_{n-1}) \longrightarrow \RH_k(\bbX;\cF_{n-1}) \longrightarrow \RH_k(\bbX;\cF'_{n-1}) \longrightarrow \RH_{k-1}(\bbX;\cF''_{n-1}).
\end{equation}
Since $\cF$ is polynomial of degree $d$ up to
dimension $N$, for all simplices $\sigma$ of dimension at most $N$ and all faces $\sigma'$ of $\sigma$
the map $\cF(\sigma) \rightarrow \cF(\sigma')$ is injective.  This implies that the map
$\cF_{n-1}(\sigma) \rightarrow \cF_n(\sigma)$ is injective as long as $\sigma$ has dimension
at most $N$, and thus that $\cF''_{n-1}(\sigma) = 0$.  It follows that $\RH_k(\bbX;\cF''_{n-1}) = 0$ for $k \leq N$.
Combining this with \eqref{eqn:les2}, we see that
\begin{equation}
\label{eqn:les2conclude}
\RH_k(\bbX;\cF_{n-1}) \cong \RH_k(\bbX;\cF'_{n-1}) \quad \text{for $-1 \leq k \leq N$}.
\end{equation}

We now turn to the long exact sequence associated to \eqref{eqn:ses1}, which contains the segment
\[\RH_{k+1}(\bbX;\ocF_{n}) \longrightarrow \RH_k(\bbX;\cF'_{n-1}) \longrightarrow \RH_k(\bbX;\cF_n) \longrightarrow \RH_k(\bbX;\ocF_{n}).\]
In light of \eqref{eqn:les2conclude}, this implies that to
prove that the map $\cF_{n-1} \rightarrow \cF_n$ induces an isomorphism $\RH_k(\bbX;\cF_{n-1}) \cong \RH_k(\bbX;\cF_n)$
for $-1 \leq k \leq N$, it is enough to prove that
\[\RH_k(\bbX;\ocF_{n}) = 0 \quad \text{for $-1 \leq k \leq N+1$}.\]
This can be simplified a little further: if $\sigma$ is a $k$-simplex with $k \leq n-1$, then
$\cF_{n-1}(\sigma) = \cF_n(\sigma) = \cF(\emptyset)$, so $\ocF_{n}(\sigma) = 0$.  The group
$\RH_k(\bbX;\ocF_{n})$ is thus automatically $0$ for $k \leq n-1$, so it is actually
enough to prove that
\[\RH_k(\bbX;\ocF_{n}) = 0 \quad \text{for $n \leq k \leq N+1$}.\]

\begin{stepsa}
Fix some $0 \leq n \leq N+1$.  For each $n$-simplex $\tau$, we construct a coefficient system $\ocF_{n,\tau}$ and prove that
$\ocF_{n}$ is the direct sum of the $\ocF_{n,\tau}$ as $\tau$ ranges over the $n$-simplices, reducing
the theorem to proving that $\RH_k(\bbX;\ocF_{n,\tau}) = 0$ for $n \leq k \leq N+1$.
\end{stepsa}

Consider a $k$-simplex $\sigma$ of $\bbX$ with $k \geq n$.  We claim that the following holds:
\begin{itemize}
\item Let $\sigma'$ be a face of $\sigma$.  Then the map $\ocF_n(\sigma) \rightarrow \ocF_n(\sigma')$ is the
zero map unless $\sigma'$ also has dimension at least $n$ and the first $(n+1)$ vertices of $\sigma$ and $\sigma'$ are equal.
\end{itemize}
To prove this, it is enough to prove the following special case:
\begin{itemize}
\item Let $(w_0,\ldots,w_n,v_0,\ldots,v_m)$ be a simplex of $\bbX$ and let $0 \leq i \leq n$.  Then
the map $\ocF_n(w_0,\ldots,w_n,v_0,\ldots,v_m) \rightarrow \ocF_n(w_0,\ldots,\widehat{w_i},\ldots,w_n,v_0,\ldots,v_m)$
is the zero map.
\end{itemize}
For this, observe that
\[\ocF_n(w_0,\ldots,w_n,v_0,\ldots,v_m) = \frac{\cF(v_0,\ldots,v_m)}{\Image\left(\cF\left(w_n,v_0,\ldots,v_m\right) \rightarrow \cF\left(v_0,\ldots,v_m\right)\right)}\]
and
\[\ocF_n(w_0,\ldots,\widehat{w_i},\ldots,w_n,v_0,\ldots,v_m) = \frac{\cF(v_1,\ldots,v_m)}{\Image\left(\cF\left(v_0,\ldots,v_m\right) \rightarrow \cF\left(v_1,\ldots,v_m\right)\right)}.\]
The map between these is visibly the zero map.

For each $n$-simplex $\tau$ of $\bbX$, this suggests defining a coefficient system $\ocF_{n,\tau}$ on $\bbX$ via the formula
\[\ocF_{n,\tau}(\sigma) = \begin{cases}
\ocF_{n}(\sigma) & \text{if $\sigma$ starts with $\tau$}\\
0 & \text{otherwise}
\end{cases}
\quad \quad \text{for a simplex $\sigma$ of $\bbX$}.\]
By the above, we have a decomposition
\[\ocF_{n} = \bigoplus_{\tau \in \bbX^n} \ocF_{n,\tau}\]
of coefficient systems.  To prove that
\[\RH_k(\bbX;\ocF_{n}) = 0 \quad \text{for $n \leq k \leq N+1$},\]
it is thus enough to prove that for all $n$-simplices $\tau$ we have
\[\RH_k(\bbX;\ocF_{n,\tau}) = 0 \quad \text{for $n \leq k \leq N+1$}.\]

\begin{stepsa}
Fix an $n$-simplex $\tau$ with $0 \leq n \leq N+1$.  We prove that
$\RH_k(\bbX;\ocF_{n,\tau}) = 0$ for $n \leq k \leq N+1$.
\end{stepsa}

Define an augmented coefficient system $\cH$ on $\FLink_{\bbX}(\tau)$ via the formula
\[\cH(\sigma) = \ocF_{n,\tau}(\tau \cdot \sigma) \quad \text{for a simplex $\sigma$ of $\FLink_{\bbX}(\tau)$}.\]
On the level of reduced chain complexes, up to multiplying the differentials by $(-1)^{n+1}$ we have
\[\RC_{\bullet}(\FLink_{\bbX}(\tau);\cH) \cong \RC_{\bullet+n+1}(\bbX;\ocF_{n,\tau}),\]
so $\RH_{k}(\FLink_{\bbX}(\tau);\cH) \cong \RH_{k+n+1}(\bbX;\ocF_{n,\tau})$.  It is thus enough to prove that
$\RH_k(\FLink_{\bbX}(\tau);\cH) = 0$ for $-1 \leq k \leq N-n$.

Write $\tau = (w_0,\ldots,w_n)$.  For a simplex $\sigma$ of $\FLink_{\bbX}(\tau)$, we have
\[\cH(\sigma) = \ocF_{n,\tau}((w_0,\ldots,w_n) \cdot \sigma) = \frac{\cF(\sigma)}{\Image\left(\cF\left(w_n \cdot \sigma\right) \rightarrow \cF\left(\sigma\right)\right)}.\]
Let $D_{w_n} \cF$ be the coefficient system on $\FLink_{\bbX}(w_n)$ defined via the formula
\[D_{w_n} \cF(\sigma) = \frac{\cF(\sigma)}{\Image\left(\cF\left(w_n \cdot \sigma\right) \rightarrow \cF\left(\sigma\right)\right)} \quad \text{for a simplex $\sigma$ of $\FLink_{\bbX}(w_n)$}.\]
By definition, $D_{w_n} \cF$ is polynomial of degree $(d-1)$ up to dimension $N$.  Since $\cH$ is the restriction of
$D_{w_n} \cF$ to $\FLink_{\bbX}(\tau)$, Lemma \ref{lemma:polynomialsubcomplex} implies that $\cH$ is also polynomial of degree $(d-1)$
up to dimension $N$.  In fact, all we will need is the weaker statement that it is polynomial of degree $(d-1)$
up to dimension $N-n$.

Since $\bbX$ is weakly forward Cohen--Macaulay of dimension $N+d+1$, the
forward link $\FLink_{\bbX}(\tau)$ of the $n$-simplex $\tau$ is weakly forward Cohen--Macaulay of dimension\footnote{This uses the
observation that if $\tau$ is a simplex of $\bbX$ and $\tau'$ is a simplex of $\FLink_{\bbX}(\tau)$, then
$\FLink_{\FLink_{\bbX}(\tau)}(\tau') = \FLink_{\bbX}(\tau \cdot \tau')$.}
$N+d-n$.
Since
\[N+d-n=(N-n)+(d-1)+1,\]
we can apply our inductive hypothesis to conclude that $\RH_k(\FLink_{\bbX}(\tau);\cH) = 0$ for $-1 \leq k \leq N-n$, as desired.
\end{proof}

\section{Stability for symmetric groups}
\label{section:sn}

We now turn to applications of our machinery, starting with Theorems \ref{maintheorem:sn}
and \ref{maintheorem:snprime}.

\begin{proof}[Proof of Theorem \ref{maintheorem:sn}]
We first recall the statement.  Let $\bbk$ be a commutative ring and let $M$ be an $\FI$-module over $\bbk$ that is polynomial
of degree $d \geq -1$ starting at $m \geq 0$.  For each $k \geq 0$, we must prove that the map
\[\HH_k(\fS_{n};M(\overline{n})) \rightarrow \HH_k(\fS_{n+1};M(\overline{n+1}))\]
is an isomorphism for $n \geq 2k+\max(d,m)+1$ and a surjection for $n=2k+\max(d,m)$.  This is trivial
for $d = -1$ since in that case these homology groups are $0$ in the indicated range, so we can assume that $d \geq 0$.

Recall that $\overline{n} = \{1,\ldots,n\}$ and $[n-1] = \{0,\ldots,n-1\}$.  The group $\fS_n$ acts
on both $M(\overline{n})$ and $M([n-1])$, and there is a $\bbk[\fS_n]$-module isomorphism
$M(\overline{n}) \cong M([n-1])$.  In light of this, it is enough to deal with the map
\[\HH_k(\fS_{n};M([n-1])) \rightarrow \HH_k(\fS_{n+1};M([n])),\]
which will fit into our framework a little better.

The group $\fS_n$ acts on $\OSim_{n-1}$.  Let $\cF_{M,n-1}$ be
the $\fS_n$-equivariant augmented system of coefficients on $\OSim_{n-1}$ from Example \ref{example:snequivariant}, so
\[\cF_{M,n-1}(i_0,\ldots,i_k) = M([n-1] \setminus \{i_0,\ldots,i_k\}) \quad \text{for a simplex $(i_0,\ldots,i_k)$ of $\OSim_{n-1}$}.\]
The following claim will be used to show that with an appropriate degree shift, this all satisfies
the hypotheses of Theorem \ref{theorem:stabilitymachine}.

\begin{claim}
Assume that $n \geq \max(m,1)$, so we can apply Lemma \ref{lemma:fisystempoly} to the coefficient
system $\cF_{M,n-1}$ on $\OSim_{n-1}$.  The following then hold:
\begin{enumerate}
\item[(A)] For all $-1 \leq k \leq n-\max(d,m-1)-2$, we have $\RH_k(\OSim_{n-1};\cF_{M,n-1}) = 0$.
\item[(B)] For all $-1 \leq k \leq n-1$, the group $\fS_{n-k-1}$ is the $\fS_n$-stabilizer of a $k$-simplex
$\sigma_k$ of $\OSim_{n-1}$ with $\cF_{M,n-1}(\sigma_k) = M([n-k-2])$.
\item[(C)] For all $0 \leq k \leq n-1$, the group $\fS_n$ acts transitively on the $k$-simplices of $\OSim_{n-1}$.
\item[(D)] For all $n \geq 2$ and all $1$-simplices $e$ of $\OSim_{n-1}$ of the form $e = (i_0,i_1)$, there
exists some $\lambda \in \fS_n$ with $\lambda(i_0) = i_1$ such that $\lambda$ commutes with all
elements of $(\fS_n)_e$ and fixes all elements of $\cF_{M,n-1}(e)$.
\end{enumerate}
\end{claim}
\begin{proof}[Proof of claim]
For (A), Lemma \ref{lemma:fisystempoly} says that $\cF_{M,n-1}$ is a polynomial coefficient system of
degree $d$ up to dimension $n-m-1$.  Also, $\OSim_{n-1}$ is the large ordering of the $(n-1)$-simplex
$\Sim_{n-1}$ and $\Sim_{n-1}$ is weakly Cohen--Macaulay of dimension $(n-1)$, so by Lemma \ref{lemma:largerorderingcm}
the ordered simplicial complex $\OSim_{n-1}$ is weakly forward Cohen--Macaulay of dimension $(n-1)$.  Letting
\[N = \min(n-m-1,n-d-2) = n - \max(d,m-1)-2,\]
the above shows that $\Sim_{n-1}$ is weakly forward Cohen--Macaulay of dimension $N+d+1$ and
$\cF_{M,n-1}$ is a polynomial coefficient system of degree $d$ up to dimension $N$.  Theorem \ref{theorem:vanishing} thus
implies that $\RH_k(\OSim_{n-1};\cF_{M,n-1}) = 0$ for $-1 \leq k \leq N$.

For (B), the group $\fS_{n-k-1}$ is the $\fS_n$-stabilizer of the $k$-simplex 
\[\sigma_k = \begin{cases}
(\emptyset) & \text{if $k=-1$},\\
(n-k-1,n-k+1,\ldots,n-1) & \text{if $0 \leq k \leq n-1$}
\end{cases}\]
of $\OSim_{n-1}$, and by definition
\[\cF_{M,n-1}(\sigma_k) = M([n-1] \setminus \{n-k-1,n-k+1,\ldots,n-1\}) = M([n-k-2]).\]

Condition (C) is obvious.

For (D), simply let $\lambda$ be the transposition $(i_0,i_1)$, which acts trivially on
\[\cF_{M,n-1}(e) = M([n-1] \setminus \{i_0,i_1\})\]
by basic properties of $\FI$-modules.
\end{proof}

Let $e = \max(d,m)$, so since $m \geq 0$ we have $e \geq 0$.  For $n \geq 0$ define
\[G_n = \fS_{n+e} \quad \text{and} \quad M_n = M([n+e-1])\]
and for $n \geq 1$ define
\[\bbX_n = \OSim_{n+e-1} \quad \text{and} \quad \cM_n = \cF_{M,n+e-1}.\]
This makes sense since $n+e \geq n$.\footnote{Note that if $n+e = 0$ then
$[n+e-1] = [-1] = \emptyset$, which can be plugged into the $\FI$-module $M$.}  We then have
\[G_0 \subset G_1 \subset G_2 \subset \cdots,\]
and since $M$ is a polynomial $\FI$-module of degree $d$ starting at $m \geq 0$ the maps
$M_n \rightarrow M_{n+1}$ are injective for $n \geq 0$,\footnote{Remember, $[n+e-1]$ has
$n+e$ elements, and for $n \geq 0$ we have $n+e \geq e = \max(d,m) \geq m$.} so we have an increasing
sequence
\[M_0 \subset M_1 \subset M_2 \subset \cdots.\]
In other words, the pair $\{(G_n,M_n)\}_{n=0}^{\infty}$ is an increasing sequence
of groups and modules in the sense of \S \ref{section:twistedsetup}.  

The above claim verifies the conditions of Theorem \ref{theorem:stabilitymachine} with $c=2$.  
The shift by $e$ is needed for condition (a) of Theorem \ref{theorem:stabilitymachine}, which
requires that $\RH_k(\bbX_n;\cM_n) = 0$ for all $n \geq 1$ and $-1 \leq k \leq \left\lfloor\frac{n-2}{2}\right\rfloor$.  Conclusion (A)
of the Claim says that $\RH_k(\OSim_{n+e-1};\cF_{M,n+e-1}) = 0$ for $n+e \geq \max(m,1)$\footnote{Which holds 
for $n \geq 1$ since then $n+e = n+\max(d,m) \geq 1+m \geq \max(m,1)$.  Here we are using the fact that
$m \geq 0$.} and $-1 \leq k \leq (n+e)-\max(d,m-1)-2$, which implies
the desired range of vanishing for $\RH_k(\bbX_n;\cM_n)$ since
\[(n+e)-\max(d,m-1)-2 = n+\max(d,m) - \max(d,m-1) - 2 \geq n-2 \geq \left\lfloor\frac{n-2}{2}\right\rfloor\]
for $n \geq 1$.  Applying Theorem \ref{theorem:stabilitymachine}, we deduce that the map
\[\HH_k(\fS_{n+e-1};M([n+e-2])) \rightarrow \HH_k(\fS_{n+e};M([n+e-1]))\]
is an isomorphism for $n \geq 2k+2$ and a surjection for $n = 2k+1$, which implies
that
\[\HH_k(\fS_{n};M([n-1])) \rightarrow \HH_k(\fS_{n+1};M([n]))\]
is an isomorphism for $n \geq 2k+e+1$ and a surjection for $n = 2k+e$.
\end{proof}

\begin{proof}[Proof of Theorem \ref{maintheorem:snprime}]
We first recall the statement.  Let $\bbk$ be a commutative ring and let $M$ be an $\FI$-module over $\bbk$ that is polynomial
of degree $d \geq -1$ starting at $m \geq 0$.  For each $k \geq 0$, we must prove that the map
\begin{equation}
\label{eqn:snprimetoprove}
\HH_k(\fS_{n};M(\overline{n})) \rightarrow \HH_k(\fS_{n+1};M(\overline{n+1}))
\end{equation}
is an isomorphism for $n \geq \max(m,2k+2d+2)$ and a surjection for $n \geq \max(m,2k+2d)$.

The proof will be by double induction on $d$ and $m$.  There are two base cases:
\begin{itemize}
\item The first is where $m=0$ and $d \geq 0$.  Theorem \ref{maintheorem:sn}
says in this case that \eqref{eqn:snprimetoprove} is an isomorphism for
\[n \geq 2k+\max(d,m)+1 = 2k+\max(d,0)+1 = 2k+d+1\]
and a surjection for
\[n=2k+\max(d,m) = 2k+\max(d,0) = 2k+d.\]
Since $d \geq 0$, these bounds are even stronger than our purported bounds of
\[n \geq \max(m,2k+2d+2) = \max(0,2k+2d+2) = 2k+2d+2\]
for \eqref{eqn:snprimetoprove} to be an isomorphism and
\[n \geq \max(m,2k+2d) = \max(0,2k+2d) = 2k+2d\]
for \eqref{eqn:snprimetoprove} to be a surjection.
\item The second is where $m \geq 0$ and $d = -1$.  In this case, by the definition of an $\FI$-module
being polynomial of degree $-1$ starting at $m$ we have for $n \geq m$ that $M(\overline{n}) = 0$ and
hence $\HH_k(\fS_{n};M(\overline{n})) = 0$.  In other words, for $n \geq m$ the domain
and codomain of \eqref{eqn:snprimetoprove} are both $0$, so it is trivially an isomorphism.
\end{itemize}
 
Assume now that $m \geq 1$ and $d \geq 0$, and that the theorem is true for all such $(m',d')$ with $m' \leq m$ and $d' \leq d$ such that
either $m' < m$ or $d' < d$ (or both).
As in Definition \ref{definition:derivedfi}, let $\Sigma M$ be the shifted $\FI$-module and $D M$ be the derived
$\FI$-module.  For $n \geq m$, we have a short exact sequence
\begin{equation}
\label{eqn:fishiftseq}
0 \longrightarrow M(\on) \longrightarrow \Sigma M(\on) \longrightarrow DM(\on) \longrightarrow 0
\end{equation}
of $\bbk[\fS_n]$-modules (see Remark \ref{remark:exactsequence}).  The $\FI$-module $\Sigma M$ is polynomial
of degree $d$ starting at $(m-1)$, and the $\FI$-module $DM$ is polynomial of degree $(d-1)$ starting at $(m-1)$.

To simplify our notation, for all $r \geq 0$ and
all $\bbk[\fS_r]$-modules $N$, we will denote $\HH_k(\fS_r;N)$ by $\HH_k(N)$.
The long exact sequence in $\fS_n$-homology associated to \eqref{eqn:fishiftseq}
maps to the one in $\fS_{n+1}$-homology, so for $n \geq m$ and all $k$ we have a commutative diagram
\begin{center}
\scalebox{0.89}{$\minCDarrowwidth10pt\begin{CD}
\HH_{k+1}(\Sigma M(\on))            @>>> \HH_{k+1}(DM(\on))            @>>> \HH_k(M(\on))            @>>> \HH_k(\Sigma M(\on))            @>>> \HH_k(DM(\on)) \\
@VV{g_1}V                                @VV{g_2}V                          @VV{f_1}V                     @VV{f_2}V                            @VV{f_3}V \\
\HH_{k+1}(\Sigma M(\overline{n+1})) @>>> \HH_{k+1}(DM(\overline{n+1})) @>>> \HH_k(M(\overline{n+1})) @>>> \HH_k(\Sigma M(\overline{n+1})) @>>> \HH_k(DM(\overline{n+1}))
\end{CD}$}
\end{center}
with exact rows.  Our inductive hypothesis says the following about the $g_i$ and $f_i$:
\begin{itemize}
\item Since $\Sigma M$ is polynomial of degree $d$ starting at $(m-1)$, the map $f_2$ is an isomorphism
for $n \geq \max(m-1,2k+2d+2)$ and a surjection for $n \geq \max(m-1,2k+2d)$.  Also, the map
$g_1$ is an isomorphism for 
\[n \geq \max(m-1,2(k+1)+2d+2) = \max(m-1,2k+2d+4)\] 
and a surjection for $n \geq \max(m-1,2k+2d+2)$.
\item Since $DM$ is polynomial of degree $(d-1)$ starting at $(m-1)$, the map $f_3$ is an
isomorphism for 
\[n \geq \max(m-1,2k+2(d-1)+2) = \max(m-1,2k+2d)\]
and a surjection for $n \geq \max(m-1,2k+2d-2)$.  Also, the map $g_2$ is an isomorphism for
\[n \geq \max(m-1,2(k+1)+2(d-1)+2) = \max(m-1,2k+2d+2)\]
and a surjection for $n \geq \max(m-1,2k+2d)$.
\end{itemize}
For $n \geq \max(m,2k+2d+2)$, the maps $g_2$ and $f_2$ and $f_3$ are isomorphisms and the map $g_1$ is a surjection, so
by the five-lemma the map $f_1$ is an isomorphism.  For $n \geq \max(m,2k+2d)$, the maps $g_2$ and $f_2$ are surjections
and the map $f_3$ is an isomorphism, so by the five-lemma\footnote{Or, more precisely, one of the four-lemmas.} the
map $f_1$ is a surjection.  The claim follows.
\end{proof}

\section{The stable rank of rings}
\label{section:stablerank}

The rest of the paper is devoted to the general linear group and its congruence subgroups.
We begin with a discussion of a ring-theoretic condition called the stable
rank that was introduced by Bass \cite{BassKTheory}.  To make this paper a bit
more self-contained,\footnote{And to avoid depending on \cite{BassKTheory},
which is out of print.} we will include the proofs of the results we need (most of which are due
to Bass) if they are short.

\subsection{Unimodular vectors}

The starting point is the following.

\begin{definition}
Let $R$ be a ring.  A vector $v \in R^n$ is {\em unimodular} if there is a homomorphism
$\phi\colon R^n \rightarrow R$ of right $R$-modules such that $\phi(v) = 1$.
\end{definition}

\begin{example}
\label{example:unimodular}
The columns of matrices in $\GL_n(R)$ are unimodular.
\end{example}

The following lemma might clarify this definition.

\begin{lemma}
\label{lemma:unimodularsum}
Let $R$ be a ring and let $v = (c_1,\ldots,c_n) \in R^n$.  Then $v$ is unimodular if and only if
there exist $a_1,\ldots,a_n \in R$ such that $a_1 c_1 + \cdots + a_n c_n = 1$.
\end{lemma}
\begin{proof}
Immediate from the fact that all right $R$-module morphisms $\phi\colon R^n \rightarrow R$ are of the
form $\phi(x_1,\ldots,x_n) = a_1 x_1 + \cdots + a_n x_n$ for some $a_1,\ldots,a_n \in R$.
\end{proof}

\subsection{Stable rank}
In light of Example \ref{example:unimodular}, one might hope that all unimodular vectors
in $R^n$ can appear as columns of matrices in $\GL_n(R)$.  This holds if $R$ is a field or
if $R = \Z$, but does not hold in general.  To clarify this, we make the following definition.

\begin{definition}
\label{definition:stablerank}
A ring $R$ satisfies {\em Bass's stable rank condition} $(\SR_r)$ if the following holds for all
$n \geq r$.  Let $(c_1,\ldots,c_n) \in R^n$ be a unimodular vector.  Then there
exist $b_1,\ldots,b_{n-1} \in R$
such that $(c_1+b_1 c_n,\ldots,c_{n-1}+b_{n-1} c_n) \in R^{n-1}$ is unimodular.
\end{definition}

\begin{example}
\label{example:srremark}
This condition only makes sense for $r \geq 2$.  It is easy to see that fields satisfy $(\SR_2)$ and that PIDs satisfy $(\SR_3)$.
More generally, if a ring $R$ is finitely generated as a module over a Noetherian commutative ring $S$ of
Krull dimension $r$, then $R$ satisfies $(\SR_{r+2})$ (see \cite[Theorem V.3.5]{BassKTheory}).
\end{example}

\subsection{Elementary matrices and their action on unimodular vectors}

Recall that an {\em elementary matrix} in $\GL_n(R)$ is a matrix
that differs from the identity matrix at exactly one off-diagonal position.  Let $\EL_n(R)$ be
the subgroup of $\GL_n(R)$ generated by elementary matrices.
One of the key properties of rings satisfying Bass's stable rank condition is the following lemma,
which implies in particular that under its hypotheses all unimodular vectors appear as
columns of matrices in $\GL_n(R)$.  In fact, they even appear as columns of matrices in $\EL_n(R)$.

\begin{lemma}
\label{lemma:unimodulartransitive}
Let $R$ be a ring satisfying $(\SR_r)$ and let $n \geq r$.  Then $\EL_n(R)$ acts
transitively on the set of unimodular vectors in $R^n$.
\end{lemma}
\begin{proof}
Let $v \in R^n$ be a unimodular vector.  It is enough to find some $M \in \EL_n(R)$
such that $M \cdot v = (0,\ldots,0,1)$.  Write $v = (c_1,\ldots,c_n)$.  Using elementary
matrices, for any distinct $1 \leq i,j \leq n$ and any $\lambda \in R$ we can
add $\lambda$ times the $i^{\text{th}}$ entry to the $j^{\text{th}}$ one, changing the entry
$c_j$ to $c_j + \lambda c_i$.  We will use a series of these moves to transform
$v$ into $(0,\ldots,0,1)$.

Applying $(\SR_r)$, we can add multiples of the last entry to the other ones to ensure
that the first $(n-1)$ entries of $v$ form a unimodular vector in $R^{n-1}$.  Using Lemma \ref{lemma:unimodularsum},
we can find $a_1,\ldots,a_{n-1} \in R$ such that $a_1 c_1 + \cdots + a_{n-1} c_{n-1} = 1$.  For each $1 \leq i \leq n-1$, add $(1-c_n) a_i$ times
the $i^{\text{th}}$ entry to the $n^{\text{th}}$ one.  This transforms $v$ into
$(c_1,\ldots,c_{n-1},1)$.  Finally, add appropriate multiples of the last entry
of $v$ to the other ones to transform it into $(0,\ldots,0,1)$.
\end{proof}

\subsection{Generation by elementary matrices}

If $\bbk$ is a field then $\EL_n(\bbk) = \SL_n(\bbk)$,
so $\GL_n(\bbk)$ is generated by $\EL_n(\bbk)$ and $\GL_1(\bbk)$, which is embedded in $\GL_n(\bbk)$ via the upper
left hand corner matrix embedding.  The following lemma shows that the stable rank condition implies
a similar type of result:

\begin{lemma}
\label{lemma:generategl}
Let $R$ be a ring satisfying $(\SR_r)$ and let $n \geq r-1$.  Then $\GL_n(R)$ is generated
by $\EL_n(R)$ and $\GL_{r-1}(R) \subset \GL_n(R)$.
\end{lemma}
\begin{proof}
The proof will be by induction on $n$.  The base case $n=r-1$ is trivial, so assume that $n \geq r$ and that
the lemma is true for smaller $n$.
Let $\{v_1,\ldots,v_n\}$ be the standard basis for the right $R$-module $R^n$ and let $C = \oplus_{i=1}^{n-1} v_i \cdot R$.
Consider some $M \in \GL_n(R)$.  Applying Lemma \ref{lemma:splitunimodulartransitive} below, we can find some $N \in \EL_n(R)$
such that $N \cdot (M \cdot v_n) = v_n$ and $N \cdot (M \cdot C) = C$.  It follows that $N M \in \GL_{n-1}(R)$, so
by induction $N M$ lies in the subgroup generated by elementary matrices and $\GL_{r-1}(R)$.  We conclude that $M$
does as well.
\end{proof}

The above proof used the following lemma, which refines Lemma \ref{lemma:unimodulartransitive}.

\begin{lemma}
\label{lemma:splitunimodulartransitive}
Let $R$ be a ring satisfying $(\SR_r)$ and let $n \geq r$.  Let $x,y \in R^n$ be unimodular vectors and
$C,D \subset R^n$ be $R$-submodules such that 
$R^n = C \oplus x \cdot R$ and $R^n = D \oplus y \cdot R$.  Then
there exists some $M \in \EL_n(R)$ such that $M \cdot x = y$ and $M \cdot C = D$.
\end{lemma}
\begin{proof}
Let $\{v_1,\ldots,v_n\}$ be the standard basis for the right $R$-module $R^n$.  It is enough to deal with
the case where $x = v_n$ and $C = \oplus_{i=1}^{n-1} v_i \cdot R$.  What is more, 
Lemma \ref{lemma:unimodulartransitive} says that $\EL_n(R)$ acts
transitively on unimodular vectors in $R^n$, so we can assume without loss of generality that $y = v_n$ as well.
Let $\rho\colon R^n \rightarrow R$ be the projection with $D = \ker(\rho)$ and $\rho(v_n) = 1$.
For $1 \leq i \leq n-1$, let $\lambda_i = \rho(v_i)$.  Define $M \in \GL_n(R)$ via the formula
\[M \cdot v_i = v_i - v_n \cdot \lambda_i \quad \text{for $1 \leq i \leq n-1$, and $M \cdot v_n = v_n$}.\]
It is easy to see that $M$ can be written as a product of $(n-1)$ elementary matrices, so $M \in \EL_n(R)$.
For $1 \leq i \leq n-1$, we have
\[\rho(M \cdot v_i) = \rho(v_i) - \rho(v_n) \lambda_i = \lambda_i - \lambda_i = 0,\]
so $M \cdot v_i \in D$.  We conclude that $M \cdot C = D$. 
\end{proof}

\subsection{Stable freeness}

For a general ring $R$, there can exist non-free $R$-modules $C$ that are stably free in the sense
that $C \oplus R^k \cong R^n$ for some $k \geq 1$.  The stable rank condition prevents this, at least
if $n-k$ (which should be thought of as the ``rank'' of $C$) is large enough:

\begin{lemma}
\label{lemma:stabledecomp}
Let $R$ be a ring satisfying $(\SR_r)$.  Let $C$ be a right $R$-module such that $C \oplus R^k \cong R^n$.
Assume that $n-k \geq r-1$.  Then $C \cong R^{n-k}$.
\end{lemma}
\begin{proof}
It is enough to deal with the case where $k=1$, so $n \geq r$.  Identify $C$ with a submodule of $R^n$ and
let $x \in R^n$ be a unimodular vector with $R^n = C \oplus x \cdot R$.  Letting $\{v_1,\ldots,v_n\}$
be the standard basis for $R^n$, Lemma \ref{lemma:splitunimodulartransitive} implies that there exists
some $M \in \GL_n(R)$ with $M \cdot x = v_n$ and $M \cdot C = \oplus_{i=1}^{n-1} v_i \cdot R$.  In particular,
$M \cdot C \cong R^{n-1}$, so $C \cong R^{n-1}$ as well.
\end{proof}

\subsection{Quotients of rings}

The following lemma shows that the stable rank condition is preserved by quotients:

\begin{lemma}
\label{lemma:stablequotient}
Let $R$ be a ring satisfying $(\SR_r)$ and let $\fq$ be a two-sided ideal in $R$.  Then $R/\fq$ satisfies
$(\SR_r)$.
\end{lemma}
\begin{proof}
Let $n \geq r$ and let $\ov \in (R/\fq)^n$ be a unimodular vector.  It is enough to prove
that $\ov$ can be lifted to a unimodular vector in $R^n$.  Let $(c_1,\ldots,c_n) \in R^n$ be
any vector projecting to $\ov$.  Since $\ov$ is unimodular, Lemma \ref{lemma:unimodularsum} implies that
there exist $a_1,\ldots,a_n \in R$ and $q \in \fq$ such that $a_1 c_1 + \cdots + a_n c_n = 1+q$.  It
follows that $(c_1,\ldots,c_n,q) \in R^{n+1}$ is unimodular, so by $(\SR_r)$ we can find
$b_1,\ldots,b_n \in R$ such that $v = (c_1 + b_1 q,\ldots,c_n+b_n q) \in R^n$ is unimodular.  The
vector $v$ projects to $\ov$.
\end{proof}

\subsection{Stable rank modulo an ideal}

The following lemma gives a variant of the stable rank condition that takes into account an ideal.

\begin{lemma}
\label{lemma:stableideal}
Let $R$ be a ring satisfying $(\SR_r)$ and let $\fq$ be a two-sided ideal of $R$.  For some
$n \geq r$, let $(c_1,\ldots,c_n) \in R^n$ be a unimodular vector with $c_n \in \fq$.
Then there exists $b_1,\ldots,b_{n-1} \in \fq$ such that
$(c_1+b_1 c_n,\ldots,c_{n-1} + b_{n-1} c_n) \in R^{n-1}$ is unimodular.
\end{lemma}
\begin{proof}
Using Lemma \ref{lemma:unimodularsum}, we can find $a_1,\ldots,a_n \in R$ with $a_1 c_1+\ldots+a_n c_n = 1$.
We claim that $(c_1,\ldots,c_{n-1},c_n a_n c_n) \in R^n$ is unimodular.  Indeed, we have
\begin{align*}
&\left(a_1+a_n c_n a_1\right)c_1 + \cdots + \left(a_{n-1}+a_n c_n a_{n-1}\right) c_{n-1} + \left(a_n\right) c_n a_n c_n \\
&\ \ \ \ = \left(a_1 c_1 + \cdots + a_{n-1} c_{n-1}\right) + a_n c_n \left(a_1 c_1 + \cdots + a_n c_n\right) \\
&\ \ \ \ = \left(a_1 c_1 + \cdots + a_{n-1} c_{n-1}\right) + a_n c_n = 1,
\end{align*}
as desired.  Applying $(\SR_r)$, we can find $b'_1,\ldots,b'_{n-1} \in R$ such that
$(c_1 + b'_1 c_n a_n c_n,\ldots,c_{n-1} + b'_{n-1} c_n a_n c_n) \in R^{n-1}$ is unimodular.  Since
$c_n \in \fq$, so is $b_i = b'_i c_n a_n$.  The lemma follows.
\end{proof}

\subsection{Elementary congruence subgroups and their action on unimodular vectors}

Recall that if $\fq$ is an ideal in $R$, then $\GL_n(R,\fq)$ is the level-$\fq$ congruence subgroup
of $\GL_n(R)$, i.e., the kernel of the map $\GL_n(R) \rightarrow \GL_n(R/\fq)$.  An elementary
matrix lies in $\GL_n(R,\fq)$ precisely when its single off-diagonal entry lies in $\fq$.  Let
$\EL_n(R,\fq)$ be the subgroup of $\EL_n(R)$ that is {\bf normally} generated by elementary
matrices lying in $\GL_n(R,\fq)$.  

Lemma \ref{lemma:unimodulartransitive}
says that the stable rank condition implies that $\EL_n(R)$ acts transitively
on the set of unimodular vectors when $n$ is sufficiently large.  The following lemma strengthens this:

\begin{lemma}
\label{lemma:unimodulartransitivecongruence}
Let $R$ be a ring satisfying $(\SR_r)$ and let $\fq$ be an ideal in $R$.  For some
$n \geq r$, let $v,v' \in R^n$ be unimodular vectors that map to the same vector in $(R/\fq)^n$.
Then there exists some $M \in \EL_n(R,\fq)$ with $M \cdot v = v'$.
\end{lemma}
\begin{proof}
We will prove this in two steps.

\begin{stepsb}
This is true if $v' = (1,0,\ldots,0)$.
\end{stepsb}

Since $v$ and $(1,0,\ldots,0)$ map to the same vector in $(R/\fq)^n$, we can write
$v = (1+q_1,q_2,\ldots,q_n)$ with $q_1,\ldots,q_n \in \fq$.  Recall that an elementary matrix
lies in $\GL_n(R,\fq)$ if its single off-diagonal entry lies in $\fq$.  Just like
in the proof of Lemma \ref{lemma:unimodulartransitive}, using such elementary
matrices, for any distinct $1 \leq i,j \leq n$ and any $\lambda \in \fq$ we can add
$\lambda$ times the $i^{\text{th}}$ entry to the $j^{\text{th}}$ one.  We will use
a series of these moves (plus one extra trick) to transform $v$ into $(1,0,\ldots,0)$.

Applying the relative version of $(\SR_r)$ given by Lemma \ref{lemma:stableideal},
we can add $\fq$-multiples of the last entry to the other ones to ensure
that the first $(n-1)$ entries of $v$ form a unimodular vector in $R^{n-1}$.
Using Lemma \ref{lemma:unimodularsum}, we can thus find $a_1,\ldots,a_{n-1} \in R$ such
that $a_1(1+q_1) + a_2 q_2 + \cdots + a_{n-1} q_{n-1} = 1$.  For each
$1 \leq i \leq n-1$, add $(q_1-q_n) a_i \in \fq$ times the
$i^{\text{th}}$ entry to the $n^{\text{th}}$ one.  This transforms $v$ into
$(1+q_1,q_2,\ldots,q_{n-1},q_1)$.

At this point, we will do something slightly tricky.  Let $E \in \EL_n(R)$ be the
elementary matrix that subtracts the $n^{\text{th}}$ row from the first one (notice
that $E$ does not lie in $\GL_n(R,\fq)$).  We thus
have
\[E \cdot v = E \cdot (1+q_1,q_2,\ldots,q_{n-1},q_1) = (1,q_2,\ldots,q_{n-1},q_1).\]
We can then find a product $F \in \EL_n(R,\fq)$ of elementary matrices such that
\[F \cdot (1,q_2,\ldots,q_{n-1},q_1) = (1,0,\ldots,0).\]
Indeed, $F$ first subtracts $q_2 \in \fq$ times the first entry from the $2^{\text{nd}}$, then
subtracts $q_3 \in \fq$ times the first entry from the $3^{\text{rd}}$, etc., and
finishes by subtracting $q_1$ times the first entry from the $n^{\text{th}}$.  Since $E$ fixes
$(1,0,\ldots,0)$, it follows that
\[E^{-1} F E \cdot v = (1,0,\ldots,0).\]
Since $\EL_n(R,\fq)$ is a normal subgroup of $\EL_n(R)$, we have $M = E^{-1} F E \in \EL_n(R,\fq)$, as desired.

\begin{stepsb}
This is true in general.
\end{stepsb}

Set $w = (1,0,\ldots,0)$.  Lemma \ref{lemma:unimodulartransitive} says that there exists some $A \in \EL_n(R)$ with $A \cdot v' = w$.
Applying the previous step to $A \cdot v$, we can find some $B \in \EL_n(R,\fq)$ with $B \cdot A \cdot v = w$, so
$A^{-1} B A \cdot v = A^{-1} \cdot w = v'$.  Since $\EL_n(R,\fq)$ is a normal subgroup of $\EL_n(R)$, it follows
that $M = A^{-1} B A$ lies in $\EL_n(R,\fq)$.
\end{proof}

\subsection{Action of congruence subgroups on split unimodular vectors}

The following lemma refines Lemma \ref{lemma:unimodulartransitivecongruence} in the same way
that Lemma \ref{lemma:splitunimodulartransitive} refines Lemma \ref{lemma:unimodulartransitive}:

\begin{lemma}
\label{lemma:splitunimodulartransitivecongruence}
Let $R$ be a ring satisfying $(\SR_r)$, let $\fq$ be an ideal in $R$, and let $n \geq r$.
Let $x,y \in R^n$ be unimodular vectors and $C,D \subset R^n$ be $R$-submodules such that
$R^n = C \oplus x \cdot R$ and $R^n = D \oplus y \cdot R$.  Assume that $x$ and $y$ (resp.\ $C$ and
$D$) map to the same vector (resp.\ submodule) in $(R/\fq)^n$.  Then there exists some $M \in \EL_n(R,\fq)$
with $M \cdot x = y$ and $M \cdot C = D$.
\end{lemma}
\begin{proof}
Using Lemma \ref{lemma:unimodulartransitivecongruence}, we can assume without loss of generality that $x = y$.
Let $\rho\colon R^n \rightarrow R$ be the projection with $D = \ker(\rho)$ and $\rho(x) = 1$.
Lemma \ref{lemma:stabledecomp} implies that $C$ and $D$ are both free of rank $(n-1)$.  Let
$\{v_1,\ldots,v_{n-1}\}$ be a free basis for $C$ and let $v_n = x = y$, so $\{v_1,\ldots,v_n\}$ is a free
basis for $R^n$.  For $1 \leq i \leq n-1$, let $\lambda_i = \rho(v_i)$.  
Since $C$ and $D$ map to the same submodule of $(R/\fq)^n$, we have $\lambda_i \in \fq$.  Define
$M \in \GL_n(R,\fq)$ via the formula
\[M \cdot v_i = v_i - v_n \cdot \lambda_i \quad \text{for $1 \leq i \leq n-1$, and $M \cdot v_n = v_n$}.\]
It is easy to write $M$ as a product of elementary matrices lying in $\GL_n(R,\fq)$, so $M \in \EL_n(R,\fq)$.
For $1 \leq i \leq n-1$, we have
\[\rho(M \cdot v_i) = \rho(v_i) - \rho(v_n) \lambda_i = \lambda_i - \lambda_i = 0,\]
so $M \cdot v_i \in D$.  We conclude that $M \cdot C = D$, as desired.
\end{proof}

\section{Subgroups from K-theory}
\label{section:ktheory}

Fix a ring $R$.  This section discusses subgroups of $\GL_n(R)$ arising from algebraic K-theory.  One
can view them as generalizations of the special linear group.

\subsection{Absolute K-groups}

Define
\[\GL(R) = \bigcup_{n=1}^{\infty} \GL_n(R) \quad \text{and} \quad \EL(R) = \bigcup_{n=1}^{\infty} \EL_n(R).\]
It turns out that $\EL(R) = [\GL(R),\GL(R)]$ (this is the ``Whitehead Lemma''; see \cite[Lemma 3.1]{MilnorKTheory}).
By definition, the first algebraic K-theory group of $R$ is the abelian group
\[\KK_1(R) = \HH_1(\GL(R)) = \GL(R) / \EL(R).\]
For a subgroup $\cK \subset \KK_1(R)$, we define $\GL_n^{\cK}(R)$ to be the preimage of $\cK$ under
the projection
\[\GL_n(R) \hookrightarrow \GL(R) \longrightarrow \KK_1(R).\]
It follows that $\EL_n(R) \subset \GL_n^{\cK}(R) \subset \GL_n(R)$.  

\begin{example}
Let $R$ be a ring.  For $\cK = \KK_1(R)$, we have $\GL_n^{\cK}(R) = \GL_n(R)$.
\end{example}

\begin{example}
Let $R$ be a commutative ring and let $\det\colon \GL(R) \rightarrow R^{\times}$
be the determinant map.  Since the determinant of an elementary matrix is $1$, the map
$\det$ induces a map
\[\fd\colon \KK_1(R) = \GL(R)/\EL(R) \rightarrow R^{\times}.\]
Letting $\cK = \ker(\fd)$, we then have $\GL_n^{\cK}(R) = \SL_n(R)$.  More generally, if $\cD \subset R^{\times}$
is a subgroup and $\cK = \fd^{-1}(\cD)$, then $\GL_n^{\cK}(R)$ is the subgroup of $\GL_n(R)$ consisting
of matrices $M$ with $\det(M) \in \cD$. 
\end{example}

\subsection{Relative K-groups}

Let $\alpha$ be a two-sided ideal of $R$.  Define
\[\GL(R,\alpha) = \bigcup_{n=1}^{\infty} \GL_n(R,\alpha) \quad \text{and} \quad \EL(R,\alpha) = \bigcup_{n=1}^{\infty} \EL_n(R,\alpha).\]
Just like in the absolute case, $\EL(R,\alpha)$ is a normal subgroup of $\GL(R,\alpha)$ with abelian
quotient, and this abelian quotient is the first relative K-theory group of $(R,\alpha)$:
\[\KK_1(R,\alpha) = \GL(R,\alpha) / \EL(R,\alpha).\]
See \cite[Lemma 4.2]{MilnorKTheory}.  There is also a relative version of the Whitehead lemma discussed above
that says that
\[\EL(R,\alpha) = [\GL(R,\alpha),\GL(R)].\]
See \cite[Lemma 4.3]{MilnorKTheory}.  In particular, $\EL(R,\alpha)$ is generally larger
than $[\GL(R,\alpha),\GL(R,\alpha)]$.

For a subgroup $\cK' \subset \KK_1(R,\alpha)$, we define $\GL_n^{\cK'}(R,\alpha)$  to
be the preimage of $\cK'$ under the composition
\[\GL_n(R,\alpha) \hookrightarrow \GL(R,\alpha) \longrightarrow \KK_1(R,\alpha).\]
Letting $\cK \subset \KK_1(R)$ be the image of $\cK$ under the map $\KK_1(R,\alpha) \rightarrow \KK_1(R)$, we have
$\GL_n^{\cK'}(R,\alpha) \subset \GL_n^{\cK}(R)$.  Just like in the above examples, by
choosing $\cK'$ appropriately we can have $\GL_n^{\cK'}(R,\alpha) = \GL_n(R,\alpha)$ or if $R$ is commutative
$\GL_n^{\cK'}(R,\alpha) = \SL_n(R,\alpha)$.

\subsection{Injective stability}

To clarify the nature of these groups, we need the following deep theorem of Vaserstein (``injective
stability for $K_1$''):

\begin{theorem}[{Vaserstein, \cite{Vaserstein}}]
\label{theorem:injectivek1}
Let $R$ be a ring satisfying $(\SR_r)$ and let $n \geq r$.  Then $\EL_n(R)$ is a normal
subgroup of $\GL_n(R)$, and the composition
\[\GL_n(R)/\EL_n(R) \longrightarrow \GL(R)/\EL(R) = \KK_1(R)\]
is an isomorphism.  Moreover, if $\alpha$ is a two-sided ideal of $R$, then
$\EL_n(R,\alpha)$ is a normal subgroup of $\GL_n(R,\alpha)$, and the composition
\[\GL_n(R,\alpha) / \EL_n(R,\alpha) \longrightarrow \GL(R,\alpha) / \EL(R,\alpha) = \KK_1(R,\alpha)\]
is an isomorphism.
\end{theorem}

This allows us to deal with another example:

\begin{example}
\label{example:elementaryk}
Let $R$ be a ring satisfying $(\SR_r)$ and let $n \geq r$.  Set $\cK = 0$.  Then by Theorem \ref{theorem:injectivek1}
we have $\GL_n^{\cK}(R) = \EL_n(R)$, and for a two-sided ideal $\alpha$ of $R$ we have
$\GL_n^{\cK}(R,\alpha) = \EL_n(R,\alpha)$.
\end{example}

It also allows us to prove the following useful result:

\begin{lemma}
\label{lemma:stablesubgroup}
Let $R$ be a ring satisfying $(\SR_r)$ and let $n \geq r$.  For all subgroups $\cK \subset \KK_1(R)$,
we have
\[\GL_{n+m}^{\cK}(R) \cap \GL_n(R) = \GL_n^{\cK}(R) \quad \text{for all $m \geq 1$}.\]
Similarly, for all two-sided ideals $\alpha$ of $R$ and for all subgroups $\cK' \subset \KK_1(R,\alpha)$,
we have
\[\GL_{n+m}^{\cK'}(R,\alpha) \cap \GL_n(R,\alpha) = \GL_n^{\cK}(R,\alpha) \quad \text{for all $m \geq 1$}.\]
\end{lemma}
\begin{proof}
It is enough to deal with the case $m=1$.  By Theorem \ref{theorem:injectivek1}, the maps
\[\GL_n(R) / \EL_n(R) \rightarrow \GL_{n+1}(R) / \EL_{n+1}(R) \rightarrow \KK_1(R)\]
are both isomorphisms.  This implies that
\[\GL_{n+1}^{\cK}(R) \cap \GL_n(R) = \GL_n^{\cK}(R).\]
The relative statement is proved similarly.
\end{proof}

\subsection{Elementary vs relative elementary subgroups}

Let $\alpha$ be a two-sided ideal of $R$ and let $\cK' \subset \KK_1(R,\alpha)$ be a subgroup.  Define
\[\oEL^{\cK'}_n(R,\alpha) = \GL^{\cK'}_n(R,\alpha) \cap \EL_n(R).\]
We thus have $\EL_n(R,\alpha) \subset \oEL^{\cK'}_n(R,\alpha)$.  This
need not be an equality:

\begin{lemma}
\label{lemma:elo}
Let $R$ be a ring satisfying $(\SR_r)$, let $n \geq r$, let $\alpha$ be a two-sided ideal of $R$,
and let $\cK' \subset \KK_1(R,\alpha)$ be a subgroup.  Let $\cL \subset \KK_1(R,\alpha)$
be the intersection of $\cK'$ with the kernel of the map $\KK_1(R,\alpha) \rightarrow \KK_1(R)$.  We
then have a short exact sequence
\[1 \longrightarrow \EL_n(R,\alpha) \longrightarrow \oEL^{\cK'}_n(R,\alpha) \longrightarrow \cL \longrightarrow 0.\]
\end{lemma}
\begin{proof}
By Theorem \ref{theorem:injectivek1}, the group $\oEL^{\cK'}_n(R,\alpha)$ is the kernel of the composition
\[\GL^{\cK'}_n(R,\alpha) \hookrightarrow \GL_n(R) \rightarrow \KK_1(R).\]
Theorem \ref{theorem:injectivek1} also implies that this composition can be factored as
\[\GL^{\cK'}_n(R,\alpha) \rightarrow \GL^{\cK'}_n(R,\alpha)/\EL_n(R,\alpha) \cong \cK' \hookrightarrow \KK_1(R,\alpha) \rightarrow \KK_1(R).\]
Combining these two facts, we deduce that
\[\oEL^{\cK'}_n(R,\alpha) / \EL_n(R,\alpha) \cong \ker\left(\cK' \hookrightarrow \KK_1(R,\alpha) \rightarrow \KK_1(R)\right) = \cL,\]
which is equivalent to the claimed short exact sequence.
\end{proof}

The point of introducing the group $\oEL_n^{\cK'}(R,\alpha)$ is the following lemma: 

\begin{lemma}
\label{lemma:bigdiagram}
Let $R$ be a ring satisfying $(\SR_r)$, let $n \geq r$, and let $\alpha$ be a two-sided ideal of $R$.
Let $\cK \subset \KK_1(R)$ and $\cK' \subset \KK_1(R,\alpha)$ be subgroups such that $\cK'$ maps
to a subgroup $\ocK'$ of $\cK$ under the map $\KK_1(R,\alpha) \rightarrow \KK_1(R)$.  We then
have a commutative diagram with exact rows and injective columns
\[\begin{tikzcd}
1 \arrow{r} & \oEL^{\cK'}_n(R,\alpha) \arrow{r} \arrow[hook]{d} & \GL_n^{\cK'}(R,\alpha) \arrow{r} \arrow[hook]{d} & \ocK' \arrow{r} \arrow[hook]{d} & 1  \\
1 \arrow{r} & \EL_n(R)                \arrow{r}                 & \GL_n^{\cK}(R)         \arrow{r}                 & \cK   \arrow{r}                 & 1
\end{tikzcd}\]
\end{lemma}
\begin{proof}
Immediate from Theorem \ref{theorem:injectivek1} and the definitions.
\end{proof}

\subsection{Finite-index subgroups}

Our next goal is to clarify when $\GL_n^{\cK'}(R,\alpha)$ is a finite-index subgroup
of $\GL_n^{\cK}(R)$.  This requires the following lemma:

\begin{lemma}
\label{lemma:relativek}
Let $R$ be a ring and let $\alpha$ be a two-sided ideal of $R$ with $|R/\alpha|<\infty$.
Then the map $\KK_1(R,\alpha) \rightarrow \KK_1(R)$ has finite kernel and cokernel.
\end{lemma}
\begin{proof}
Immediate from the following exact sequence \cite[Theorem 6.2]{MilnorKTheory}:
\[\KK_2(R/\alpha) \rightarrow \KK_1(R,\alpha) \rightarrow \KK_1(R) \rightarrow \KK_1(R/\alpha).\qedhere\]
\end{proof}

Our result is as follows:

\begin{lemma}
\label{lemma:finiteindex}
Let $R$ be a ring satisfying $(\SR_r)$, let $n \geq r$, and let $\alpha$ be a two-sided ideal of $R$ with
$|R/\alpha| < \infty$.  The following then hold:
\begin{itemize}
\item[(i)] $\EL_n(R,\alpha)$ is a finite-index subgroup of $\EL_n(R)$.
\item[(ii)] Let $\cK \subset \KK_1(R)$ and $\cK' \subset \KK_1(R,\alpha)$ be subgroups such that $\cK'$ maps
to a finite-index subgroup of $\cK$ under the map $\KK_1(R,\alpha) \rightarrow \KK_1(R)$.  Then
$\GL_n^{\cK'}(R,\alpha)$ is a finite-index subgroup of $\GL^{\cK}_n(R)$.
\end{itemize}
\end{lemma}
\begin{proof}
We start with (i).  Since $\GL_n(R,\alpha)$ is the kernel of the map 
$\GL_n(R) \rightarrow \GL_n(R/\alpha)$ and $|R/\alpha|<\infty$, the group $\GL_n(R,\alpha)$ 
is a finite-index subgroup of $\GL_n(R)$.  Letting
\[\oEL_n(R,\alpha) = \oEL_n^{\KK_1(R,\alpha)}(R,\alpha),\]
by definition $\oEL_n(R,\alpha)$ is the intersection of $\EL_n(R) \subset \GL_n(R)$ with the
finite-index subgroup $\GL_n(R,\alpha)$ of $\GL_n(R)$.  It follows that $\oEL_n(R,\alpha)$
is a finite-index subgroup of $\EL_n(R)$.  Finally, by 
Lemmas \ref{lemma:elo} and \ref{lemma:relativek}, the group $\EL_n(R,\alpha)$ is a
finite-index subgroup of $\oEL_n(R,\alpha)$.  Summarizing this all, both inclusions
\[\EL_n(R,\alpha) \subset \oEL_n(R,\alpha) \subset \EL_n(R)\]
are inclusions of finite-index subgroups, so $\EL_n(R,\alpha)$ is a finite-index subgroup of $\EL_n(R)$,
as desired.

We next prove (ii).  Let $\ocK'$ be the image of $\cK'$ in $\cK$.  Lemma \ref{lemma:bigdiagram}
gives a commutative diagram with exact rows and injective columns
\[\begin{tikzcd}
1 \arrow{r} & \oEL^{\cK'}_n(R,\alpha) \arrow{r} \arrow[hook]{d}{i_1} & \GL_n^{\cK'}(R,\alpha) \arrow{r} \arrow[hook]{d}{i_2} & \ocK' \arrow{r} \arrow[hook]{d}{i_3} & 1  \\
1 \arrow{r} & \EL_n(R)                \arrow{r}                 & \GL_n^{\cK}(R)         \arrow{r}                 & \cK   \arrow{r}                 & 1
\end{tikzcd}\]
The vertical maps $i_1,i_2,i_3$ are inclusions, and one of our assumptions is that $\ocK'$
is a finite-index subgroup of $\cK$.  Since $\EL_n(R,\alpha)$ is a finite-index
subgroup of $\EL_n(R)$ and
\[\EL_n(R,\alpha) \subset \oEL^{\cK'}_n(R,\alpha) \subset \EL_n(R),\]
the group $\oEL^{\cK'}_n(R,\alpha)$ is a finite-index subgroup of $\EL_n(R)$.  
Since both $i_1$ and $i_3$ are inclusions of finite-index subgroups, we conclude that $i_2$
is also an inclusion of a finite-index subgroup, as desired.
\end{proof}

\section{Complex of split partial bases}
\label{section:splitbases}

We now introduce a simplicial complex called the complex of split partial bases, which
is a tiny variant on one introduced by Charney \cite{CharneyCongruence}.

\begin{remark}
The lemmas we prove in this section are all essentially due to Charney \cite{CharneyCongruence}.  Many
of them are also proved in \cite[\S 5.3]{RandalWilliamsWahl}.  We include proofs because they are short, and also
because it is annoyingly nontrivial to match up our notation and indexing conventions with those papers.
\end{remark}

\subsection{Definition}
Let $R$ be a ring and let $M$ be a right $R$-module.  We say that elements $y_1,\ldots,y_m \in M$ 
are {\em linearly independent} if the map $R^m \rightarrow M$ taking $(\lambda_1,\ldots,\lambda_m) \in R^m$
to $y_1 \cdot \lambda_1 + \cdots + y_m \cdot \lambda_m$ is injective.  The {\em complex of split partial bases} for $M$, denoted
$\Bases(M)$, is the following simplicial complex:
\begin{itemize}
\item The vertices are pairs $(x;C)$ with $x \in M$ a linearly independent element\footnote{This means
that $x \cdot \lambda \neq 0$ for all nonzero $\lambda \in R$.} and
$C \subseteq M$ a submodule such that 
$M = C \oplus x \cdot R$.
\item A collection $\{(x_0;C_0),\ldots,(x_k;C_k)\}$ of vertices forms a $k$-simplex
if $x_i \in C_j$ for all distinct $0 \leq i,j \leq k$.
\end{itemize}
The group $\Aut_R(M)$ acts on $\Bases(M)$ on the left.  In particular, $\GL_n(R)$ acts on $\Bases(R^n)$.

\subsection{Splitting from a simplex}

The following lemma might clarify the definition of a simplex in $\Bases(M)$.

\begin{lemma}
\label{lemma:splitsimplex}
Let $R$ be a ring, let $M$ be a right $R$-module, and let $\{(x_0;C_0),\ldots,(x_k;C_k)\}$ be a $k$-simplex
of $\Bases(M)$.  Setting $C = C_0 \cap \cdots \cap C_k$, we then have
\begin{equation}
\label{eqn:splitsimplex}
M = C \oplus \bigoplus_{i=0}^k x_i \cdot R.
\end{equation}
In particular, the elements $x_0,\ldots,x_k$ are linearly independent.
\end{lemma}
\begin{proof}
The proof is by induction on $k$.  The base case $k=0$ is trivial, so assume that $k>0$ and that
the lemma is true for all smaller $k$.  For $0 \leq i \leq k-1$, let $C_i' = C_i \cap C_k$.  For
these $i$, we claim that $C_k = C_i' \oplus x_i \cdot R$.  Since $x_i \in C_k$ and
$M = C_i \oplus x_i \cdot R$, it is enough to show that
for all $z \in C_k$, we can write $z = z' + x_i \cdot a$ for some
$z' \in C_i'$ and $a \in R$.  Since $M = C_i \oplus x_i \cdot R$, we can
write $z = z' + x_i \cdot a$ for some $z' \in C_i$ and $a \in R$, and
what we must show is that $z' \in C_i'$.  But since
$z,x_i \in C_k$ we have $z' = z-x_i \cdot a \in C_k$, so $z' \in C_i'$, as
desired.

It follows that $\{(x_0;C'_0),\ldots,(x_{k-1};C'_{k-1})\}$ is a $(k-1)$-simplex of
$\Bases(C_k)$.  Since
\[C = C_0 \cap \cdots \cap C_k = C'_0 \cap \cdots \cap C'_{k-1},\]
our inductive hypothesis implies that
\[C_k = C \oplus \bigoplus_{i=0}^{k-1} x_i \cdot R.\]
Since $M = C_k \oplus x_k \cdot R$, equation \eqref{eqn:splitsimplex} follows.
\end{proof}

This has the following useful consequence:

\begin{lemma}
\label{lemma:identifysimplex}
Let $R$ be a ring and let $M$ be a right $R$-module.  The $k$-simplices of $\Bases(M)$ are in bijection
with tuples $(\{x_0,\ldots,x_k\};C)$, where $\{x_0,\ldots,x_k\}$ is an unordered collection of $(k+1)$
linearly independent elements of $M$ and $C \subset M$ is a submodule such that
\[M = C \oplus \bigoplus_{i=0}^k x_i \cdot R.\]
\end{lemma}
\begin{proof}
For a $k$-simplex $\{(x_0;C_0),\ldots,(x_k,C_k)\}$ of $\Bases(M)$, by Lemma \ref{lemma:splitsimplex} we can associate
the tuple 
\[(\{x_0,\ldots,x_k\};\bigcap_{i=0}^k C_i).\]
Conversely, consider such a tuple $(\{x_0,\ldots,x_k\};C)$.  For $0 \leq j \leq k$, define
\[C_j = C \oplus \bigoplus_{\substack{0 \leq i \leq k \\ i \neq j}} x_i \cdot R.\]
Then $\{(x_0;C_0),\ldots,(x_k;C_k)\}$ is a $k$-simplex of $\Bases(M)$.  It is clear that these
two operations are inverses to one another.
\end{proof}

Subsequently, we will freely move between the notations $\{(x_0;C_0),\ldots,(x_k;C_k)\}$ and
$(\{x_0,\ldots,x_k\};C)$ for the $k$-simplices of $\Bases(M)$.

\subsection{Split partial bases and unimodular vectors}

We now focus our attention on $\Bases(R^n)$.  The first observation about this
complex is as follows:

\begin{lemma}
\label{lemma:unimodular}
Let $R$ be a ring and let $x \in R^n$.  Then $x$ is unimodular if and only if
there exists some submodule $C \subset R^n$ such that $(x;C)$ is a vertex
of $\Bases(R^n)$.
\end{lemma}
\begin{proof}
If $x$ is unimodular, then by definition we can find a homomorphism
$\phi\colon R^n \rightarrow R$ of right $R$-modules such that $\phi(x) = 1$.
This implies that $x$ is linearly independent, and letting $C = \ker(\phi)$
the pair $(x;C)$ is a vertex of $\Bases(R^n)$.  Conversely, if $(x;C)$ is 
a vertex of $\Bases(R^n)$, then we can define $\phi\colon R^n \rightarrow R$
to be the composition of the projection
\[R^n = C \oplus x \cdot R \stackrel{\text{proj}}{\longrightarrow} x \cdot R\]
with the map $x \cdot R \rightarrow R$ taking $x$ to $1$.
\end{proof}

\subsection{Links in complex of split partial bases}

If $R$ satisfies a stable rank condition, then the links in $\Bases(R^n)$
are well-behaved:

\begin{lemma}
\label{lemma:splitbaseslink}
Let $R$ be a ring satisfying $(\SR_r)$.  For some $n$ and $k$ with $k \leq n-r$, let $\sigma$ be a $k$-simplex
of $\Bases(R^n)$.  We then have $\Link_{\Bases(R^n)}(\sigma) \cong \Bases(R^{n-k-1})$.
\end{lemma}
\begin{proof}
Let $\sigma = (\{x_0,\ldots,x_k\};C)$.  Since
\[R^n = C \oplus \bigoplus_{i=0}^k x_i \cdot R,\]
Lemma \ref{lemma:stabledecomp} implies that $C \cong R^{n-k-1}$, so it is enough to
prove that $\Link_{\Bases(R^n)}(\sigma)$ and $\Bases(C)$ are isomorphic.  For this,
note that the maps $\phi\colon \Link_{\Bases(R^n)}(\sigma) \rightarrow \Bases(C)$ and
$\psi\colon \Bases(C) \rightarrow \Link_{\Bases(R^n)}(\sigma)$ defined on $\ell$-simplices via the formulas
\[\phi((\{y_0,\ldots,y_{\ell}\};D)) = (\{y_0,\ldots,y_{\ell}\};D \cap C)\]
and
\[\psi((\{z_0,\ldots,z_{\ell}\};E)) = 
(\{z_0,\ldots,z_{\ell}\};E \oplus \bigoplus_{i=0}^k x_i \cdot R)\]
are inverse isomorphisms.
\end{proof}

\subsection{General linear group action}

We now turn to the action of $\GL_n(R)$ and its elementary subgroup $\EL_n(R)$
on $\Bases(R^n)$.  Our main result will be as follows:

\begin{lemma}
\label{lemma:glbasesaction}
Let $R$ be a ring satisfying $(\SR_r)$.  Let
\[\{(x_0;C_0),\ldots,(x_k;C_k)\} \quad \text{and} \quad \{(y_0;D_0),\ldots,(y_k;D_k)\}\]
be $k$-simplices of $\Bases(R^n)$.  Assume that $k \leq n-r$.  The following then hold:
\begin{itemize}
\item[(a)] There exists $M \in \EL_n(R)$ such that $M \cdot (x_i;C_i) = (y_i;D_i)$ for $0 \leq i \leq k$.
\item[(b)] Let $\fq$ be a two-sided ideal of $R$ and let $\pi\colon R^n \rightarrow (R/\fq)^n$ be the projection.
Assume that $\pi(x_i) = \pi(y_i)$ and $\pi(C_i) = \pi(D_i)$ for $0 \leq i \leq k$.  Then the $M$ in
part (a) can be chosen to lie in $\EL_n(R,\fq)$.
\end{itemize}
\end{lemma}
\begin{proof}
Lemma \ref{lemma:unimodular} implies that the $x_i$ and $y_j$ are unimodular.  For $k=0$, the lemma
is now immediate from Lemmas \ref{lemma:splitunimodulartransitive} and \ref{lemma:splitunimodulartransitivecongruence}.
The general case can be deduced from this by induction using Lemma \ref{lemma:splitbaseslink}.
\end{proof}

\subsection{Topology}

Recall that we defined what it means for a simplicial complex to be weakly Cohen-Macaulay in
\S \ref{section:cmcomplex}.  The complex of split partial bases has this property:

\begin{theorem}
\label{theorem:basescm}
Let $R$ be a ring satisfying $(\SR_r)$ and let $n \geq r-1$.  Then $\Bases(R^n)$ is weakly Cohen-Macaulay of
dimension $\left\lfloor\frac{n-r+1}{2}\right\rfloor$.
\end{theorem}

The key to the proof of Theorem \ref{theorem:basescm} is the following theorem.  It is essentially due to Charney \cite[Theorem 3.5]{CharneyCongruence},
though she works with a slightly different complex, so one needs to adapt her proof.  For a complete proof of exactly
the statement below, see \cite[Lemma 5.10]{RandalWilliamsWahl} (or rather its proof -- the lemma in the reference deals with the large
ordering of $\Bases(R^n)$, but proves this theorem along the way):

\begin{theorem}
\label{theorem:basesconnected}
Let $R$ be a ring satisfying $(\SR_r)$.  Then for all $n \geq 1$, the complex $\Bases(R^n)$ is $\left\lfloor\frac{n-r-1}{2}\right\rfloor$-connected.
\end{theorem}

\begin{remark}
Here is a guide for connecting our notation and indexing conventions to those of \cite[Lemma 5.10]{RandalWilliamsWahl}.
Let $R$ be a ring satisfying $(\SR_r)$.  The paper \cite{RandalWilliamsWahl} defines
\[s = \sr(R) = \min\Set{$k$}{$R$ satisfies $(\SR_{k+1})$} = r-1.\]
Our complex $\Bases(R^n)$ is the complex $S_{n-s}(R^s,R)$ of \cite{RandalWilliamsWahl}.  The
proof of \cite[Lemma 5.10]{RandalWilliamsWahl} shows that $S_{n-s}(R^s,R)$ is 
\[\left\lfloor\frac{(n-s)-2}{2}\right\rfloor = \left\lfloor\frac{n-(r-1)-2}{2}\right\rfloor = \left\lfloor\frac{n-r-1}{2}\right\rfloor\]
connected.  The complex $W_{n-s}(R^s,R)$ in the statement of \cite[Lemma 5.10]{RandalWilliamsWahl} is the
large ordering of $S_{n-s}(R^s,R)$.
\end{remark}

\begin{proof}[Proof of Theorem \ref{theorem:basescm}]
Theorem \ref{theorem:basesconnected} says that $\Bases(R^n)$ is 
\[\left\lfloor\frac{n-r+1}{2}\right\rfloor-1=\left\lfloor\frac{n-r-1}{2}\right\rfloor\]
 connected, so
what we must prove is that for all $k$-simplices $\sigma$ of $\Bases(R^n)$, the complex
$\Link_{\Bases(R^n)}(\sigma)$ is $\left(\left\lfloor\frac{n-r-1}{2}\right\rfloor-k-1\right)$-connected.  This is a non-tautological condition
precisely when
\begin{equation}
\label{eqn:meaningfulk}
\frac{n-r-1}{2}-k-1 \geq -1, \quad \text{i.e., when } k \leq \frac{n-r-1}{2}.
\end{equation}
The connectivity condition we want will follow from Theorem \ref{theorem:basesconnected} if
$\Link_{\Bases(R^n)}(\sigma) \cong \Bases(R^{n-k-1})$, which by Lemma \ref{lemma:splitbaseslink} holds
if
\begin{equation}
\label{eqn:goodk}
k \leq n-r.
\end{equation}
So we must prove that \eqref{eqn:meaningfulk} implies \eqref{eqn:goodk}, i.e., that
\[\frac{n-r-1}{2} \leq n-r.\]
This holds precisely when $n-r \geq -1$, which is exactly our assumption $n \geq r-1$.
\end{proof}

\section{The large ordering of the complex of split partial bases}
\label{section:glncoefficient}

Let $R$ be a ring.  Recall that we introduced $\VIC(R)$-modules in \S \ref{section:vicmodules}.
We now discuss the large ordering of $\Bases(R^n)$ and show how to use a $\VIC(R)$-module
$M$ to define a $\GL_n(R)$-equivariant coefficient system on it (or, rather, on an appropriate subcomplex).

\subsection{Large ordering and shifted large ordering}
For $n \geq 0$, define $\OBases(R^n)$ to be the large ordering of $\Bases(R^n)$.  As we discussed
in Lemma \ref{lemma:identifysimplex}, the $k$-simplices of $\Bases(R^n)$ can be identified with
tuples $(\{x_0,\ldots,x_k\};C)$, where $x_0,\ldots,x_k \in R^n$ are linearly independent elements and $C \subset R^n$
is a submodule such that
\[R^n = C \oplus \bigoplus_{i=0}^k x_i \cdot R.\]
The $k$-simplices of $\OBases(R^n)$ are obtained by imposing an ordering on the $x_i$, so can
be identified with ordered tuples $(x_0,\ldots,x_k;C)$ where the $x_i$ and $C$ are as above.  It
will be technically annoying that sometimes $C$ is not a free module.  To fix this, assume
that $R$ satisfies $(\SR_r)$, and define $\OBases(R^{n,r})$ to be the semisimplicial subset
of $\OBases(R^{n+r})$ consisting of simplices of dimension at most $n$.  The group
$\GL_{n+r}(R)$ acts on $\OBases(R^{n,r})$, and the following lemma shows that 
$\OBases(R^{n,r})$ avoids some of the annoying features of $\OBases(R^n)$.  Recall
that we defined groups $\GL_n^{\cK}(R)$ in \S \ref{section:ktheory}.

\begin{lemma}
\label{lemma:obasesgood}
Let $R$ be a ring satisfying $(\SR_r)$ and let $n \geq 0$.  The following hold:
\begin{itemize}
\item[(i)] For each $k$-simplex $(x_0,\ldots,x_k;C)$ of $\OBases(R^{n,r})$, we have
$C \cong R^{n+r-k-1}$.
\item[(ii)] For all $k$ and all subgroups $\cK \subset \KK_1(R)$, the group $\GL_{n+r}^{\cK}(R)$ acts transitively on the $k$-simplices
of $\OBases(R^{n,r})$.
\item[(iii)] For all two-sided ideals $\alpha$ of $R$, we have 
\[\OBases(R^{n,r}) / \EL_{n+r}(R,\alpha) = \OBases((R/\alpha)^{n,r}).\]
\item[(iv)] The semisimplicial set $\OBases(R^{n,r})$ is the large ordering of a weakly
Cohen--Macaulay complex of dimension $\left\lfloor\frac{n+1}{2}\right\rfloor$.
\end{itemize}
\end{lemma}
\begin{proof}
Conclusions (i) and (ii) follow from Lemmas \ref{lemma:stabledecomp} and \ref{lemma:glbasesaction} along
with the fact that we have restricted the simplices of $\OBases(R^{n,r})$ to have dimension at most $n$.
Here for (ii) we are using the fact that $\EL_{n+r}(R) \subset \GL_{n+r}^{\cK}(R)$.

For (iii), the projection
\[\pi\colon \OBases(R^{n,r}) \longrightarrow \OBases((R/\alpha)^{n,r})\]
is $\GL_{n+r}(R)$-equivariant, where the group $\GL_{n+r}(R)$ acts on $\OBases((R/\alpha)^{n,r})$ via
the projection $\GL_{n+r}(R) \rightarrow \GL_{n+r}(R/\alpha)$.  Lemma \ref{lemma:stablequotient} says
that $R/\alpha$ satisfies $(\SR_r)$, so by Lemma \ref{lemma:glbasesaction} the group $\EL_{n+r}(R/\alpha)$ acts
transitively on the $k$-simplices of $\OBases((R/\alpha)^{n,r})$ for all $k$.  Since
$\EL_{n+r}(R)$ maps surjectively onto $\EL_{n+r}(R/\alpha)$ and $\pi$ is $\EL_{n+r}(R)$-equivariant,
it follows that $\pi$ is surjective.  Since $\GL_{n+1}(R,\alpha)$ is the kernel of
the map $\GL_{n+r}(R) \rightarrow \GL_{n+r}(R/\alpha)$, it acts trivially on
$\OBases((R/\alpha)^{n,r})$.  It follows that its subgroup $\EL_{n+r}(R,\alpha)$ also
acts trivially on $\OBases((R/\alpha)^{n,r})$, so 
the map $\pi$ descends to a map
\[\opi \colon \OBases(R^{n,r}) / \EL_{n+r}(R,\alpha) \rightarrow \OBases((R/\alpha)^{n,r}).\]
The map $\opi$ is surjective since $\pi$ is.  Lemma \ref{lemma:glbasesaction} implies
that two simplices of $\OBases(R^{n,r})$ that map to the
same simplex of $\OBases((R/\alpha)^{n,r})$ differ by an element of $\EL_{n+r}(R,\alpha)$, and
thus are identified with the same simplex of $\OBases(R^{n,r}) / \EL_{n+r}(R,\alpha)$.  It follows
that $\opi$ is injective, and hence an isomorphism, as desired.

For (iv), Theorem \ref{theorem:basescm} says that $\Bases(R^{n+r})$ is weakly Cohen--Macaulay
of dimension $\left\lfloor\frac{n+1}{2}\right\rfloor$.  Since $n \geq 0$, we have $n \geq \left\lfloor\frac{n+1}{2}\right\rfloor$, 
so the $n$-skeleton
of $\Bases(R^{n+r})$ is also weakly Cohen--Macaulay of dimension $\left\lfloor\frac{n+1}{2}\right\rfloor$, as desired.
\end{proof}

\subsection{Coefficient system}
\label{section:glcoef}
Let $R$ be a ring satisfying $(\SR_r)$, let $\bbk$ be a commutative ring, and let $M$ be a
$\VIC(R)$-module over $\bbk$.  For each $n \geq 0$, let $\cG_{M,n,r}$ be the augmented coefficient system
on $\OBases(R^{n,r})$ defined via the formula
\[\cG_{M,n,r}(x_0,\ldots,x_k;C) = M(C) \quad \text{for a simplex $(x_0,\ldots,x_k;C)$ of $\OBases(R^{n,r})$},\]
where our convention is that the $-1$-simplex of $\OBases(R^{n,r})$ is $(;R^{n+r})$.  For this
to make sense, we need $C$ to be a free $R$-module, which is ensured by Lemma \ref{lemma:obasesgood}.(i).
To see how $\cG_{M,n,r}$ behaves under face maps, note that for a simplex $\sigma = (x_0,\ldots,x_k;C)$ of
$\OBases(R^{n,r})$, the faces of $\sigma$ are of the form
\[\sigma' = (x_{i_0},\ldots,x_{i_{\ell}};C') \quad \text{with} \quad  C' = C \oplus \bigoplus_{j \notin \{i_0,\ldots,i_{\ell}\}} x_j \cdot R\]
for increasing sequences $0 \leq i_0 < \cdots < i_{\ell} \leq k$.  Letting $\iota\colon C \rightarrow C'$ be the natural
inclusion and letting
\[D = \bigoplus_{j \notin \{i_0,\ldots,i_{\ell}\}} x_j \cdot R,\]
the induced map $\cG_{M,n,r}(\sigma) \rightarrow \cG_{M,n,r}(\sigma')$ is the one induced by the $\VIC(R)$-structure
on $M$ as follows:
\[\cG_{M,n,r}(\sigma) = M(C) \stackrel{(\iota,D)_{\ast}}{\longrightarrow} M(C') = \cG_{M,n,r}(\sigma').\]
Just like in Example \ref{example:snequivariant}, the augmented coefficient system $\cG_{M,n,r}$ is
$\GL_{n+r}(R)$-equivariant.  It satisfies the following lemma:

\begin{lemma}[{c.f.\ Lemma \ref{lemma:fisystempoly}}]
\label{lemma:vicsystempoly}
Let $R$ be a ring satisfying $(\SR_r)$, let $\bbk$ be a commutative ring, and 
let $M$ be a $\VIC(R)$-module over $\bbk$ that is polynomial of degree $d \geq -1$ starting at $m \geq 0$ (see
Definition \ref{definition:polyvic}).
Fix\footnote{The purpose of the condition $n \geq m-r$ is simply to ensure that $n+r-m-1 \geq -1$.  Here $n+r-m-1$ is the dimension we
are claiming that $\cG_{M,n,r}$ is polynomial of degree $d$ up to, and for this to make sense we need it to be at least $-1$.} some 
$n \geq \max(m-r,0)$, and let
$\cG_{M,n,r}$ be the coefficient system on $\OBases(R^{n,r})$ discussed above.
Then $\cG_{M,n,r}$ is polynomial of 
degree $d$ up to dimension $n+r-m-1$.
\end{lemma}
\begin{proof}
The proof will be by induction on $d$. If $d=-1$, then consider a simplex $\sigma = (x_0,\ldots,x_k;C)$ with
$k \leq n+r-m-1$.  Lemma \ref{lemma:obasesgood}.(i) says that $C \cong R^{n+r-k-1}$, and since
\begin{equation}
\label{eqn:calccodim2}
n+r-k-1 \geq n+r-(n+r-m-1)-1 = m
\end{equation}
the fact that $M$ is polynomial of degree $-1$ starting at $m$ implies that $\cG_{M,n,r}(\sigma) = M(C) = 0$, as desired.

Now assume that $d \geq 0$.  There are two things to check.  For the first, let $\sigma = (x_0,\ldots,x_k;C)$ be a simplex
with $k \leq n+r-m-1$.  We must prove that the map $\cG_{M,n,r}(\sigma) \rightarrow \cG_{M,n,r}(\emptyset)$ is injective, i.e., that
the map
\[M(C) \rightarrow M(R^{n+r})\]
is injective.  The calculation \eqref{eqn:calccodim2} shows that this injectivity follows from the fact that
$M$ is polynomial of degree $d$ starting at $m$.

For the second, let $w$ be any vertex.
Let $D_{w} \cG_{M,b,r}$ be the coefficient system on the forward link $\FLink_{\OBases(R^{n,r})}(w)$ defined by the formula
\[D_{w} \cG_{M,b,r}(\sigma) = \frac{\cG_{M,n,r}(\sigma)}{\Image\left(\cG_{M,n,r}\left(w \cdot \sigma\right) \rightarrow \cG_{M,n,r}\left(\sigma\right)\right)} \quad \text{for a simplex $\sigma$ of $\FLink_{\OBases(R^{n,r})}(w)$}.\]
We must prove that $D_{w} \cG_{M,b,r}$ is polynomial of degree $d-1$ up to dimension $n+r-m-1$.
This condition is invariant under the action of $\GL_{n+r}(R)$.  Letting $\{v_1,\ldots,v_{n+r}\}$ be the standard
basis for $R^{n+r}$, Lemma \ref{lemma:obasesgood}.(ii) says that by applying an appropriate
element of $\GL_{n+r}(R)$ to $w$ we can assume without loss of generality that $w = (v_{n+r};R^{n+r-1})$.
This implies that
\[\FLink_{\OBases(R^{n,r})}(w) = \OBases(R^{n-1,r}).\]
Recall that we defined the derived $\VIC(R)$-module $DM$ in Definition \ref{definition:derivedvic}.
By construction, there is an isomorphism between
the coefficient systems $D_{w} \cG_{M,b,r}$ and $\cG_{DM,n-1,r}$ on $\OBases(R^{n-1,r})$.  Since $M$ is polynomial of
degree $d$ starting at $m$, the $\VIC(R)$-module $DM$ is
polynomial of degree $d-1$ starting at $m-1$.  By induction, $D_{w} \cG_{M,b,r}$ is polynomial of degree $d-1$ starting at
\[(n-1+r)-(m-1)-1=n+r-m-1,\]
as desired.
\end{proof}

\section{Stability for general linear groups}
\label{section:glstability}

We now are in a position to prove Theorems \ref{maintheorem:gl} and \ref{maintheorem:glprime}.  In
fact, we will prove more general results that also deal with the groups $\GL_n^{\cK}(R)$ discussed
in \S \ref{section:ktheory}.  The following generalizes Theorem \ref{maintheorem:gl}.

\begin{theorem}
\label{theorem:xl}
Let $R$ be a ring satisfying $(\SR_r)$, let $\cK \subset \KK_1(R)$ be a subgroup, let $\bbk$ be a commutative ring, and let
$M$ be a $\VIC(R)$-module over $\bbk$ that is polynomial of degree $d \geq -1$ starting at $m \geq 0$.
For each $k \geq 0$, the map
\[\HH_k(\GL_n^{\cK}(R);M(R^n)) \rightarrow \HH_k(\GL_{n+1}^{\cK}(R);M(R^{n+1}))\]
is an isomorphism for $n \geq 2k+\max(2d+r-1,m,r)+1$ and a surjection for $n = 2k+\max(2d+r-1,m,r)$.
\end{theorem}
\begin{proof}
The group $\GL_{n+r}^{\cK}(R)$ acts on $\OBases(R^{n,r})$.  Let $\cG_{M,n,r}$ be
the $\GL_{n+r}^{\cK}(R)$-equivariant augmented system of coefficients on $\OBases(R^{n,r})$ discussed
in \S \ref{section:glcoef}, so
\[\cG_{M,n,r}(x_0,\ldots,x_k;C) = M(C) \quad \text{for a simplex $(x_0,\ldots,x_k;C)$ of $\OBases(R^{n,r})$}.\]
The following claim will be used to show that with an appropriate degree shift, this all satisfies
the hypotheses of Theorem \ref{theorem:stabilitymachine}.

\begin{claim}
Assume that $n \geq \max(m-r,0)$, so we can apply Lemma \ref{lemma:vicsystempoly}.  The following then hold:
\begin{enumerate}
\item[(A)] For all $-1 \leq k \leq \min(\left\lfloor\frac{n-2d-1}{2}\right\rfloor,n+r-m-1)$, we have $\RH_k(\OBases(R^{n,r});\cG_{M,n,r}) = 0$.
\item[(B)] For all $-1 \leq k < n$, the group $\GL_{n+r-k-1}^{\cK}(R)$ is the $\GL_{n+r}^{\cK}(R)$-stabilizer of a $k$-simplex
$\sigma_k$ of $\OBases(R^{n,r})$ with $\cG_{M,n,r}(\sigma_k) = M(R^{n+r-k-1})$.
\item[(C)] For all $0 \leq k \leq n$, the group $\GL_{n+r}^{\cK}(R)$ acts transitively on the $k$-simplices of $\OBases(R^{n,r})$.
\item[(D)] For all $n \geq 2$ and all $1$-simplices $e$ of $\OBases(R^{n,r})$ of the form $e = ((x_0;C_0),(x_1;C_1))$, there
exists some $\lambda \in \GL_{n+r}^{\cK}(R)$ with $\lambda(x_0;C_0) = (x_1;C_1)$ such that $\lambda$ commutes with all
elements of $(\GL_{n+r}^{\cK}(R))_e$ and fixes all elements of $\cG_{M,n,r}(e)$.
\end{enumerate}
\end{claim}
\begin{proof}[Proof of claim]
For (A), Lemma \ref{lemma:vicsystempoly} says that $\cG_{M,n,r}$ is a polynomial coefficient system of
degree $d$ up to dimension $n+r-m-1$.  Also, by Lemma \ref{lemma:obasesgood}.(iv) the semisimplicial set
$\OBases(R^{n,r})$ is the large ordering of a simplicial complex that is weakly
Cohen--Macaulay of dimension $\left\lfloor\frac{n+1}{2}\right\rfloor$.
Letting
\[N = \min(\left\lfloor\frac{n+1}{2}\right\rfloor-d-1,n+r-m-1) = \min(\left\lfloor\frac{n-2d-1}{2}\right\rfloor,n+r-m-1),\]
the complex $\OBases(R^{n,r})$ is the large ordering of a simplicial
complex that is weakly Cohen--Macaulay of dimension $N+d+1$ and $\cG_{M,n,r}$ is a polynomial
coefficient system of degree $d$ up to dimension $N$.  Theorem \ref{theorem:vanishing} thus
implies that $\RH_k(\OBases(R^{n,r});\cG_{M,n,r}) = 0$ for $-1 \leq k \leq N$.

For (B), let $\{v_1,\ldots,v_{n+r}\}$ be the standard basis for $R^{n+r}$.  
By Lemma \ref{lemma:stablesubgroup}, the group $\GL_{n+r-k-1}^{\cK}(R)$ is the $\GL_{n+r}^{\cK}(R)$-stabilizer of the $k$-simplex
\[\sigma_k = \begin{cases}
(R^{n+r}) & \text{if $k=-1$},\\
(v_{n+r-k},v_{n+r-k+1},\ldots,v_{n+r}; R^{n+r-k-1}) & \text{if $0 \leq k < n$}
\end{cases}\]
of $\OBases(R^{n,r})$, and by definition
\[\cG_{M,n,r}(\sigma_k) = M(R^{n+r-k-1}).\]

Condition (C) is Lemma \ref{lemma:obasesgood}.(ii).  

For (D), define $\lambda\colon R^{n+r} \rightarrow R^{n+r}$ to be the $R$-module homomorphism
defined via the formulas
\[\lambda(x_0) = x_1 \quad \text{and} \quad \lambda(x_1) = -x_0 \quad \text{and} \quad \lambda|_{C_0 \cap C_1} = \text{id}.\]
This lies in $\EL_{n+r}(R) \subset \GL_{n+r}^{\cK}(R)$; indeed, the group $\SL_2(\Z)$ is generated by elementary matrices and
\[\left(\begin{matrix} 0 & -1 \\ 1 & 0 \end{matrix}\right) \in \SL_2(\Z),\]
so $\lambda$ can be written as a product of elementary matrices.  Here we are using the fact that $n+r \geq r$, so
by Theorem \ref{theorem:injectivek1} the group $\EL_{n+r}(R)$ is a normal subgroup of $\GL_{n+r}(R)$.
In particular, the fact that something can be written
as a product of elementary matrices is independent of the choice of basis.  The map $\lambda$ acts trivially on
\[\cG_{M,n,r}(e) = M(C_0 \cap C_1)\]
by basic properties of $\VIC(R)$-modules.
\end{proof}

Let $e = \max(2d-1,m-r,0) \geq 0$.  For $n \geq 0$ define
\[G_n = \GL_{n+e+r}^{\cK}(R) \quad \text{and} \quad M_n = M(R^{n+e+r})\]
and for $n \geq 1$ define
\[\bbX_n = \OBases(R^{n+e,r}) \quad \text{and} \quad \cM_n = \cG_{M,n+e,r}.\]
This makes sense since $n+e \geq n \geq 0$ for $n \geq 0$.
We then have
\[G_0 \subset G_1 \subset G_2 \subset \cdots,\]
and since $M$ is a polynomial $\VIC(R)$-module of degree $d$ starting at $m \geq 0$ the maps
$M_n \rightarrow M_{n+1}$ are injective for $n \geq 0$,\footnote{The point here is that for $n \geq 0$
we have $n+e+r \geq n+\max(2d-1,m-r,0)+r \geq n+m \geq m$.} so we have an increasing
sequence
\[M_0 \subset M_1 \subset M_2 \subset \cdots.\]
In other words, the pair $\{(G_n,M_n)\}_{n=0}^{\infty}$ is an increasing sequence
of groups and modules in the sense of \S \ref{section:twistedsetup}.

The above claim verifies the conditions of Theorem \ref{theorem:stabilitymachine} with $c=2$.
The shift by $e+r$ is needed for condition (a) of Theorem \ref{theorem:stabilitymachine}, which
requires that $\RH_k(\bbX_n;\cM_n) = 0$ for all $n \geq 1$ and $-1 \leq k \leq \left\lfloor\frac{n-2}{2}\right\rfloor$.  Conclusion (A)
of the Claim says that $\RH_k(\OBases(R^{n+e,r});\cG_{M,n+e,r}) = 0$ for 
$n+e \geq \max(m-r,0)$\footnote{Which holds for $n \geq 0$ since
then $n+e = n + \max(2d-1,m-r,0) \geq \max(m-r,0)$.}
and $-1 \leq k \leq \min(\left\lfloor\frac{(n+e)-2d-1}{2}\right\rfloor,(n+e)+r-m-1)$,
which implies the desired range of vanishing for $\RH_k(\bbX_n;\cM_n)$ since
\[\left\lfloor\frac{(n+e)-2d-1}{2}\right\rfloor \geq \left\lfloor\frac{(n+2d-1)-2d-1}{2}\right\rfloor = \left\lfloor\frac{n-2}{2}\right\rfloor \quad \text{for all $n$}\]
and
\[(n+e)+r-m-1 \geq n+(m-r)+r-m-1 = n-1 \geq \left\lfloor\frac{n-2}{2}\right\rfloor \quad \text{for all $n \geq 0$},\]
so
\[\left\lfloor\frac{n-2}{2}\right\rfloor \leq \min(\left\lfloor\frac{(n+e)-2d-1}{2}\right\rfloor, (n+e)+r-m-1) \quad \text{for all $n \geq 0$}.\]
Applying Theorem \ref{theorem:stabilitymachine}, we deduce that the map
\[\HH_k(\GL_{n+e+r-1}^{\cK}(R);M(R^{n+e+r-1})) \rightarrow \HH_k(\GL_{n+e+r}^{\cK}(R);M(R^{n+e+r}))\]
is an isomorphism for $n \geq 2k+2$ and a surjection for $n = 2k+1$, which implies
that
\[\HH_k(\GL_{n}^{\cK}(R);M(R^{n})) \rightarrow \HH_k(\GL_{n+1}^{\cK}(R);M(R^{n+1}))\]
is an isomorphism for 
\[n \geq 2k + (e+r-1) + 2 = 2k+\max(2d-1,m-r,0)+r+1 = 2k+\max(2d+r-1,m,r)+1\]
and a surjection for
\[n = 2k+\max(2d+r-1,m,r).\qedhere\]
\end{proof}

The following theorem generalizes Theorem \ref{maintheorem:glprime}:

\begin{theorem}
\label{theorem:xlprime}
Let $R$ be a ring satisfying $(\SR_r)$, let $\cK \subset \KK_1(R)$ be a subgroup, let $\bbk$ be a commutative ring, and let
$M$ be a $\VIC(R)$-module over $\bbk$ that is polynomial of degree $d \geq -1$ starting at $m \geq 0$.
For each $k \geq 0$, the map
\begin{equation}
\label{eqn:glprimetoprove}
\HH_k(\GL_n^{\cK}(R);M(R^n)) \rightarrow \HH_k(\GL_{n+1}^{\cK}(R);M(R^{n+1}))
\end{equation}
is an isomorphism for $n \geq \max(m,2k+2d+r+2)$ and a surjection for $n \geq \max(m,2k+2d+r)$.
\end{theorem}
\begin{proof}
The proof will be by double induction on $d$ and $m$.  There are three base cases:
\begin{itemize}
\item The first is where $m=0$ and $d = 0$.  Theorem \ref{theorem:xl}
says in this case that \eqref{eqn:glprimetoprove} is an isomorphism for
\[n \geq 2k+\max(2d+r-1,m,r)+1 = 2k+\max(r-1,0,r)+1 = 2k+r+1\]
and a surjection for 
\[n    = 2k+\max(2d+r-1,m,r)   = 2k+\max(r-1,0,r)   = 2k+r.\]
These bounds are even stronger than our purported bounds of
\[n \geq \max(m,2k+2d+r+2) = \max(0,2k+r+2) = 2k+r+2\]
for \eqref{eqn:glprimetoprove} to be an isomorphism and
\[n =    \max(m,2k+2d+r)   = \max(0,2k+r) = 2k+r\]
for \eqref{eqn:glprimetoprove} to be a surjection.
\item The second is where $m=0$ and $d \geq 1$.  Theorem \ref{theorem:xl}
says in this case that \eqref{eqn:glprimetoprove} is an isomorphism for
\[n \geq 2k+\max(2d+r-1,m,r)+1 = 2k+\max(2d+r-1,0,r)+1 = 2k+2d+r\]
and a surjection for
\[n    = 2k+\max(2d+r-1,m,r)   = 2k+\max(2d+r-1,0,r) = 2k+2d+r-1.\]
These bounds are even stronger than our purported bounds of
\[n \geq \max(m,2k+2d+r+2) = \max(0,2k+2d+r+2) = 2k+2d+r+2\]
for \eqref{eqn:glprimetoprove} to be an isomorphism and
\[n =    \max(m,2k+2d+r)   = \max(0,2k+2d+r) = 2k+2d+r\]
for \eqref{eqn:glprimetoprove} to be a surjection.
\item The third is where $m \geq 0$ and $d = -1$.  In this case, by the definition of a $\VIC(R)$-module
being polynomial of degree $-1$ starting at $m$ we have for $n \geq m$ that $M(R^n) = 0$ and
hence $\HH_k(\GL_{n}^{\cK}(R);M(R^n)) = 0$.  In other words, for $n \geq m$ the domain
and codomain of \eqref{eqn:glprimetoprove} are both $0$, so it is trivially an isomorphism.
\end{itemize}

Assume now that $m \geq 1$ and $d \geq 0$, and that the theorem is true for all such $(m',d')$ with $m' \leq m$ and $d' \leq d$ such that
either $m' < m$ or $d' < d$ (or both).
As in Definition \ref{definition:derivedvic}, let $\Sigma M$ be the shifted $\VIC(R)$-module and $D M$ be the derived
$\VIC(R)$-module.  For $n \geq m$, we have a short exact sequence
\begin{equation}
\label{eqn:vicshiftseq}
0 \longrightarrow M(R^n) \longrightarrow \Sigma M(R^n) \longrightarrow DM(R^n) \longrightarrow 0
\end{equation}
of $\bbk[\GL_n^{\cK}(R)]$-modules.  The $\VIC(R)$-module $\Sigma M$ is polynomial
of degree $d$ starting at $(m-1)$, and the $\VIC(R)$-module $DM$ is polynomial of degree $(d-1)$ starting at $(m-1)$.

To simplify our notation, for all $s \geq 0$ and
all $\bbk[\GL_s^{\cK}(R)]$-modules $N$, we will denote $\HH_k(\GL_s^{\cK}(R);N)$ by $\HH_k(N)$.
The long exact sequence in $\GL_n^{\cK}(R)$-homology associated to \eqref{eqn:vicshiftseq}
maps to the one in $\GL_{n+1}^{\cK}(R)$-homology, so for $n \geq m$ and all $k$ we have a commutative diagram
\begin{center}
\scalebox{0.89}{$\minCDarrowwidth10pt\begin{CD}
\HH_{k+1}(\Sigma M(R^n))            @>>> \HH_{k+1}(DM(R^n))            @>>> \HH_k(M(R^n))            @>>> \HH_k(\Sigma M(R^n))            @>>> \HH_k(DM(R^n)) \\
@VV{g_1}V                                @VV{g_2}V                          @VV{f_1}V                     @VV{f_2}V                            @VV{f_3}V \\
\HH_{k+1}(\Sigma M(R^{n+1})) @>>> \HH_{k+1}(DM(R^{n+1})) @>>> \HH_k(M(R^{n+1})) @>>> \HH_k(\Sigma M(R^{n+1})) @>>> \HH_k(DM(R^{n+1}))
\end{CD}$}
\end{center}
with exact rows.  Our inductive hypothesis says the following about the $g_i$ and $f_i$:
\begin{itemize}
\item Since $\Sigma M$ is polynomial of degree $d$ starting at $(m-1)$, the map $f_2$ is an isomorphism
for $n \geq \max(m-1,2k+2d+r+2)$ and a surjection for $n \geq \max(m-1,2k+2d+r)$.  Also, the map
$g_1$ is an isomorphism for
\[n \geq \max(m-1,2(k+1)+2d+r+2) = \max(m-1,2k+2d+r+4)\]
and a surjection for $n \geq \max(m-1,2k+2d+r+2)$.
\item Since $DM$ is polynomial of degree $(d-1)$ starting at $(m-1)$, the map $f_3$ is an
isomorphism for
\[n \geq \max(m-1,2k+2(d-1)+r+2) = \max(m-1,2k+2d+r)\]
and a surjection for $n \geq \max(m-1,2k+2d+r-2)$.  Also, the map $g_2$ is an isomorphism for
\[n \geq \max(m-1,2(k+1)+2(d-1)+r+2) = \max(m-1,2k+2d+r+2)\]
and a surjection for $n \geq \max(m-1,2k+2d+r)$.
\end{itemize}
For $n \geq \max(m,2k+2d+r+2)$, the maps $g_2$ and $f_2$ and $f_3$ are isomorphisms and the map $g_1$ is a surjection, so
by the five-lemma the map $f_1$ is an isomorphism.  For $n \geq \max(m,2k+2d+r)$, the maps $g_2$ and $f_2$ are surjections
and the map $f_3$ is an isomorphism, so by the five-lemma\footnote{Or, more precisely, one of the four-lemmas.} the
map $f_1$ is a surjection.  The claim follows.
\end{proof}

\section{Unipotence and its consequences}
\label{section:congruenceunipotence}

We now turn our attention to congruence subgroups.  Before we can prove Theorem \ref{maintheorem:congruence}, we
need some preliminary results about unipotent representations.

\subsection{Unipotent representations}
Let $\bbk$ be a field and let $V$ be a vector space over $\bbk$.  A {\em unipotent operator}
on $V$ is a linear map $f\colon V \rightarrow V$ that can be written as
$f = \text{id}_V + \phi$ where $\phi\colon V \rightarrow V$ is nilpotent, i.e., there exists
some $k \geq 1$ such that $\phi^k = 0$.  If $G$ is a group and $V$ is a $\bbk[G]$-module,
then we say that $V$ is a {\em unipotent representation} of $G$ if all elements of $G$
act on $V$ via unipotent operators.  We will mostly be interested in abelian $G$, where
this can be checked on generators:

\begin{lemma}
\label{lemma:unipotentbygen}
Let $G$ be an abelian group generated by a set $S$, let $\bbk$ be a field, and
let $V$ be a $\bbk[G]$-module.  Assume that each $s \in S$ acts on $V$
via a unipotent operator.  Then $V$ is a unipotent representation of $G$.
\end{lemma}
\begin{proof}
It is enough to prove that if $g_1,g_2 \in G$ are elements that both act
on $V$ via unipotent operators, then $g_1 g_2$ also acts via a unipotent operator.
Let $g_i$ act on $V$ via the linear map $f_i\colon V \rightarrow V$, and write
$f_i = \text{id}_V + \phi_i$ with $\phi_i$ nilpotent.  We thus have
\[f_1 f_2 = \text{id}_V + \phi_1 + \phi_2 + \phi_1 \phi_2.\]
Since the $g_i$ commute, the $\phi_i$ also commute.  This implies that
$\phi_1+\phi_2+\phi_1 \phi_2$ is nilpotent, so $f_1 f_2$ is a unipotent operator.
\end{proof}

\subsection{Unipotence and VIC(R)-modules}
The following shows how to find many unipotent operators within a $\VIC(R)$-module.  We thank Harman
for explaining its proof to us.

\begin{lemma}
\label{lemma:vicunipotent}
Let $R$ be a ring, let $\bbk$ be a field of characteristic $0$, and let $M$ be a
$\VIC(R)$-module over $\bbk$ that is polynomial of degree $d$ starting at $m$.
Assume that $M(R^n)$ is a finite-dimensional vector space over $\bbk$ for all $n$.
Then there exists some $u \geq 0$ such that for $n \geq u$, all elementary
matrices in $\GL_n(R)$ act on $M(R^n)$ via unipotent operators.
\end{lemma}
\begin{proof}
Whether or not an operator is unipotent is unchanged by field extensions, so without
loss of generality we can assume that $\bbk$ is algebraically closed.  Via the ring
homomorphism $\Z \rightarrow R$, regard $M$ as a $\VIC(\Z)$-module.  This does
not change the fact that it is polynomial of degree $d$ starting at $m$.  We
can thus appeal to a theorem of Harman \cite[Proposition 4.4]{HarmanVIC}\footnote{The
statement of Harman's theorem requires $\bbk = \C$, but the proof just uses
the fact that $\bbk$ is algebraically closed and has characteristic $0$.  Harman's theorem also requires $M$ to
be a finitely generated $\VIC(\Z)$-module.  This turns out to be a consequence of the
fact that $M$ is polynomial of degree $d$ and each $M(\Z^n)$ is finite-dimensional.  Indeed, since
each $M(\Z^n)$ is assumed to be finite-dimensional it is enough to prove that there exists some $N \geq 0$ such that $M$ is generated in
degree $N$, i.e., for all $n \geq N$, the map
\[\bigoplus_{\substack{(f,C)\colon [\Z^{n}] \rightarrow [\Z^{n+1}] \\ \text{$\VIC(\Z)$-morphism}}} M(\Z^{n}) \rightarrow M(\Z^{n+1})\]
is surjective.  This can be proved by induction on the polynomial degree $d$.  The base case $d=-1$ is trivial since
in that case $M(\Z^n)=0$ for $n \gg 0$,
so assume that $d \geq 0$ and that the result is true for smaller degrees.  The derived $\VIC(\Z)$-module $DM$ is
polynomial of degree $d-1$.  By induction, we can thus 
choose $N \geq 0$ such that $DM$ is generated in degree $N$, so for $n \geq N$ the map
\[\bigoplus_{\substack{(f,C)\colon [\Z^{n}] \rightarrow [\Z^{n+1}] \\ \text{$\VIC(\Z)$-morphism}}} DM(\Z^{n}) \rightarrow DM(\Z^{n+1})\]
is surjective.  Since
\[DM(\Z^m) = \frac{M(\Z^m \oplus \Z)}{\Image(M(\Z^m) \rightarrow M(\Z^{m} \oplus \Z)} \quad \text{for all $m \geq 0$},\]
we deduce that for $n \geq N$ the map
\[M(\Z^{n+1}) \oplus \left(\bigoplus_{\substack{(f,C)\colon [\Z^{n}] \rightarrow [\Z^{n+1}] \\ \text{$\VIC(\Z)$-morphism}}} M(\Z^{n} \oplus \Z) \right) \rightarrow M(\Z^{n+1} \oplus \Z) = M(\Z^{n+2})\]
is surjective, so $M$ is generated in degree $N+1$.}
saying
that there exists some $u \geq 0$ such that for all $n \geq u$, the action of
$\SL_n(\Z)$ on $M(\Z^n) = M(R^n)$ extends to a rational representation of the
algebraic group $\SL_n$.

Increasing $u$ if necessary, we can assume that $u \geq 3$.
Consider some $n \geq u$.  For distinct $1 \leq i,j \leq n$ and $r \in R$, let 
$e_{ij}^r \in \GL_n(R)$ be the elementary
matrix obtained from the identity by putting $r$ at position $(i,j)$.
We must check that each $e_{ij}^r$ acts on $M(R^n)$ as a unipotent
operator.

Since the action of $\SL_n(\Z)$ on $M(R^n)$ extends to a
rational representation of $\SL_n$, each elementary matrix in $\SL_n(\Z)$ acts as a unipotent
operator \cite[Theorem I.4.4]{BorelAlgebraic}.  For distinct $1 \leq i,j,k \leq n$,
we have the Steinberg relation
\[e_{ij}^r = [e_{ik}^1,e_{kj}^r] = e_{ik}^1 e_{kj}^r \left(e_{ik}^1\right)^{-1} \left(e_{kj}^r\right)^{-1}.\]
Manipulating this, we get
\[\left(e_{ik}^1\right)^{-1} e_{ij}^r = e_{kj}^r \left(e_{ik}^1\right)^{-1} \left(e_{kj}^r\right)^{-1}.\]
Since $e_{ik}^1$ acts on $M(R^n)$ as a unipotent operator and the class of unipotent operators
is closed under conjugation and inversion, the right hand side of this expression acts on $M(R^n)$
as a unipotent operator.  The matrices
\[e_{ik}^1 \quad \text{and} \quad \left(e_{ik}^1\right)^{-1} e_{ij}^r\]
commute and act on $M(R^n)$ as unipotent operators, so by Lemma \ref{lemma:unipotentbygen} their
product $e_{ij}^r$ acts on $M(R^n)$ as a unipotent operator, as desired.
\end{proof}

\subsection{Stability of invariants}
If $V$ is a module over a commutative ring $\bbk$ and $f\colon V \rightarrow V$ is
a module homomorphism, then let
\[V^f = \Set{$v \in V$}{$f(v) = v$}\]
denote the submodule of invariants.  One basic property of unipotent operators on vector spaces of characteristic $0$ is as
follows:

\begin{lemma}
\label{lemma:stabilityinvariants}
Let $V$ be a finite-dimensional vector space over a field $\bbk$ of characteristic $0$ and let $f\colon V \rightarrow V$
be a unipotent operator.  Then for all $n \geq 1$ we have $V^f = V^{f^n}$.
\end{lemma}
\begin{proof}
Since $V^f \subset V^{f^n}$, it is enough to prove that $\dim(V^f) = \dim(V^{f^n})$.
Write $f = \text{id}_V + \phi$ with $\phi$ nilpotent.  Set $c_i = \binom{n}{i}$, so
\[f^n = \text{id}_V + \phi' \quad \text{with} \quad \phi' = \sum_{i=1}^n c_i \phi^i.\]
We have
\[V^f = \ker(\phi) \quad \text{and} \quad V^{f^n} = \ker(\phi'),\]
so our goal is to prove that $\ker(\phi)$ and $\ker(\phi')$ have the same dimension.
Since $\bbk$ has characteristic $0$, we have $c_1 = n \neq 0$.
This allows us to write 
\[\phi' = c_1 \phi \circ \left(\text{id}_V + \sum_{i=1}^{n-1} \frac{c_{i+1}}{c_1} \phi^i\right) = c_1 \phi \circ\left(\text{id}_V + \phi''\right) \quad \text{with} \quad \phi'' = \sum_{i=1}^{n-1} \frac{c_{i+1}}{c_1} \phi^i.\]
The linear map $\phi''$ is nilpotent, so $\text{id}_V + \phi''$ is invertible.  It follows that
\[\dim \ker(\phi') = \dim \ker(c_1 \phi \circ (\text{id}_V + \phi'')) = \dim \ker(c_1 \phi) = \dim \ker(\phi),\]
as desired.
\end{proof}

\subsection{Stability of homology}
Lemma \ref{lemma:stabilityinvariants} is the key input to the following result.

\begin{lemma}
\label{lemma:stabilityhomology}
Let $G$ be an abelian group and let $V$ be a finite-dimensional unipotent
representation of $G$ over a field $\bbk$ of characteristic $0$.  Then for
all finite-index subgroups $G' < G$, the inclusion map $G' \hookrightarrow G$
induces an isomorphism $\HH_k(G';V) \cong \HH_k(G;V)$ for all $k \geq 0$.
\end{lemma}
\begin{proof}
We divide the proof into several steps.

\begin{stepsc}
This holds if $G$ is an infinite cyclic group.
\end{stepsc}

We thus have $G = \Z$ and $G' = n \Z$ for some $n \geq 1$.  Recall that the
zeroth homology group is the coinvariants, which is isomorphic to the 
invariants of the dual.  The dual representation $V^{\ast}$ is also a unipotent
representation of $G$, so by Lemma \ref{lemma:stabilityinvariants} we
have
\[\HH_0(G;V) = \left(V^{\ast}\right)^G = \left(V^{\ast}\right)^{G'} = \HH_0(G';V).\]
For the first homology group, since the circle is a classifying space for $G = \Z$,
we can apply Poincar\'{e} duality and use Lemma \ref{lemma:stabilityinvariants} to see that
\[\HH_1(G;V) \cong \HH^0(G;V) = V^G = V^{G'} = \HH^0(G';V) \cong \HH_1(G';V).\]
Finally, for $k \geq 2$, we have $\HH_k(G;V) = \HH_k(G';V) = 0$.

\begin{stepsc}
This holds if $G$ is a finite abelian group.
\end{stepsc}

In this case, we claim that $V$ is a trivial representation of $G$.  Indeed, since
$G$ is finite it is enough to prove that every nontrivial unipotent operator
$f\colon V \rightarrow V$ has infinite order.  Write $f = \text{id}_V + \phi$
with $\phi \neq 0$ a nilpotent operator.  For all $n \geq 1$, we then have
\[f^n = \text{id}_V + \sum_{i=1}^n \binom{n}{i} \phi^i.\]
Since $\bbk$ has characteristic $0$, we have $\binom{n}{1} = n \neq 0$, so
the coefficient of $\phi$ in this expression is nonzero.  Letting $r$ be the order
of $\phi$, the operators $\{\text{id}_V, \phi, \phi^2,\ldots \phi^{r-1}\}$
are linearly independent in the vector space of linear operators, so 
$f^n \neq \text{id}_V$, as claimed.

From this, we see that
\[\HH_k(G';V) = \HH_k(G;V) = \begin{cases} V & \text{if $k=0$}, \\ 0 & \text{if $k \geq 1$}.\end{cases}\]

\begin{stepsc}
This holds if $G$ is a finitely generated abelian group.
\end{stepsc}

We can find a chain of subgroups
\[G = G_1 \supset G_2 \supset \cdots \supset G_n = G'\]
where for each $1 \leq i < n$ the group $G_{i+1}$ is an index-$p_i$ subgroup of $G_i$ for some prime $p_i$.
From this, we see that we can assume without loss of generality that $G'$ is an index-$p$ subgroup of $G$
for some prime $p$.

Since $G'$ is an index-$p$ subgroup of the finitely generated abelian group $G$, we can
write $G = C \oplus G''$ and $G' = C' \oplus G''$ with $C$ a cyclic subgroup of $G$ and
$C'$ an index-$p$ subgroup of $C$.  We thus have a commutative diagram of short exact sequences
\[\begin{CD}
1 @>>> C'   @>>> G'   @>>> G''     @>>> 1 \\
@.     @VVV      @VVV      @VV{=}V      @.\\
1 @>>> C    @>>> G    @>>> G''     @>>> 1.\end{CD}\]
This induces a map between the Hochschild--Serre spectral sequences computing $\HH_{\bullet}(G';V)$ and
$\HH_{\bullet}(G;V)$.  On the $E^2$-page, this morphism takes the form
\begin{equation}
\label{eqn:abelianspectralsequence}
\HH_p(G'';\HH_q(C';V)) \rightarrow \HH_p(G'';\HH_q(C;V)).
\end{equation}
The previous two steps imply that the inclusion $C' \hookrightarrow C$ induces isomorphisms
$\HH_q(C';V) \cong \HH_q(C;V)$ for all $q$, so the map \eqref{eqn:abelianspectralsequence}
is an isomorphism for all $p$ and $q$.  The spectral sequence comparison theorem now implies that
$\HH_k(G';V) \cong \HH_k(G;V)$ for all $k$, as desired.

\begin{stepsc}
This holds if $G$ is an arbitrary abelian group.
\end{stepsc}

Let $\fF$ be the set of finitely generated subgroups of $G$.  We thus have
\[G = \lim_{\substack{\longrightarrow \\ H \in \fF}} H \quad \text{and} \quad G' = \lim_{\substack{\longrightarrow \\ H \in \fF}} H \cap G'.\]
Since homology commutes with direct limits, this reduces us to the previous case.
\end{proof}

\section{Stability for congruence subgroups}
\label{section:congruencestability}

We now turn to the proof of Theorem \ref{maintheorem:congruence}.  We will
actually prove three increasingly stronger results, with the third a generalization
of Theorem \ref{maintheorem:congruence}.

\subsection{Elementary matrices and unipotence}
The first is as follows, which we regard as the heart of the whole proof.
The differences between this result and Theorem \ref{maintheorem:congruence}
are as follows:
\begin{itemize}
\item It assumes that elementary matrices act unipotently on $M(R^n)$.
\item It concerns the elementary congruence subgroup $\EL_n(R,\alpha)$
rather than $\GL_n(R,\alpha)$.
\item Finally, it has a better range of stability.
\end{itemize}

\begin{theorem}
\label{theorem:congruenceweak}
Let $R$ be a ring satisfying $(\SR_r)$, let $\bbk$ be a field of characteristic $0$, and let
$M$ be a $\VIC(R)$-module over $\bbk$ that is polynomial of degree $d \geq -1$ starting at $m \geq 0$.
For each $n \geq 0$, assume that $M(R^n)$ is a finite-dimensional vector space over $\bbk$
and that each elementary matrix in $\GL_n(R)$ acts unipotently on $M(R^n)$.
Then for all two sided ideals $\alpha$ of $R$ such that $|R/\alpha| < \infty$, the map
\[\HH_k(\EL_n(R,\alpha);M(R^n)) \rightarrow \HH_k(\EL_n(R);M(R^n))\]
is an isomorphism for $n \geq 2k+\max(2d+r-1,m,r)+1$.
\end{theorem}
\begin{proof}
Letting $\cK = 0$, Example \ref{example:elementaryk} says that for $n \geq r$ we have
$\GL_n^{\cK}(R) = \EL_n(R)$.  In light of this, Theorem \ref{theorem:xl} implies that the stabilization map
\[\HH_k(\EL_n(R);M(R^n)) \rightarrow \HH_k(\EL_{n+1}(R);M(R^{n+1}))\]
is an isomorphism for $n \geq 2k+\max(2d+r-1,m,r)+1$.
We proved this by
verifying the conditions of our stability machine (Theorem \ref{theorem:stabilitymachine}) with
the parameter $c=2$ (corresponding to the multiple $2$ in front of $k$ in $n \geq 2k+\max(2d+r-1,m,r)+1$).
Letting 
\[e = \max(2d-1,m-r,0) \geq 0,\]
the inputs to Theorem \ref{theorem:stabilitymachine} were
\[G_n = \EL_{n+e+r}(R) \quad \text{and} \quad M_n = M(R^{n+e+r}) \quad \text{for $n \geq 0$}\]
and
\[\bbX_n = \OBases(R^{n+e,r}) \quad \text{and} \quad \cM_n = \cG_{M,n+e,r} \quad \text{for $n \geq 1$}.\]
For $n \geq r$, Lemma \ref{lemma:finiteindex} implies that $\EL_n(R,\alpha)$ is a finite-index
normal subgroup of $\EL_n(R)$.  To prove that the map
\[\HH_k(\EL_n(R,\alpha);M(R^n)) \rightarrow \HH_k(\EL_n(R);M(R^n))\]
is an isomorphism for $n \geq 2k+\max(2d+r-1,m,r)+1$, it is enough to verify the additional
hypotheses of Theorem \ref{theorem:stabilitymachinefi} for
\[G'_n = \EL_{n+e+r}(R,\alpha).\]
These additional hypotheses are numbered (e)-(h), and we verify each of them in turn.

Condition (e) says that each $M_n$ is a vector space over a field $\bbk$ of characteristic $0$,
which is one of our hypotheses.

Condition (f) says that for the $k$-simplex $\sigma_k$ of $\bbX_n$ from condition (b) of Theorem
\ref{theorem:stabilitymachine} whose $G_n$-stabilizer is $G_{n-k-1}$, the $G'_n$-stabilizer of 
$\sigma_k$ is $G'_{n-k-1}$.  Looking back at our proof of Theorem \ref{theorem:xl}, the
simplex $\sigma_k$ is as follows.  Let $\{v_1,\ldots,v_{n+e+r}\}$ be the standard basis for $R^{n+e+r}$.
We then have
\[\sigma_k = \begin{cases}
(R^{n+e+r}) & \text{if $k=-1$},\\
(v_{n+e+r-k},v_{n+e+r-k+1},\ldots,v_{n+e+r}; R^{n+e+r-k-1}) & \text{if $0 \leq k < n$}.
\end{cases}\]
Thus the $\GL_{n+e+r}(R)$-stabilizer of $\sigma_k$ is $\GL_{n+e+r-k-1}(R)$, and
by Lemma \ref{lemma:stablesubgroup} the $G_n = \EL_{n+e+r}(R)$ stabilizer of
$\sigma_k$ is $G_{n-k-1} = \EL_{n+e+r-k-1}(R)$.  Another application of
Lemma \ref{lemma:stablesubgroup} says that the $G'_n = \EL_{n+e+r}(R,\alpha)$
stabilizer of $\sigma_k$ is $G'_{n-k-1} = \EL_{n+e+r-k-1}(R,\alpha)$, as desired.

Condition (g) says that the quotient $\bbX_n / G'_n$ is $\left\lfloor\frac{n-2}{2}\right\rfloor$-connected.
Conclusion (iii) of Lemma \ref{lemma:obasesgood} says that
\[\bbX_n / G'_n = \OBases(R^{n+e,r}) / \EL_{n+e+r}(R,\alpha) \cong \OBases((R/\alpha)^{n+e,r}).\]
Lemma \ref{lemma:stablequotient} says that $R/\alpha$ satisfies $(\SR_r)$, so we
can apply Conclusion (iv) of Lemma \ref{lemma:obasesgood} to see that
$\OBases((R/\alpha)^{n+e,r})$ is the large ordering of a weakly Cohen--Macaulay complex
of dimension $\left\lfloor\frac{n+e+1}{2}\right\rfloor$.  By Theorem \ref{theorem:largeordering}, this implies that
$\OBases((R/\alpha)^{n+e,r})$ is
\[\left\lfloor\frac{n+e+1}{2}\right\rfloor - 1 = \left\lfloor\frac{n+e-1}{2}\right\rfloor \geq \left\lfloor\frac{n-2}{2}\right\rfloor\]
connected.  Here we use the fact that $e = \max(2d-1,m-r,0)$ is nonnegative.

Finally, the key condition (h) says that for $k \geq 0$ and $n \geq 2k+2$, the action
of $G_n$ on $\HH_k(G'_n;M_n)$
induced by the conjugation action of $G_n$ on $G'_n$ fixes pointwise the image of the stabilization
map
\[\HH_k(G'_{n-1};M_{n-1}) \rightarrow \HH_k(G'_n;M_n).\]
This is the content of the following claim, which is the heart of our proof.  Note
that this claim is even stronger since it holds for all $G_n = \EL_{n+e+r}(R)$, not
just those where $n \geq 2k+2$.

\begin{claim}
For all $n \geq 1$, the group $\EL_n(R)$ acts trivially on the image of the stabilization map
\begin{equation}
\label{eqn:elstab}
\HH_k(\EL_{n-1}(R,\alpha);M(R^{n-1})) \rightarrow \HH_k(\EL_n(R,\alpha);M(R^n)).
\end{equation}
\end{claim}

The proof of this claim generalizes a beautiful argument of Charney from
\cite{CharneyCongruence}.  For distinct $1 \leq i,j \leq n$ and $r \in R$, let $e_{ij}^r \in \EL_n(R)$
denote the elementary matrix obtained from the identity by putting $r$ at position
$(i,j)$.  Define $A$ to be the subgroup of $\EL_n(R)$ generated by
$\Set{$e_{nj}^r$}{$1 \leq j \leq n-1$, $r \in R$}$ and let $B$ be the subgroup
generated by $\Set{$e_{in}^r$}{$1 \leq i \leq n-1$, $r \in R$}$.  Using the 
Steinberg relation
\[e_{ij}^r = [e_{ik}^r,e_{kj}^1] \quad \text{for distinct $1 \leq i,j,k \leq n$ and $r \in R$},\]
we see that $\EL_n(R)$ is generated by $A \cup B$.  It is thus enough to prove
that $A$ and $B$ act trivially on the image of \eqref{eqn:elstab}.  The arguments
for $A$ and $B$ are similar, so we will give the details for $A$ and leave $B$
to the reader.

The group $A$ is the abelian subgroup of $\GL_n(R)$ consisting of matrices that differ
from the identity only in their $n^{\text{th}}$ row.  In particular, $A$ is isomorphic to the
abelian group $R^{n-1}$.  Moreover, letting $\Gamma$ be the subgroup of $\EL_n(R)$ generated
by $A$ and $\EL_{n-1}(R)$, the group $\Gamma$ is a sort of ``affine group'' and in particular
\begin{equation}
\label{eqn:affine}
\Gamma = A \rtimes \EL_{n-1}(R).
\end{equation}
We would like to imitate this for $\EL_n(R,\alpha)$.

Define $A_{\alpha}$ to be the subgroup of $\EL_n(R,\alpha)$
generated by $\Set{$e_{nj}^r$}{$1 \leq j \leq n-1$, $r \in \alpha$}$.  As
an abelian group, we have $A_{\alpha} \cong \alpha^{n-1}$.  Define $\Gamma_{\alpha}$
to be the subgroup of $\EL_n(R,\alpha)$ generated by $A_{\alpha}$ and $\EL_{n-1}(R,\alpha)$,
so just like \eqref{eqn:affine} we have
\begin{equation}
\label{eqn:affine2}
\Gamma_{\alpha} = A_{\alpha} \rtimes \EL_{n-1}(R,\alpha).
\end{equation}
The stabilization map \eqref{eqn:elstab} factors through the map
\[\HH_k(\Gamma_{\alpha};M(R^n)) \rightarrow \HH_k(\EL_n(R,\alpha);M(R^n))\]
induced by the inclusion $\Gamma_{\alpha} \hookrightarrow \EL_n(R,\alpha)$.

The conjugation action of $A$ on $\EL_n(R,\alpha)$ takes $\Gamma_{\alpha}$ to itself.
It is thus enough to prove that the conjugation action of $A$ on
$\HH_k(\Gamma_{\alpha};M(R^n))$ is trivial.  Define $\Gamma'_{\alpha}$ to be the subgroup
of $\EL_n(R)$ generated by $A$ and $\EL_{n-1}(R,\alpha)$.  Since inner
automorphisms act trivially on homology (even with twisted coefficients; see 
\cite[Proposition III.8.1]{BrownCohomology}), the conjugation action of
$A$ on $\HH_k(\Gamma'_{\alpha};M(R^n))$ is trivial.  It is thus enough to prove
that the inclusion $\Gamma_{\alpha} \hookrightarrow \Gamma'_{\alpha}$ induces an
isomorphism
$\HH_k(\Gamma_{\alpha};M(R^n)) \cong \HH_k(\Gamma'_{\alpha};M(R^n))$.

For this, observe that we have a commutative diagram
\[\begin{CD}
1  @>>> A_{\alpha} @>>> \Gamma_{\alpha}  @>>> \EL_{n-1}(R,\alpha) @>>> 1 \\
@.      @VVV            @VVV             @VV{=}V                  @.\\
1  @>>> A          @>>> \Gamma'_{\alpha} @>>> \EL_{n-1}(R,\alpha) @>>> 1
\end{CD}\]
with exact rows.  The first row is the split exact sequence corresponding
to \eqref{eqn:affine2}, and the second row is the split exact sequence
corresponding to the similar semidirect product decomposition of $\Gamma'_{\alpha}$.
This induces a morphism between the Hochschild--Serre spectral sequences
computing the homology of $\Gamma_{\alpha}$ and $\Gamma'_{\alpha}$ with coefficients
in $M(R^n)$.  On the $E^2$-page, this map of spectral sequences takes the form
\begin{equation}
\label{eqn:affiness}
\HH_p(\EL_{n-1}(R,\alpha);\HH_q(A_{\alpha};M(R^n))) \rightarrow \HH_p(\EL_{n-1}(R,\alpha);\HH_q(A;M(R^n))).
\end{equation}
The abelian group $A_{\alpha}$ is a finite-index subgroup of the abelian group $A$, and
since we assumed that elementary matrices act unipotently on $M(R^n)$ we can appeal
to Lemma \ref{lemma:unipotentbygen} to see that the action of $A$ on $M(R^n)$ is a unipotent
action.  Lemma \ref{lemma:stabilityhomology} therefore implies that the inclusion
$A_{\alpha} \hookrightarrow A$ induces isomorphisms
\[\HH_q(A_{\alpha};M(R^n)) \cong \HH_q(A;M(R^n)) \quad \text{for all $q$}.\]
We deduce that \eqref{eqn:affiness} is an isomorphism for all $p$ and $q$.  The
spectral sequence comparison theorem therefore implies that
the inclusion $\Gamma_{\alpha} \hookrightarrow \Gamma'_{\alpha}$ induces an
isomorphism
\[\HH_k(\Gamma_{\alpha};M(R^n)) \cong \HH_k(\Gamma'_{\alpha};M(R^n)),\]
as desired.
\end{proof}

\subsection{Non-elementary subgroups}
We now generalize Theorem \ref{theorem:congruenceweak} by extending it to subgroups other than the elementary ones.
This causes us to have a worse range of stability.

\begin{theorem}
\label{theorem:xlcongruenceweak}
Let $R$ be a ring satisfying $(\SR_r)$, let $\cK \subset \KK_1(R)$ be a subgroup, let $\bbk$ be a field of characteristic $0$, and let
$M$ be a $\VIC(R)$-module over $\bbk$ that is polynomial of degree $d \geq -1$ starting at $m \geq 0$.
Assume furthermore that $M(R^n)$ is a finite-dimensional vector space over $\bbk$ for all $n \geq 0$
and that each elementary matrix in $\GL_n(R)$ acts unipotently on $M(R^n)$.
Then for all two sided ideals $\alpha$ of $R$ such that $|R/\alpha| < \infty$ and all subgroups 
$\cK' \subset \KK_1(R,\alpha)$ that map to finite-index subgroups of $\cK$
under the map $\KK_1(R,\alpha) \rightarrow \KK_1(R)$, the map
\[\HH_k(\GL_n^{\cK'}(R,\alpha);M(R^n)) \rightarrow \HH_k(\GL_n^{\cK}(R);M(R^n))\]
is an isomorphism for $n \geq 2k+\max(2d+r-1,m,r)+r-1$.
\end{theorem}
\begin{proof}
Consider some $n \geq 2k+\max(2d+r-1,m,r)+r-1$.  Using Theorem \ref{theorem:injectivek1}, we have a commutative diagram\footnote{The first two columns
are injections, but the third is not.  If we wanted all the columns to be injective we could have used Lemma \ref{lemma:bigdiagram}, but
at the cost of having to deal with the slightly more complicated group $\oEL^{\cK'}_n(R,\alpha)$.  Since our spectral sequence argument
does not need the columns to be injective, this is unnecessary.}
\[\begin{CD}
1  @>>> \EL_n(R,\alpha) @>>> \GL_n^{\cK'}(R,\alpha) @>>> \cK' @>>> 1  \\
@.      @VVV                 @VVV                        @VVV      @. \\
1  @>>> \EL_n(R)        @>>> \GL_n^{\cK}(R)         @>>> \cK  @>>> 1
\end{CD}\]
with exact rows.  This induces a map between the Hochschild--Serre spectral sequences computing the homology
of $\GL_n^{\cK'}(R,\alpha)$ and $\GL_n^{\cK}(R)$ with coefficients in $M(R^n)$.  On the $E^2$-page, this
map of spectral sequences is of the form
\begin{equation}
\label{eqn:sstoanalyze}
\HH_p(\cK';\HH_q(\EL_n(R,\alpha);M(R^n))) \rightarrow \HH_p(\cK;\HH_q(\EL_n(R);M(R^n))).
\end{equation}
We analyze this via a sequence of two claims.

\begin{claim}
For all $q \leq k$, the map $\HH_q(\EL_n(R,\alpha);M(R^n)) \rightarrow \HH_q(\EL_n(R);M(R^n))$
is an isomorphism.
\end{claim}
\begin{proof}[Proof of claim]
Since $r \geq 2$ (see Example \ref{example:srremark}), we have
$2k+\max(2d+r-1,m,r)+r-1 \geq 2k+\max(2d+r-1,m,r)+1$.  The claim is now
immediate from Theorem \ref{theorem:congruenceweak}.
\end{proof}

\begin{claim}
For all $q \leq k$, the action of the group $\cK$ (resp.\ $\cK'$) on $\HH_q(\EL_n(R);M(R^n))$ (resp.\ $\HH_q(\EL_n(R,\alpha);M(R^n))$)
is trivial.
\end{claim}
\begin{proof}[Proof of claim]
The previous claim implies that it is enough to prove that all of $\KK_1(R)$ 
acts trivially on $\HH_q(\EL_n(R);M(R^n))$.  Since
\[n-(r-1) \geq (2k+\max(2d+r-1,m,r)+r-1)-(r-1) = 2k+\max(2d+r-1,m,r),\]
Theorem \ref{theorem:xl} implies that the map\footnote{Here we are applying
Theorem \ref{theorem:xl} to $\GL_{n-(r-1)}^{0}(R)$, which by Example
\ref{example:elementaryk} we can identify with $\EL_{n-(r-1)}(R)$.  This uses 
the fact that $n-(r-1) \geq r$.}
\begin{equation}
\label{eqn:destabilize}
\HH_q(\EL_{n-(r-1)}(R);M(R^{n-(r-1)})) \rightarrow \HH_q(\EL_n(R);M(R^n))
\end{equation}
is surjective.  Theorem \ref{theorem:injectivek1} together with Lemma \ref{lemma:generategl} implies
that the subgroup $\GL_{r-1}(R)$ of $\GL_n(R)$ surjects onto $\KK_1(R)$ under the map
$\GL_n(R) \rightarrow \KK_1(R)$ whose kernel is $\EL_n(R)$.  The group $\GL_{r-1}(R)$ is embedded
in $\GL_n(R)$ using the upper-left-hand matrix embedding, but we can conjugate it by any
element of $\GL_n(R)$ and it will still surject onto $\KK_1(R)$.  We can thus use the lower-right-hand
embedding, which makes $\GL_{r-1}(R)$ commute with $\EL_{n-(r-1)}(R)$.  This implies that
$\KK_1(R)$ acts trivially on the image of \eqref{eqn:destabilize}, and hence on
$\HH_q(\EL_n(R);M(R^n))$.
\end{proof}

By the second Claim, for $q \leq k$ we can rewrite our map of spectral sequences \eqref{eqn:sstoanalyze} as
\begin{equation}
\label{eqn:ssanalyzed}
\HH_p(\cK';\bbk) \otimes \HH_q(\EL_n(R,\alpha);M(R^n)) \rightarrow \HH_p(\cK;\bbk) \otimes \HH_q(\EL_n(R);M(R^n)).
\end{equation}
We assumed that the map $\cK' \rightarrow \cK$ has finite cokernel, and Lemma \ref{lemma:relativek} implies
that it has finite kernel.  Since $\bbk$ is a field of characteristic $0$, this implies that the map
$\HH_p(\cK';\bbk) \rightarrow \HH_p(\cK;\bbk)$ is an isomorphism for all $p$.  Combining this with the first Claim,
we see that \eqref{eqn:ssanalyzed} is an isomorphism for all $p$ and $q$ with $q \leq k$.  By the spectral sequence
comparison theorem, we deduce that the map
\[\HH_k(\GL_n^{\cK'}(R,\alpha);M(R^n)) \rightarrow \HH_k(\GL_n^{\cK}(R);M(R^n))\]
is an isomorphism, as desired.
\end{proof}

\subsection{The general case: removing unipotence}
We now prove the following theorem: 

\begin{theorem}
\label{theorem:xlcongruence}
Let $R$ be a ring satisfying $(\SR_r)$, let $\cK \subset \KK_1(R)$ be a subgroup, let $\bbk$ be a field of characteristic $0$, and let
$M$ be a $\VIC(R)$-module over $\bbk$ that is polynomial of degree $d \geq -1$ starting at $m \geq 0$.
Assume furthermore that $M(R^n)$ is a finite-dimensional vector space over $\bbk$ for all $n \geq 0$.
Then for all two sided ideals $\alpha$ of $R$ such that $|R/\alpha|<\infty$ and all subgroups $\cK' \subset \KK_1(R,\alpha)$ that
map to finite-index subgroups of $\cK$ under the map $\KK_1(R,\alpha) \rightarrow \KK_1(R)$, the map
\begin{equation}
\label{eqn:xlcongruence}
\HH_k(\GL_n^{\cK'}(R,\alpha);M(R^n)) \rightarrow \HH_k(\GL_n^{\cK}(R);M(R^n))
\end{equation}
is an isomorphism for $n \geq \max(m,2k+2d+2r-1)$.
\end{theorem}

\begin{remark}
This implies Theorem \ref{maintheorem:congruence}.  Indeed,
let $R$ and $\alpha$ and $\bbk$ and $M$ be as in Theorem \ref{theorem:xlcongruence}.
By Lemma \ref{lemma:relativek}, the cokernel of the map $\KK_1(R,\alpha) \rightarrow \KK_1(R)$ 
is finite.  We can thus apply Theorem \ref{theorem:xlcongruence} with $\cK = \KK_1(R)$ and 
$\cK' = \KK_1(R,\alpha)$.  With these choices, we have by definition that
\[\GL_n^{\cK'}(R,\alpha) = \GL_n(R,\alpha) \quad \text{and} \quad \GL_n^{\cK}(R) = \GL_n(R).\]
Theorem \ref{theorem:xlcongruence} thus says that the map
\[\HH_k(\GL_n(R,\alpha);M(R^n)) \rightarrow \HH_k(\GL_n(R);M(R^n))\]
is an isomorphism for $n \geq \max(m,2k+2d+2r-1)$, as claimed by Theorem \ref{maintheorem:congruence}.
\end{remark}

\begin{proof}
Lemma \ref{lemma:vicunipotent} says that there exists some $u \geq 0 $ such that for $n \geq u$,
all elementary matrices in $\GL_n(R)$ act unipotently on $M(R^n)$.  
The {\em parameters} of $M$ are the triple $(d,m,u)$.  They satisfy $d \geq -1$ and $m,u \geq 0$.
However, to avoid having to treat $m=0$ and $u=0$ separately we will allow $m$ and $u$ to be
arbitrary integers, interpreted in the obvious way (so for instance if $u=-3$ then all elementary
matrices in $\GL_n(R)$ act unipotently on $M(R^n)$ for all $n \geq 0$).

We will prove the theorem by induction on the parameters $(d,m,u)$.  There are two base cases:
\begin{itemize}
\item The first is where $d \geq 0$ and $m,u \leq 0$.  In this case, Theorem \ref{theorem:xlcongruenceweak} says
that \eqref{eqn:xlcongruence} is an isomorphism for
\[n \geq 2k+\max(2d+r-1,0,r)+(r-1).\]
Since $d \geq 0$, we have $2d+r \geq \max(2d+r-1,0,r)$, so this is even better than the bound
\[n \geq 2k+(2d+r) + (r-1) = \max(0,2k+2d+2r-1)\]
we are trying to prove.
\item The second is where $d = -1$ and $m,u \in \Z$ are arbitrary.
In this case, by the definition of a $\VIC(R)$-module
being polynomial of degree $-1$ starting at $m$ we have for $n \geq m$ that $M(R^n) = 0$ and
hence 
\[\HH_k(\GL_n^{\cK'}(R,\alpha);M(R^n)) = \HH_k(\GL_n^{\cK}(R);M(R^n)) = 0.\]
In other words, for $n \geq m$ the domain
and codomain of \eqref{eqn:xlcongruence} are both $0$, so it is trivially an isomorphism.
\end{itemize}
Now consider parameters $(d,m,u)$ that do not lie in our base cases.  They thus satisfy
the following:
\begin{itemize}
\item $d \geq 0$, and either $m \geq 1$ or $u \geq 1$ (possibly both hold: $m \geq 1$ and $u \geq 1$).
\end{itemize}
Moreover, we can assume as an inductive hypothesis that the theorem is true for all $M$ with parameters $(d',m',u')$ satisfying
the following:
\begin{itemize}
\item Either $d' \leq d-1$, or $d' = d$ and both $m' \leq m-1$ and $u' \leq u-1$.
\end{itemize}
Note that it is {\em not} enough to shrink just one of $m$ and $u$ since 
otherwise we might never reach one of our base cases.

As in Definition \ref{definition:derivedvic}, let $\Sigma M$ be the shifted $\VIC(R)$-module and $D M$ be the derived
$\VIC(R)$-module.  For $n \geq m$, we have a short exact sequence
\begin{equation}
\label{eqn:vicshiftseqcong}
0 \longrightarrow M(R^n) \longrightarrow \Sigma M(R^n) \longrightarrow DM(R^n) \longrightarrow 0
\end{equation}
of $\bbk[\GL_n(R)]$-modules.  The $\VIC(R)$-module $\Sigma M$ has parameters $(d,m-1,u-1)$, and
the $\VIC(R)$-module $DM$ has parameters $(d-1,m-1,u-1)$.\footnote{The reason $u-1$ appears in
both $\Sigma M$ and $DM$ here is as follows.  Consider $n \geq u-1$.  We have 
$\Sigma M(R^n) = M(R^{n+1})$.  Since all
elementary matrices in $\GL_{n+1}(R)$ act unipotently on $M(R^{n+1})$, all elementary matrices
in its subgroup $\GL_n(R)$ do as well.  If a linear operator $\phi$ acts unipotently on a vector space $V$
and preserves a subspace $W \subset V$, then $\phi$ also acts unipotently on $V/W$.  It follows that
all elementary matrices in $\GL_n(R)$ also act unipotently on the quotient space $DM(R^n)$ of $M(R^{n+1})$.}

To simplify our notation, for all all $\bbk[\GL_n(R)]$-modules $N$, we will denote 
\begin{itemize}
\item $\HH_k(\GL_n^{\cK}(R);N)$ by $\HH_k(N)$.
\item $\HH_k(\GL_n^{\cK'}(R,\alpha);N)$ by $\HH_k(\alpha,N)$.
\end{itemize}
The long exact sequence in $\GL_n^{\cK'}(R,\alpha)$-homology associated to \eqref{eqn:vicshiftseqcong}
maps to the one in $\GL_n^{\cK}(R)$-homology, so for $n \geq m$ and all $k$ we have a commutative diagram
\begin{center}
\scalebox{0.89}{$\minCDarrowwidth10pt\begin{CD}
\HH_{k+1}(\alpha, \Sigma M(R^n)) @>>> \HH_{k+1}(\alpha, DM(R^n)) @>>> \HH_k(\alpha, M(R^n)) @>>> \HH_k(\alpha, \Sigma M(R^n)) @>>> \HH_k(\alpha, DM(R^n)) \\
@VV{g_1}V                             @VV{g_2}V                       @VV{f_1}V                  @VV{f_2}V                         @VV{f_3}V \\
\HH_{k+1}(\Sigma M(R^n))         @>>> \HH_{k+1}(DM(R^n))         @>>> \HH_k(M(R^n))         @>>> \HH_k(\Sigma M(R^n))         @>>> \HH_k(DM(R^n))
\end{CD}$}
\end{center}
with exact rows.  When $n \geq r$, Lemma \ref{lemma:finiteindex} says that $\GL_n^{\cK'}(R,\alpha)$ is a finite-index subgroup of $\GL_n^{\cK}(R)$.
In that case, since all our coefficients are vector spaces over a field
$\bbk$ of characteristic $0$, the transfer map (see \cite[Chapter III.9]{BrownCohomology}) implies that all the $f_i$ and $g_i$ are surjections.
Our inductive hypothesis says the following about them:
\begin{itemize}
\item Since $\Sigma M$ has parameters $(d,m-1,u-1)$, the map $f_2$ is an isomorphism
for $n \geq \max(m-1,2k+2d+2r-1)$ and the map $g_1$ is an isomorphism for
\[n \geq \max(m-1,2(k+1)+2d+2r-1) = \max(m-1,2k+2d+2r+1).\]
\item Since $DM$ has parameters $(d-1,m-1,u-1)$, the map $f_3$ is an isomorphism for
\[n \geq \max(m-1,2k+2(d-1)+2r-1) = \max(m-1,2k+2d+2r-3)\]
and the map $g_2$ is an isomorphism for
\[n \geq \max(m-1,2(k+1)+2(d-1)+2r-1) = \max(m-1,2k+2d+2r-1).\]
\end{itemize}
For $n \geq \max(m,2k+2d+2r-1)$, the maps $g_2$ and $f_2$ and $f_3$ are isomorphisms and the map $g_1$ is a surjection (remember, it is always
a surjection!), so by the five-lemma the map $f_1$ is an isomorphism, as desired.
\end{proof}

\end{document}